\documentclass[aop,preprint]{imsart}
\RequirePackage[OT1]{fontenc}
\RequirePackage{amsthm,amsmath}
\usepackage[pdftitle={Intersections of SLE Paths: the double and cut point dimension of SLE},pdfauthor={Jason Miller and Hao Wu}, bookmarks=true, bookmarksopen=true, bookmarksopenlevel=3,unicode]{hyperref}

\usepackage{palatinoutopia}

\numberwithin{equation}{section}
\numberwithin{figure}{section}

\usepackage{amssymb}
\usepackage{comment}
\usepackage{graphicx}
\usepackage[font=footnotesize]{caption}
\usepackage{color}
\usepackage{subfigure}
\usepackage{enumerate}
\usepackage{comment}
\usepackage{colonequals}
\usepackage[all]{hypcap}
\usepackage[normalem]{ulem}
\usepackage{xfrac}

\startlocaldefs

\newtheorem{theorem}{Theorem}[section]
\newtheorem{lemma}[theorem]{Lemma}
\newtheorem{proposition}[theorem]{Proposition}
\newtheorem{corollary}[theorem]{Corollary}

\theoremstyle{definition}

\newtheorem{remark}[theorem]{Remark}

\newcommand{\eps}{\epsilon}
\newcommand{\p}{\mathbf{P}}

\renewcommand{\C}{\mathbf{C}}
\newcommand{\D}{\mathbf{D}}
\newcommand{\E}{\mathbf{E}}

\newcommand{\N}{\mathbf{N}}
\newcommand{\Z}{\mathbf{Z}}
\newcommand{\Q}{\mathbf{Q}}
\newcommand{\R}{\mathbf{R}}
\newcommand{\h}{\mathbf{H}}

\newcommand{\CB}{\mathcal {B}}
\newcommand{\CC}{\mathcal {C}}
\newcommand{\CD}{\mathcal {D}}

\newcommand{\CF}{\mathcal {F}}
\newcommand{\CI}{\mathcal {I}}

\newcommand{\CK}{\mathcal {K}}
\newcommand{\CL}{\mathcal {L}}

\newcommand{\CR}{\mathcal {R}}
\newcommand{\CS}{\mathcal {S}}

\newcommand{\CU}{\mathcal {U}}
\newcommand{\CV}{\mathcal {V}}

\newcommand{\CW}{\mathcal {W}}

\newcommand{\CH}{\mathcal {H}}

\newcommand{\dist}{{\rm dist}}

\newcommand{\diam}{{\rm diam}}

\newcommand{\im}{{\rm Im}}
\newcommand{\re}{{\rm Re}}

\newcommand{\confrad}{{\rm CR}}
\newcommand{\rad}{{\rm rad}}

\newcommand{\one}{{\bf 1}}

\newcommand{\ol}{\overline}
\newcommand{\ul}{\underline}
\newcommand{\wh}{\widehat}
\newcommand{\wt}{\widetilde}

\renewcommand{\arxiv}[1]{\href{http://arxiv.org/abs/#1}{#1}}

\numberwithin{equation}{section}

\newcommand{\beq}{\begin{equation}}
\newcommand{\eeq}{\end{equation}}
\newcommand{\bthm}{\begin{thm}}
\newcommand{\ethm}{\end{thm}}
\newcommand{\bcor}{\begin{cor}}
\newcommand{\ecor}{\end{cor}}
\newcommand{\bdfn}{\begin{dfn}}
\newcommand{\edfn}{\end{dfn}}
\newcommand{\blem}{\begin{lem}}
\newcommand{\elem}{\end{lem}}
\newcommand{\bppn}{\begin{ppn}}
\newcommand{\eppn}{\end{ppn}}
\newcommand{\brmk}{\begin{rmk}}
\newcommand{\ermk}{\end{rmk}}
\newcommand{\bpf}{\begin{proof}}
\newcommand{\epf}{\end{proof}}
\newcommand{\bnm}{\begin{enumerate}}
\newcommand{\enm}{\end{enumerate}}
\newcommand{\bitm}{\begin{itemize}}
\newcommand{\eitm}{\end{itemize}}
\newcommand{\bex}{\begin{ex}}
\newcommand{\eex}{\end{ex}}

\newcommand{\ka}{\kappa}

\newcommand{\ze}{\zeta} 

\newcommand{\abs}[1]{|#1|}

\newcommand{\ball}[2]{B_{#1}(#2)}
\renewcommand{\emptyset}{\varnothing}
\renewcommand\d[1]{\ d #1}

\newcommand{\pd}{\partial}
\newcommand{\set}[1]{\{#1\}}

\renewcommand{\setminus}{\backslash}

\newcommand{\SLE}{\mathrm{SLE}}

\newcommand{\GFF}{\mathrm{GFF}}
\newcommand{\giv}{\,|\,}

\newcommand{\RohdeSchramm}{MR2153402}

\newcommand{\SmirnovIsing}{MR2680496}

\newcommand{\SSGFF}{MR2486487}

\newcommand{\MillerGL}{arXiv:1010.1356}

\newcommand{\SW}{MR2188260}

\newcommand{\Pommerenke}{MR1217706}

\newcommand{\Schramm}{MR1776084}
\newcommand{\LSW}{MR2044671}

\newcommand{\Smirnov}{MR2227824}
\newcommand{\CamiaNewmanPath}{MR2322705}

\newcommand{\DubedatDuality}{MR2571956}
\newcommand{\ZhanDuality}{MR2439609}
\newcommand{\ZhanDualityII}{MR2682265}
\newcommand{\LawlerCutpointsBM}{MR1386294}
\newcommand{\LSWBrownianIntersectionI}{MR1879850}
\newcommand{\LSWBrownianIntersectionII}{MR1879851}
\newcommand{\LSWBrownianIntersectionIII}{MR1899232}
\newcommand{\DuplantierSaluer}{PhysRevLett.63.2536}
\newcommand{\WernerSLEKappaRho}{MR2060031}
\newcommand{\SheWeld}{2010arXiv1012.4797S}

\newcommand{\T}{\mathbf{T}}
\newcommand{\HH}{\mathbf{H}}
\newcommand{\LH}{\mathcal{H}}

\definecolor{red}{rgb}{1,0,0}
\definecolor{purple}{rgb}{0.7,0,0.7}
\definecolor{gray}{rgb}{0.6,0.6,0.6}
\definecolor{dgreen}{rgb}{0.0,0.4,0.0}

\newcommand{\dimH}{\dim_{\CH}}

\begin{document}
\begin{frontmatter}

\title{Intersections of SLE Paths: the~double~and~cut~point~dimension~of~SLE}
\runtitle{Intersections of SLE Paths}

  \begin{aug}

    \runauthor{Jason Miller and Hao Wu}

    \author{Jason Miller}

    \author{Hao Wu}

\affiliation{Massachusetts Institute of Technology and Universit\'e Paris-Sud}

\end{aug}

\begin{abstract}
We compute the almost-sure Hausdorff dimension of the double points of chordal $\SLE_\kappa$ for $\kappa > 4$,  confirming a prediction of Duplantier--Saleur (1989) for the contours of the FK model.  We also compute the dimension of the cut points of chordal $\SLE_\kappa$ for $\kappa > 4$ as well as analogous dimensions for the radial and whole-plane $\SLE_\kappa(\rho)$ processes for $\kappa > 0$.  We derive these facts as consequences of a more general result in which we compute the dimension of the intersection of two flow lines of the formal vector field $e^{ih/\chi}$, where $h$ is a Gaussian free field and $\chi > 0$, of different angles with each other and with the domain boundary.
\end{abstract}

\thispagestyle{empty}

\setattribute{keyword}{AMS}{AMS 2010 subject classifications:}


\date{\today}

\end{frontmatter}

\maketitle

\section{Introduction}
\subsection{Overview}

The Schramm-Loewner evolution $\SLE_\kappa$ ($\kappa > 0$) is the canonical model for a conformally invariant probability measure on non-crossing, continuous paths in a proper simply connected domain $D$ in $\C$.   $\SLE_\kappa$ was introduced by Oded Schramm \cite{\Schramm} as the candidate for the scaling limit of loop-erased random walk and for the interfaces in critical percolation.  Since its introduction, $\SLE$ has been proved to describe the limiting interfaces in many different models from statistical mechanics \cite{\LSW,\Smirnov,\CamiaNewmanPath,\SSGFF,\MillerGL,\SmirnovIsing,chelkak-smirnov,CDCHKS,HONG_KYT_ISING}.  The purpose of this article is to study self-intersections of $\SLE$ paths as well as the intersection of multiple $\SLE$ paths when coupled together using the Gaussian free field (GFF).  Our main results are Theorems~\ref{thm::double_point_dimension}--\ref{thm::boundary_dimension} which give the dimension of the self-intersection and cut points of chordal, radial, and whole-plane $\SLE_\kappa$ and $\SLE_\kappa(\rho)$ processes as well as the dimension of the intersection of such paths with the domain boundary.  Theorems~\ref{thm::double_point_dimension}--\ref{thm::whole_plane_self_intersection_dimension_cfl} are actually derived from Theorem~\ref{thm::two_flowline_dimension} which gives the dimension of the intersection of two $\SLE_\kappa(\ul{\rho})$ processes coupled together as flow lines of a GFF \cite{She_SLE_lectures, dubedat2009, MakarovSmirnov09,SchrammShe10,HagendorfBauerBernard10,IzyurovKytola10, \SheWeld,IG1,IG2,IG3,IG4} with different angles.

\subsection{Main Results}

Throughout, unless explicitly stated otherwise we shall assume that $\kappa' > 4$ and $\kappa = 16/\kappa' \in (0,4)$.  The first result that we state is the double point dimension for chordal $\SLE_{\kappa'}$.

\begin{theorem}
\label{thm::double_point_dimension}
Let $\eta$ be a chordal $\SLE_{\kappa'}$ process for $\kappa' >4$ and let $\CD$ be the set of double points of $\eta$.  Almost surely,
\begin{equation}
\label{eqn::double_point_dimension}
 \dimH(\mathcal{D})
 =\begin{cases}
   2-\frac{(12-\kappa')(4+\kappa')}{8\kappa'} \quad&\text{for}\quad\quad \kappa' \in (4,8)\\
   1+\frac{2}{\kappa'} \quad&\text{for}\quad\quad \kappa' \geq 8.
 \end{cases}
\end{equation}
In particular, when $\kappa'=6$, $\dimH(\CD)=\tfrac{3}{4}$.
\end{theorem}

Recall that chordal $\SLE_{\kappa'}$ is self-intersecting for $\kappa' > 4$ and space-filling for $\kappa' \geq 8$ \cite{\RohdeSchramm}.  The dimension in \eqref{eqn::double_point_dimension} for $\kappa' \in (4,8)$ was first predicted by Duplantier--Saleur \cite{\DuplantierSaluer} in the context of the contours of the FK model.  The almost sure Hausdorff dimension of $\SLE_\kappa$ is $1+\tfrac{\kappa}{8}$ for $\kappa \in (0,8)$ and $2$ for $\kappa \geq 8$ \cite{beffara2008} and, by $\SLE$ duality, the outer boundary of an $\SLE_{\kappa'}$ process for $\kappa' > 4$ stopped at a positive and finite time is described by a certain $\SLE_\kappa$ process \cite{\ZhanDuality,\ZhanDualityII,\DubedatDuality,IG1,IG3,IG4}.  Thus \eqref{eqn::double_point_dimension} for $\kappa' \geq 8$ states that the double point dimension is equal to the dimension of the outer boundary of the path.  We note that chordal $\SLE_{\kappa'}$ does not have triple points for $\kappa' \in (4,8)$ and the set of triple points is countable for $\kappa' \geq 8$; see Remark~\ref{rem::triple_points}.

Our second main result is the dimension of the cut-set of chordal $\SLE_{\kappa'}$:

\begin{theorem}
\label{thm::cut_set_dimension}
Let $\eta$ be a chordal $\SLE_{\kappa'}$ process for $\kappa' > 4$ and let
\[ \CK = \{\eta(t): t\in(0,\infty),\ \eta(0,t)\cap \eta(t,\infty)=\emptyset\}\]
be the cut-set of $\eta$.  Then, for $\kappa' \in(4,8)$, almost surely
\begin{equation}
\label{eqn::cut_set_dimension}
\dimH(\CK)= 3-\frac{3\kappa'}{8}.
\end{equation}
In particular, when $\kappa'=6$, $\dimH(\CK) = \tfrac{3}{4}$. For $\kappa' \ge 8$, almost surely $\CK=\emptyset$.
\end{theorem}

The dimension \eqref{eqn::cut_set_dimension} was conjectured in \cite{DB2004} by Duplantier.  Note that we recover the cut-set dimension for Brownian motion and $\SLE_6$ established in the works of Lawler and Lawler-Schramm-Werner \cite{\LawlerCutpointsBM,\LSWBrownianIntersectionI,\LSWBrownianIntersectionII,\LSWBrownianIntersectionIII}.  The dimension of the \emph{cut times} (with respect to the capacity parameterization for $\SLE$), i.e.\ the set $\{t \in (0,\infty) : \eta(0,t)\cap \eta(t,\infty)=\emptyset\}$ is $2-\tfrac{\kappa'}{4}$ for $\kappa' \in (4,8)$ and was computed by Beffara in \cite[Theorem 5]{beffara2004}.

Our next result gives the dimension of the self-intersection points of the radial and whole-plane $\SLE_\kappa(\rho)$ processes for $\kappa \in (0,4)$.  Unlike chordal $\SLE_\kappa$ and $\SLE_\kappa(\rho)$ processes, such processes can intersect themselves depending on the value of $\rho > -2$.  The maximum number of times that such a process can hit any given point for $\kappa > 0$ is given by \cite[Proposition~3.31]{IG4}:
\begin{equation}
\label{eqn::whole_plane_intersection_count}
\lceil J_{\kappa,\rho} \rceil \quad\text{where}\quad J_{\kappa,\rho} = \frac{\kappa}{2(2+\rho)}.
\end{equation}
In particular, $J_{\kappa,\rho} \uparrow +\infty$ as $\rho \downarrow -2$ and $J_{\kappa,\rho} \downarrow 1$ as $\rho \uparrow \tfrac{\kappa}{2}-2$.  Recall that $-2$ is the lower threshold for an $\SLE_\kappa(\rho)$ process to be defined.  For radial or whole-plane $\SLE_\kappa(\rho)$, the interval of $\rho$ values in which such a process is self-intersecting is given by $(-2,\tfrac{\kappa}{2}-2)$ (see, e.g., \cite[Section~2.1]{IG4}).  (For chordal $\SLE_\kappa(\rho)$, this is the interval of $\rho$ values in which such a process is boundary intersecting.)  For $\rho \geq \tfrac{\kappa}{2}-2$, such processes are almost surely simple.

\begin{theorem}
\label{thm::whole_plane_self_intersection_dimension}
Suppose that $\eta$ is a radial $\SLE_\kappa(\rho)$ process in $\D$ for $\kappa \in (0,4)$ and $\rho \in \big( -2,\tfrac{\kappa}{2}-2\big)$.  Assume that $\eta$ starts from $1$ and has a single boundary force point of weight $\rho$ located at $1^-$ (immediately to the left of $1$ on $\partial \D$).  For each $j \in \N$, let $\CI_j$ denote the set of points in (the interior of) $\D$ that $\eta$ hits exactly $j$ times.  For each $2 \leq j \leq \lceil J_{\kappa,\rho} \rceil$, where $J_{\kappa,\rho}$ is given by \eqref{eqn::whole_plane_intersection_count}, we have that
\begin{equation}
\label{eqn::whole_plane_multiple_points}
\begin{split}
 &\dimH(\CI_j) = \\
&\frac{1}{8\kappa}\bigg( 4+\kappa + 2\rho - 2j \big( 2+\rho \big) \bigg) \bigg( 4+\kappa-2\rho + 2j \big(2+\rho \big) \bigg)
\end{split}
\end{equation}
almost surely.  For $j >\lceil  J_{\kappa,\rho} \rceil$, almost surely $\CI_j = \emptyset$.  These results similarly hold if $\eta$ is a whole-plane $\SLE_\kappa(\rho)$ process.

Let $\CB_j$ be the set of points in $\partial \D$ that $\eta$ hits exactly $j$ times.  For each $1 \leq j \leq \lceil J_{\kappa,\rho} \rceil -1$, we have that
\begin{equation}
\label{eqn::radial_boundary}
\begin{split}
\dimH(\CB_j) &= \frac{1}{2\kappa}\bigg(\kappa-2 j \big(2+\rho \big) \bigg) \bigg(2+j \big(2+\rho\big) \bigg)\\
&\quad\text{almost surely on}\quad \{\CB_j \neq \emptyset\}.
\end{split}
\end{equation}
For each $j > \lceil J_{\kappa,\rho} \rceil -1$, almost surely $\CB_j = \emptyset$.
\end{theorem}

Note that $J_{\kappa,\rho}+1$ is the value of $j$ that makes the right side of \eqref{eqn::whole_plane_multiple_points} equal to zero.  Similarly, $J_{\kappa,\rho}$ is the value of $j$ that makes the right side of \eqref{eqn::radial_boundary} equal to zero.  Inserting $j=1$ into \eqref{eqn::whole_plane_multiple_points} we recover the dimension formula for the range of an $\SLE_\kappa$ process \cite{beffara2008} (though we do not give an alternative proof of this result).

We next state the corresponding result for whole-plane and radial $\SLE_{\kappa'}(\rho)$ processes with $\kappa' > 4$.  Such a process has two types of self-intersection points.  Those which arise when the path wraps around its target point and intersects itself in either its left or right boundary (which are defined by lifting the path to the universal cover of the domain minus the target point of the path) and those which occur between the left and right boundaries.  It is explained in \cite[Section~4.2]{IG4} that these two self-intersection sets are almost surely disjoint and the dimension of the latter is almost surely given by the corresponding dimension for chordal $\SLE_{\kappa'}$ (Theorem~\ref{thm::double_point_dimension}).  In fact, the set which consists of the multiple intersection points of the path where the path hits itself without wrapping around its target point and are also contained in its left and right boundaries is almost surely countable.  The following gives the dimension of the former:

\begin{theorem}
\label{thm::whole_plane_self_intersection_dimension_cfl}
Suppose that $\eta'$ is a radial $\SLE_{\kappa'}(\rho)$ process in $\D$ for $\kappa' > 4$ and $\rho \in \big( \tfrac{\kappa'}{2}-4,\tfrac{\kappa'}{2}-2)$.  Assume that $\eta'$ starts from $1$ and has a single boundary force point of weight $\rho$ located at $1^-$ (immediately to the left of $1$ on $\partial \D$).  For each $j \in \N$, let $\CI_j'$ denote the set of points that $\eta'$ hits exactly $j$ times and which are also contained in its left and right boundaries.  For each $2 \leq j \leq \lceil J_{\kappa',\rho} \rceil$ where $J_{\kappa',\rho}$ is given by \eqref{eqn::whole_plane_intersection_count}, we have that
\begin{equation}
\label{eqn::whole_plane_multiple_points_cfl}
\begin{split}
&\dimH(\CI_j') =\\
&\frac{1}{8\kappa'}\bigg(4+\kappa'+2\rho-2j\big(2+\rho\big) \bigg) \bigg(4+\kappa'-2\rho+2j \big(2+\rho\big) \bigg)
\end{split}
\end{equation}
almost surely.  For $j > \lceil J_{\kappa',\rho} \rceil$, almost surely $\CI_j' = \emptyset$.  These results similarly hold if $\eta'$ is a whole-plane $\SLE_{\kappa'}(\rho)$ process.

Similarly, let $\CL_j'$ (resp.\ $\CR_j'$) be the set of points on $\partial \D$ which $\eta'$ hits exactly $j$ times while traveling in the clockwise (resp.\ counterclockwise) direction.  Then
\begin{equation}
\label{eqn::radial_boundary_cfl_left}
\begin{split}
\dimH(\CL_j') &= \frac{1}{2\kappa'}\bigg(\kappa'-2 j\big(2+\rho\big) \bigg) \bigg(2+j \big(2+\rho\big) \bigg)\\
&\quad\text{almost surely on}\quad \{\CL_j' \neq \emptyset\}.
\end{split}
\end{equation}
and
\begin{align}
\quad\quad\quad \dimH(\CR_j') &= \frac{1}{2\kappa'}\bigg(\kappa'+2\rho-2j(2+\rho) \bigg)\bigg(2-\rho+j(2+\rho) \bigg)\notag \\
&\quad\text{almost surely on}\quad \{\CR_j' \neq \emptyset\}. \label{eqn::radial_boundary_cfl_right}
\end{align}
\end{theorem}

The reason that we restrict to the case that $\rho > \tfrac{\kappa'}{2}-4$ is that for $\rho \leq \tfrac{\kappa'}{2}-4$ such processes almost surely fill their own outer boundary.  That is, for any time $t$, the outer boundary of the range of the path drawn up to time $t$ is almost surely contained in $\eta'([t,\infty])$ and processes of this type fall outside of the framework described in \cite{IG4}.

The proofs of Theorem~\ref{thm::double_point_dimension} and Theorem~\ref{thm::cut_set_dimension} are based on using various forms of $\SLE$ duality which arises in the interpretation of the $\SLE_\kappa$ and $\SLE_\kappa(\ul{\rho})$ processes for $\kappa \in (0,4)$ as flow lines of the vector field $e^{ih/\chi}$ where $h$ is a GFF and $\chi = \tfrac{2}{\sqrt{\kappa}}-\tfrac{\sqrt{\kappa}}{2}$ \cite{\DubedatDuality,dubedat2009,IG1,IG3,IG4}.  We will refer to these paths simply as ``GFF flow lines.''  (An overview of this theory is provided in Section~\ref{subsec::imaginary_geometry}.) The duality statement which is relevant for the cut-set (see Figure~\ref{fig::counterflowline_and_flowline}) is that the left (resp.\ right) boundary of an $\SLE_{\kappa'}$ process is given by an $\SLE_\kappa$ flow line of a GFF with angle $\tfrac{\pi}{2}$ (resp.\ $-\tfrac{\pi}{2}$).  Thus the cut set dimension is given by the dimension of the intersection of two flow lines with an angle gap of
\begin{equation}
\label{eqn::angle_cut}
\theta_{\rm cut} = \pi.
\end{equation}
Another form of duality which describes the boundary of an $\SLE_{\kappa'}$ process before and after hitting a given boundary point and also arises in the GFF framework allows us to relate the double point dimension to the dimension of the intersection of GFF flow lines with an angle gap of \cite{IG3}
\begin{equation}
\label{eqn::angle_double}
\theta_{\rm double} = \pi \left( \frac{\kappa-2}{2-\tfrac{\kappa}{2}} \right).
\end{equation}
We will explain this in more detail in Section~\ref{sec::main_results}.
The set of points which a whole-plane or radial $\SLE_\kappa(\rho)$ process for $\kappa \in (0,4)$ and $\rho \in (-2,\frac{\kappa}{2}-2)$ hits $j$ times (in the interior of the domain) is locally absolutely continuous with respect to the intersection of two flow lines with an angle gap of
\begin{equation}
\label{eqn::angle_whole_plane}
\theta_j = 2\pi (j-1) \left(\frac{2+\rho}{4-\kappa} \right) \quad\text{for}\quad 2 \leq j \leq \lceil J_{\kappa,\rho} \rceil;
\end{equation}
see \cite[Proposition~3.32]{IG4}.  The angle gap which gives the dimension of the self-intersection set contained in the interior of the domain for $\kappa' > 4$ and $\rho \in (\tfrac{\kappa'}{2}-4,\tfrac{\kappa'}{2}-2)$ is given by
\begin{equation}
\label{eqn::angle_whole_plane_cfl}
\theta_j' = \pi \left(\frac{2 j (2+\rho) -2\rho - \kappa'}{\kappa'-4}\right) \quad\text{for}\quad 2 \leq j \leq \lceil J_{\kappa',\rho} \rceil;
\end{equation}
see \cite[Proposition~4.10]{IG4}.  Thus Theorems~\ref{thm::double_point_dimension}--\ref{thm::whole_plane_self_intersection_dimension_cfl} follow from (with the exception of \eqref{eqn::radial_boundary}, \eqref{eqn::radial_boundary_cfl_left}, \eqref{eqn::radial_boundary_cfl_right}):

\begin{theorem}
\label{thm::two_flowline_dimension}
Suppose that $h$ is a GFF on $\HH$ with piecewise constant boundary data.  Fix $\kappa\in (0,4)$, angles
\[ \theta_1<\theta_2<\theta_1+\left(\frac{\kappa\pi}{4-\kappa} \right),\]
and let
\[\rho=\frac{1}{\pi} (\theta_2-\theta_1) \left(2-\frac{\kappa}{2}\right) - 2.\]
For $i=1,2$, let $\eta_{\theta_i}$ be the flow line of $h$ starting from $0$.  We have that
\[ \dimH(\eta_{\theta_1} \cap \eta_{\theta_2} \cap \h)=2-\frac{1}{2\kappa}\left(\rho+\frac{\kappa}{2}+2\right)\left(\rho-\frac{\kappa}{2}+6\right)\]
almost surely on the event $\{ \eta_{\theta_1} \cap \eta_{\theta_2} \cap \h \neq \emptyset\}$.
\end{theorem}

Theorem~\ref{thm::two_flowline_dimension} gives the dimension of the intersection of two flow lines in the bulk.  The following result gives the dimension of the intersection of one path with the boundary.

\begin{theorem}
\label{thm::boundary_dimension}
Fix $\kappa>0$ and $\rho\in((-2) \vee (\tfrac{\kappa}{2}-4),\tfrac{\kappa}{2}-2)$.  Let $\eta$ be an $\SLE_\kappa(\rho)$ process with a single force point located at $0^+$.  Almost surely,
\begin{equation}
\label{eqn::boundary_dimension}
\dimH(\eta\cap\R_+)=1-\frac{1}{\kappa}(\rho+2)\left(\rho+4-\frac{\kappa}{2}\right).
\end{equation}
\end{theorem}

(Recall that $\tfrac{\kappa}{2}-4$ is the threshold at which such processes become boundary filling and $-2$ is the threshold for these processes to be defined.)  In the case that $\rho = \tfrac{\theta}{\pi} (2-\tfrac{\kappa}{2})-2$ for $\theta > 0$ and $\kappa \in (0,4)$, we say that $\eta$ intersects $\partial \h$ with an angle gap of $\theta$.  This comes from the interpretation of such an $\SLE_\kappa(\rho)$ process as a GFF flow line explained in Section~\ref{subsec::imaginary_geometry}.  See, in particular, Figure~\ref{fig::conditional_law}.  By \cite[Proposition~3.33]{IG4}, applying Theorem~\ref{thm::boundary_dimension} with an angle gap of $\theta_{j+1}$ where $\theta_j$ is as in \eqref{eqn::angle_whole_plane} gives \eqref{eqn::radial_boundary} of Theorem~\ref{thm::whole_plane_self_intersection_dimension}.  Similarly, by \cite[Proposition~4.11]{IG4}, applying Theorem~\ref{thm::boundary_dimension} with an angle gap of
\begin{equation}
\label{eqn::angle_gap_cfl_left}
 \phi_{j,L} = \pi\left(\frac{4-\kappa'+2j(2+\rho)}{\kappa'-4}\right)
\end{equation}
gives \eqref{eqn::radial_boundary_cfl_left} and with an angle gap of
\begin{equation}
\label{eqn::angle_gap_cfl_right}
\phi_{j,R} = \pi\left(\frac{4-\kappa'-2\rho+2j(2+\rho)}{\kappa'-4}\right)
\end{equation}
gives \eqref{eqn::radial_boundary_cfl_right}.  Theorem~\ref{thm::boundary_dimension} is proved first by computing the boundary intersection dimension for $\kappa \in (0,4)$ and then using $\SLE$ duality to extend to the case that $\kappa' > 4$.  We remark that an alternative proof to the lower bound of Theorem~\ref{thm::boundary_dimension} for $\kappa \in (8/3,4)$ is given in \cite{WERNER_WU_CLE} using the relationship between the $\SLE_\kappa(\rho)$ processes for these $\kappa$ values and the Brownian loop soups.  We obtain as a corollary (when $\rho=0$) the following which was first proved in \cite{sheffield2008}.

\begin{corollary}
\label{cor::bigger_than_4_boundary_intersection}
Fix $\kappa' \in (4,8)$ and let $\eta$ be an $\SLE_{\kappa'}$ process in $\h$ from $0$ to $\infty$.  Then, almost surely
\[ \dimH(\eta \cap \R)=2-\frac{8}{\kappa'}.\]
\end{corollary}

One of the main inputs in the proof of Theorem~\ref{thm::two_flowline_dimension} and Theorem~\ref{thm::boundary_dimension} is the following theorem, which gives the exponent for the probability that an $\SLE_\kappa(\ul{\rho})$ process gets very close to a given boundary point.

\begin{theorem}
\label{thm::one_point}
Fix $\kappa >0$, $\rho_{1,R} > -2$, $\rho_{2,R} \in \R$ such that $\rho_{1,R} + \rho_{2,R} > \tfrac{\kappa}{2}-4$.  Let $\eta$ be an $\SLE_{\kappa}(\rho_{1,R},\rho_{2,R})$ process with force points $(0^+,1)$.  Let
\begin{equation}
\label{eqn::sle_kappa_rho_exponent}
 \alpha=\frac{1}{\kappa}(\rho_{1,R}+2)\left(\rho_{1,R}+\rho_{2,R}+4-\frac{\kappa}{2}\right).
\end{equation}
For each $\epsilon > 0$, we let $\tau_\epsilon = \inf\{t \geq 0 : \eta(t) \in \partial B(1,\epsilon)\}$.  We have that
\begin{equation}
\label{eqn::onepointestimate}
 \p[\tau_\epsilon < \infty]=\eps^{\alpha+o(1)}\quad\text{as}\quad \epsilon \to 0.
\end{equation}
\end{theorem}

By taking $\rho = \rho_{1,R} \in ((-2) \vee (\tfrac{\kappa}{2}-4),\tfrac{\kappa}{2}-2)$ and $\rho_{2,R} = 0$, Theorem~\ref{thm::one_point} gives the exponent for the probability that an $\SLE_\kappa(\rho)$ process gets close to a fixed point on the boundary.  Theorem~\ref{thm::one_point} is proved (in somewhat more generality) in Section~\ref{subsec::upper_bound}.

\subsection*{Outline}

The remainder of this article is structured as follows.  In Section~\ref{sec::preliminaries}, we will review the definition and important properties of the $\SLE_\kappa$ and $\SLE_\kappa(\rho)$ processes.  We will also describe the coupling between $\SLE$ and the Gaussian free field.  Next, in Section~\ref{sec::sle_kappa_rho_boundary}, we will compute the Hausdorff dimension of $\SLE_\kappa(\rho)$ intersected with the boundary.  We will extend this to compute the dimension of the intersection of two GFF flow lines in Section~\ref{sec::flow_line_intersection}.  Finally, in Section~\ref{sec::main_results} we will complete the proof of Theorem~\ref{thm::double_point_dimension}.

\section{Preliminaries}
\label{sec::preliminaries}

We will give an overview of the $\SLE_\kappa$ and $\SLE_\kappa(\rho)$ processes in Section~\ref{subsec::sle}.  Next, in Section~\ref{subsec::imaginary_geometry}, we will give an overview of the $\SLE$/GFF coupling and then use the coupling to establish several useful lemmas regarding the behavior of the $\SLE_\kappa$ and $\SLE_\kappa(\ul{\rho})$ processes.   In Section~\ref{subsec::radon_nikodym}, we will compute the Radon-Nikodym derivative associated with a change of domains and perturbation of force points for an $\SLE_\kappa(\ul{\rho})$ process.  Finally, in Section~\ref{subsec::distort} we will record some useful estimates for conformal maps. Throughout, we will make use of the following notation.  Suppose that $f,g$ are functions.  We will write $f \asymp g$ if there exists a constant $C \geq 1$ such that $C^{-1} f(x) \leq g(x) \leq C f(x)$ for all $x$.  We will write $f \lesssim g$ if there exists a constant $C > 0$ such that $f(x) \leq C g(x)$ and $f \gtrsim g$ if $g \lesssim f$.

\subsection{$\SLE_\kappa$ and $\SLE_\kappa(\rho)$ processes}
\label{subsec::sle}

We will now give a very brief introduction to $\SLE$.  More detailed introductions can be found in many excellent surveys of the subject, e.g., \cite{W03, lawler2005}.  Chordal $\SLE_\kappa$ in $\h$ from $0$ to $\infty$ is defined by the random family of conformal maps $(g_t)$ obtained by solving the Loewner ODE
\begin{equation}
\label{eqn::loewner_ode}
\partial_t g_t(z) = \frac{2}{g_t(z) - W_t},\quad\quad g_0(z) = z
\end{equation}
with $W = \sqrt{\kappa} B$ and $B$ a standard Brownian motion.  Write $K_t := \{z \in \h: \tau(z) \leq t \}$ where $\tau(z)$ is the swallowing time of $z$ defined by $\sup\{t\ge 0: \min_{s\in[0,t]}|g_s(z)-W_s|>0\}$.  Then $g_t$ is the unique conformal map from $\h_t := \h \setminus K_t$ to $\h$ satisfying $\lim_{|z| \to \infty} |g_t(z) - z| = 0$.

Rohde and Schramm showed that there almost surely exists a curve $\eta$ (the so-called $\SLE$ \emph{trace}) such that for each $t \geq 0$ the domain $\h_t$ of $g_t$ is the unbounded connected component of $\h \setminus \eta([0,t])$, in which case the (necessarily simply connected and closed) set $K_t$ is called the ``filling'' of $\eta([0,t])$ \cite{\RohdeSchramm}.  An $\SLE_\kappa$ connecting boundary points $x$ and $y$ of an arbitrary simply connected Jordan domain can be constructed as the image of an $\SLE_\kappa$ on $\h$ under a conformal transformation $\varphi \colon \h \to D$ sending $0$ to $x$ and $\infty$ to $y$.  (The choice of $\varphi$ does not affect the law of this image path, since the law of $\SLE_\kappa$ on $\h$ is scale invariant.)  For $\kappa\in[0,4],$ $\SLE_\kappa$ is simple and, for $\kappa > 4$, $\SLE_\kappa$ is self-intersecting \cite{\RohdeSchramm}.  The dimension of the path is $1+\tfrac{\kappa}{8}$ for $\kappa \in [0,8]$ and $2$ for $\kappa > 8$ \cite{beffara2008}.

An $\SLE_{\kappa}(\underline{\rho}_L;\underline{\rho}_R)$ process is a generalization of
$\SLE_{\kappa}$ in which one keeps track of additional marked points which are called \emph{force points}.  These processes were first introduced in \cite[Section 8.3]{LawlerSchrammWernerConformalRestrictionChordal}.  Fix $\underline{x}_L=(x_{\ell,L}<\cdots<x_{1,L}\le 0)$ and $\underline{x}_R=(0\le x_{1,R}<\cdots<x_{r,R})$.  We associate with each $x_{i,q}$ for $q \in \{L,R\}$ a weight $\rho_{i,q} \in \R$.  An $\SLE_{\kappa}(\underline{\rho}_L; \underline{\rho}_R)$ process with force points $(\underline{x}_L; \underline{x}_R)$ is the measure on continuously growing compact hulls $K_t$ generated by the Loewner chain with $W_t$ replaced by the solution to the system of SDEs:
\begin{equation}
\label{eqn::slesde}
\begin{split}
d W_t &= \sum_{i=1}^{\ell} \frac{\rho_{i,L}}{W_t-V_t^{i,L}} dt + \sum_{i=1}^{r} \frac{\rho_{i,R} }{W_t-V_t^{i,R}} dt + \sqrt{\kappa}d B_t,\\
d V_t^{i,q}&= \frac{2}{V_t^{i,q}- W_t} dt,\quad V_0^{i,q} = x_{i,q},\quad i\in\N,\quad q\in\{L, R\}.
\end{split}
\end{equation}
It is explained in \cite[Section 2]{IG1} that for all $\kappa>0,$ there is a unique solution to \eqref{eqn::slesde} up until the \emph{continuation threshold} is hit --- the first time $t$ for which either
\begin{equation*}
\sum_{i:V^{i,L}_t=W_t} \rho_{i,L}\le -2 \quad\mbox{ or }\quad \sum_{i:V^{i,R}_t=W_t} \rho_{i,R}\le -2.
\end{equation*}
The almost sure continuity of the $\SLE_\kappa(\ul{\rho})$ processes is proved in \cite[Theorem 1.3]{IG1}.  Let
\begin{equation}
\label{eqn::rho_bar}
\overline{\rho}_{j,q}=\sum_{i=0}^{j}\rho_{i,q} \quad\text{for}\quad q \in \{L,R\} \quad\text{and}\quad j \in \N
\end{equation}
with the convention that $\rho_{0,L}=\rho_{0,R}=0$,  $x_{0,L}=0^-$, $x_{\ell+1,L}=-\infty$, $x_{0,R}=0^+$, and $x_{r+1,R}=+\infty$.  The value of $\ol{\rho}_{k,R}$ determines how the process interacts with the interval $(x_{k,R},x_{k+1,R})$ (and likewise when $R$ is replaced with $L$).  In particular:

\begin{lemma}
\label{lem::sle_kappa_rho_boundary_interaction}
Suppose that $\eta$ is an $\SLE_\kappa(\ul{\rho}_L;\ul{\rho}_R)$ process in $\h$ from $0$ to $\infty$ with force points located at $(\ul{x}_L;\ul{x}_R)$.
\begin{enumerate}[(i)]
\item If $\ol{\rho}_{k,R} \geq \tfrac{\kappa}{2}-2$, then $\eta$ almost surely does not hit $(x_{k,R},x_{k+1,R})$.
\item If $\kappa \in (0,4)$ and $\ol{\rho}_{k,R} \in (\tfrac{\kappa}{2}-4,-2]$, then $\eta$ can hit $(x_{k,R},x_{k+1,R})$ but cannot be continued afterwards.
\item If $\kappa > 4$ and $\ol{\rho}_{k,R} \in (-2,\tfrac{\kappa}{2}-4]$, then $\eta$ can hit $(x_{k,R},x_{k+1,R})$ and be continued afterwards.  Moreover, $\eta \cap (x_{k,R},x_{k+1,R})$ is almost surely an interval.
\item If $\ol{\rho}_{k,R} \in ((-2) \vee (\tfrac{\kappa}{2}-4), \tfrac{\kappa}{2}-2)$ then $\eta$ can hit and bounce off of $(x_{k,R},x_{k+1,R})$.  Moreover, $\eta \cap (x_{k,R},x_{k+1,R})$ has empty interior.
\end{enumerate}
\end{lemma}
\begin{proof}
See \cite[Remark 5.3 and Theorem 1.3]{IG1} as well as \cite[Lemma 15]{\DubedatDuality}.
\end{proof}

In this article, it will also be important for us to consider \emph{radial} $\SLE_\kappa$ and $\SLE_\kappa(\rho)$ processes.  These are typically defined using the radial Loewner equation.  On the unit disk $\D$, this is described by the ODE
\begin{equation}
\label{eqn::lde.radial}
\partial_t g_t(z) = -g_t(z) \frac{g_t(z) + W_t}{g_t(z) - W_t},
\quad g_0(z)=z
\end{equation}
where $W_t$ is a continuous function which takes values in $\partial \D$.  For $w \in \partial \D$, radial $\SLE_\kappa$ starting from $w$ is the growth process associated with \eqref{eqn::lde.radial} where $W_t = w e^{i \sqrt{\kappa} B_t}$ and $B$ is a standard Brownian motion.  For $w,v\in \partial \D$, radial $\SLE_\kappa(\rho)$ with starting configuration $(w,v)$ is the growth process associated with the solution of \eqref{eqn::lde.radial} where the driving function solves the SDE
\begin{equation}
\label{eqn::radial_sle_kappa_rho}
dW_t = -\frac{\ka}{2} W_t\d t
	+ i \sqrt{\ka} W_t \d B_t
	- \frac{\rho}{2} W_t \frac{W_t+V_t}{W_t-V_t}\d t,\quad
W_0=w
\end{equation}
with $V_t=g_t(v)$, the force point.  The continuity of the radial $\SLE_\kappa(\rho)$ processes for $\rho > -2$ can be extracted from the continuity of chordal $\SLE_\kappa(\ul{\rho})$ processes given in \cite[Theorem 1.3]{IG1}; this is explained in \cite[Section~2.1]{IG4}.  The value of $\rho$ for a radial $\SLE_\kappa(\rho)$ process has the same interpretation as in the setting of chordal $\SLE_\kappa(\rho)$ explained in Lemma~\ref{lem::sle_kappa_rho_boundary_interaction}.  That is, the processes are boundary filling for $\rho \in (-2,\tfrac{\kappa}{2}-4]$ (for $\kappa > 4$), boundary hitting but not filling for $\rho \in ( (-2) \vee (\tfrac{\kappa}{2}-4),\tfrac{\kappa}{2}-2)$, and boundary avoiding for $\rho \geq \tfrac{\kappa}{2}-2$.  In particular, by the conformal Markov property for radial $\SLE_\kappa(\rho)$, such processes are self-intersecting for $\rho \in (-2,\tfrac{\kappa}{2}-2)$ and fill their own outer boundary for $\rho \in (-2,\tfrac{\kappa}{2}-4]$ ($\kappa > 4$).  The latter means that, for any time $t$, the outer boundary of the range of $\eta$ up to time $t$ is almost surely contained in $\eta([t,\infty))$.

\subsubsection*{Martingales}

From the form of \eqref{eqn::slesde} and the Girsanov theorem, it follows that the law of an $\SLE_\kappa(\ul{\rho})$ process can be constructed by reweighting the law of an ordinary $\SLE_\kappa$ process by a certain local martingale, at least until the first time $\tau$ that $W$ hits one of the force points $V^{i,q}$ \cite{\WernerSLEKappaRho}.  It is shown in \cite[Theorem~6 and Remark~7]{\SW} that this local martingale can be expressed in the following more convenient form.  Suppose $x_{1,L}<0<x_{1,R}$ and define
\begin{equation}
\label{eqn::martingalebetweensles}
\begin{split}
M_t &=\prod_{i,q}\left|g'_t(x_{i,q})\right|^{\frac{(4-\kappa+\rho_{i,q})\rho_{i,q}}{4\kappa}}\times \prod_{i,q} \left| W_t-V_t^{i,q} \right|^{\frac{\rho_{i,q}}{\kappa}}\\
    &\times \prod_{(i,q) \neq (i',q')}\left|V_t^{i,q}-V_t^{i',q'}\right|^{\frac{\rho_{i,q}\rho_{i',q'}}{2\kappa}}.
\end{split}
\end{equation}
Then $M_t$ is a local martingale and the law of a standard $\SLE_{\kappa}$ process weighted by $M$ (up to time $\tau$, as above) is equal to that of an $\SLE_{\kappa}(\underline{\rho}_L;\underline{\rho}_R)$ process with force points $(\underline{x}_L;\underline{x}_R)$.  We remark that there is an analogous martingale in the setting of radial $\SLE_\kappa(\rho)$ processes \cite[Equation 9]{\SW}, a special case of which we will describe and make use of in Section~\ref{sec::flow_line_intersection}.

One application of this that will be important for us is as follows.  Suppose that $\eta$ is an $\SLE_{\kappa}(\rho_L;\rho_R)$ process with only two force points $x_L<0<x_R$. If we weight the law of $\eta$ by the local martingale
\begin{equation}
\label{eqn::left_conditioning}
M^L_t=|W_t-V_t^L|^{\frac{\kappa-4-2\rho_L}{\kappa}}\times|V_t^L-V_t^R|^{\frac{(\kappa-4-2\rho_L)\rho_R}{2\kappa}}
\end{equation}
then the law of the resulting process is that of an $\SLE_{\kappa}(\wh{\rho}_L;\rho_R)$ process where $\wh{\rho}_L=\kappa-4-\rho_L$. If $\rho_L<\tfrac{\kappa}{2}-2$ so that $\wh{\rho}_L > \tfrac{\kappa}{2}-2$, Lemma~\ref{lem::sle_kappa_rho_boundary_interaction} implies that the reweighted process almost surely does not hit $(-\infty,x_L)$.

\subsection{$\SLE$ and the $\GFF$}
\label{subsec::imaginary_geometry}

We are now going to give a brief overview of the coupling between $\SLE$ and the GFF.  We refer the reader to \cite[Sections~1~and~2]{IG1} as well as \cite[Section 2]{IG2} for a more detailed overview.  Throughout, we fix $\kappa \in (0,4)$ and $\kappa'=16/\kappa > 4$.

Suppose that $D \subseteq \C$ is a given domain.  The Sobolev space $H_0^1(D)$ is the Hilbert space closure of $C_0^\infty(D)$ with respect to the Dirichlet inner product
\begin{equation}
\label{eqn::dirichlet_inner_product}
 (f,g)_\nabla = \frac{1}{2\pi} \int \nabla f(x) \cdot \nabla g(x) dx.
\end{equation}
The zero-boundary Gaussian free field (GFF) $h$ on $D$ is given by
\begin{equation}
\label{eqn::gff_series}
h = \sum_n \alpha_n f_n
\end{equation}
where $(\alpha_n)$ is a sequence of i.i.d.\ $N(0,1)$ random variables and $(f_n)$ is an orthonormal basis for $H_0^1(D)$.  The sum \eqref{eqn::gff_series} does not converge in $H_0^1(D)$ (or any space of functions) but rather in an appropriate space of distributions.  The GFF $h$ with boundary data $f$ is given by taking the sum of the zero-boundary GFF on $D$ and the function $F$ in $D$ which is harmonic and is equal to $f$ on $\partial D$.  See \cite{sheffield2007} for a detailed introduction.

\begin{figure}[ht!]
\begin{center}
\includegraphics[scale=0.85]{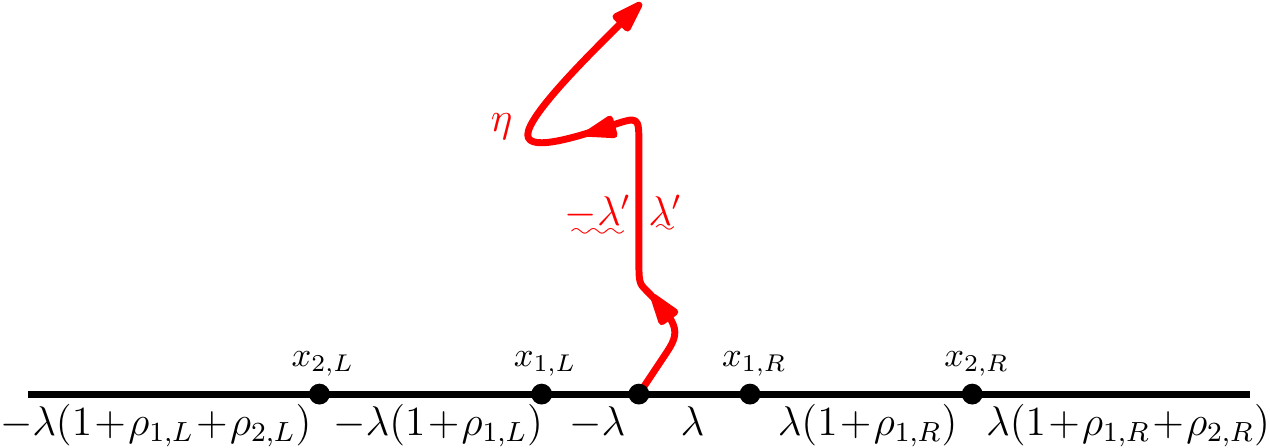}
\end{center}
\caption{\label{fig::gff_boundary_data_flow}
Suppose that $h$ is a GFF on $\h$ whose boundary data is as indicated above.  Then the flow line $\eta$ of $h$ starting from $0$ is an $\SLE_{\kappa}(\rho_{2,L},\rho_{1,L};\rho_{1,R},\rho_{2,R})$ process ($\kappa \in (0,4)$) from $0$ to $\infty$ with force points located at $x_{2,L} < x_{1,L} < 0 < x_{1,R} < x_{2,R}$.  The conditional law of $h$ given $\eta$ (or $\eta$ up to a stopping time) is that of a GFF off of $\eta$ with the boundary data as illustrated on $\eta$; the notation \uwave{$x$} is shorthand for $x + \chi \cdot {\rm winding}$ and is explained in detail in \cite[Figures~1.9~and~1.10]{IG1}.  The boundary data for the coupling of $\SLE_\kappa(\ul{\rho})$ with many force points arises as the obvious generalization of the above.}
\end{figure}

\begin{figure}[ht!]
\includegraphics[scale=0.85]{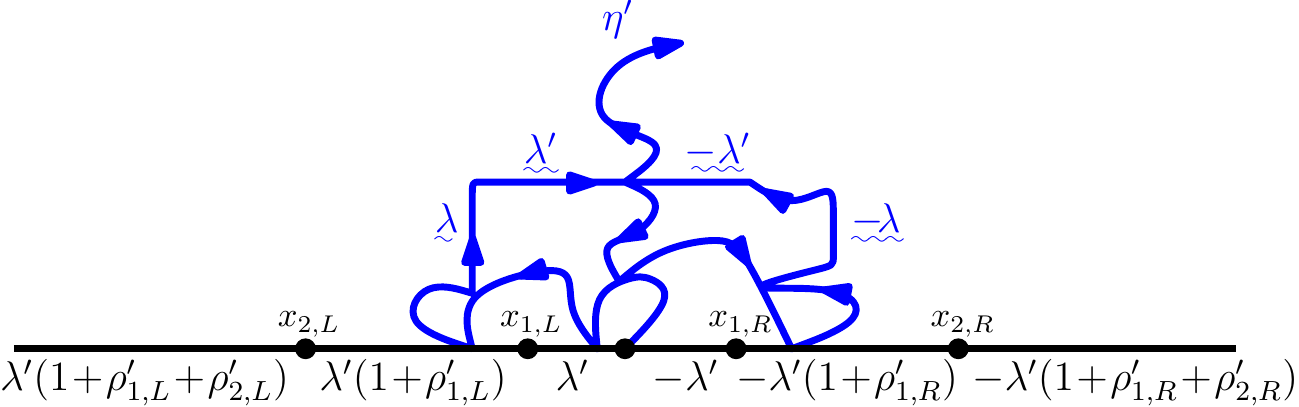}
\caption{\label{fig::gff_boundary_data_cfl}
Suppose that $h$ is a GFF on $\h$ whose boundary data is as indicated above.  Then the counterflow line $\eta'$ of $h$ starting from $0$ is an $\SLE_{\kappa'}( \rho_{2,L}',\rho_{1,L}'; \rho_{1,R}', \rho_{2,R}')$ process ($\kappa'>4$) from $0$ to $\infty$ with force points located at $x_{2,L} < x_{1,L} < 0 < x_{1,R} < x_{2,R}$.  The conditional law of $h$ given $\eta'$ (or $\eta'$ up to a stopping time) is that of a GFF off of $\eta'$ with the indicated boundary data; the notation \uwave{$x$} is shorthand for $x + \chi \cdot {\rm winding}$ and is explained in detail in \cite[Figures~1.9~and~1.10]{IG1}.  The boundary data for the coupling of $\SLE_{\kappa'}(\ul{\rho}')$ with many force points arises as the obvious generalization of the above.
}
\end{figure}

\begin{figure}[ht!]
\begin{center}
\subfigure[Monotonicity of flow lines.]{\includegraphics[scale=0.85]{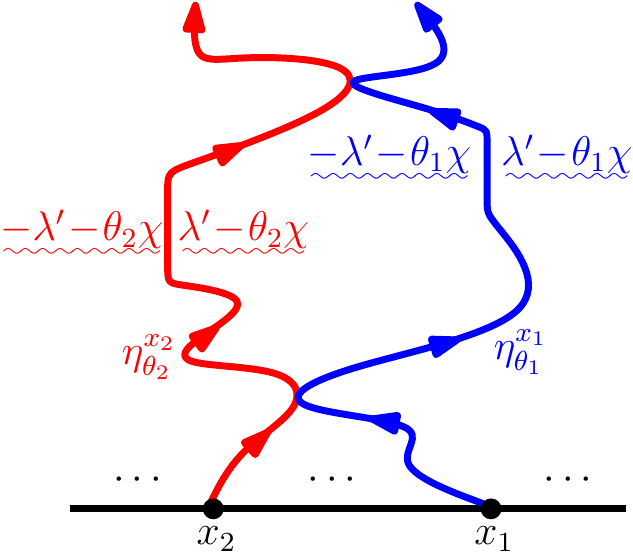}}
\hspace{0.02\textwidth}
\subfigure[Flow lines merging.]{\includegraphics[scale=0.85]{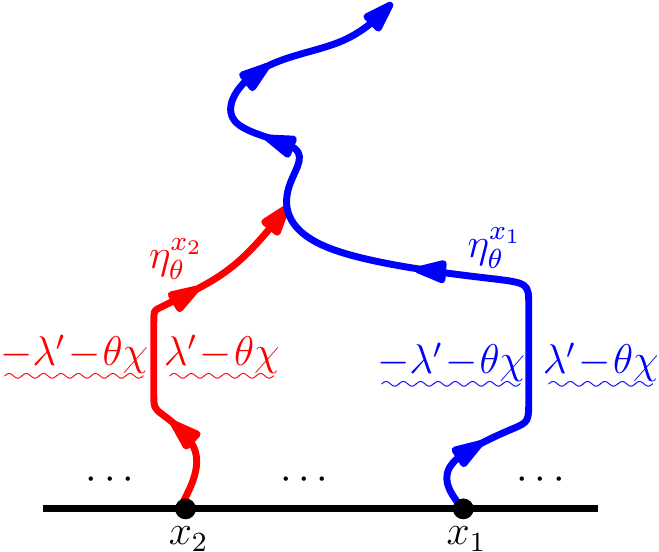}}
\end{center}
\caption{\label{fig::two_flowlines} Suppose that $h$ is a GFF on $\h$ with piecewise constant boundary data and $x_1,x_2 \in \partial \h$ with $x_2 \leq x_1$.  Fix angles $\theta_1,\theta_2$ and, for $i=1,2$, let $\eta_{\theta_i}^{x_i}$ be the flow line of $h$ with angle $\theta_i$ starting from $x_i$.  If $\theta_2 > \theta_1$, then $\eta_{\theta_2}^{x_2}$ almost surely stays to the left of (but may bounce off of) $\eta_{\theta_1}^{x_1}$.  If $\theta_1=\theta_2=\theta$, then $\eta_\theta^{x_1}$ merges with $\eta_\theta^{x_2}$ upon intersecting after which the paths never separate.}
\end{figure}

\begin{figure}[ht!]
\begin{center}
\includegraphics[scale=0.85]{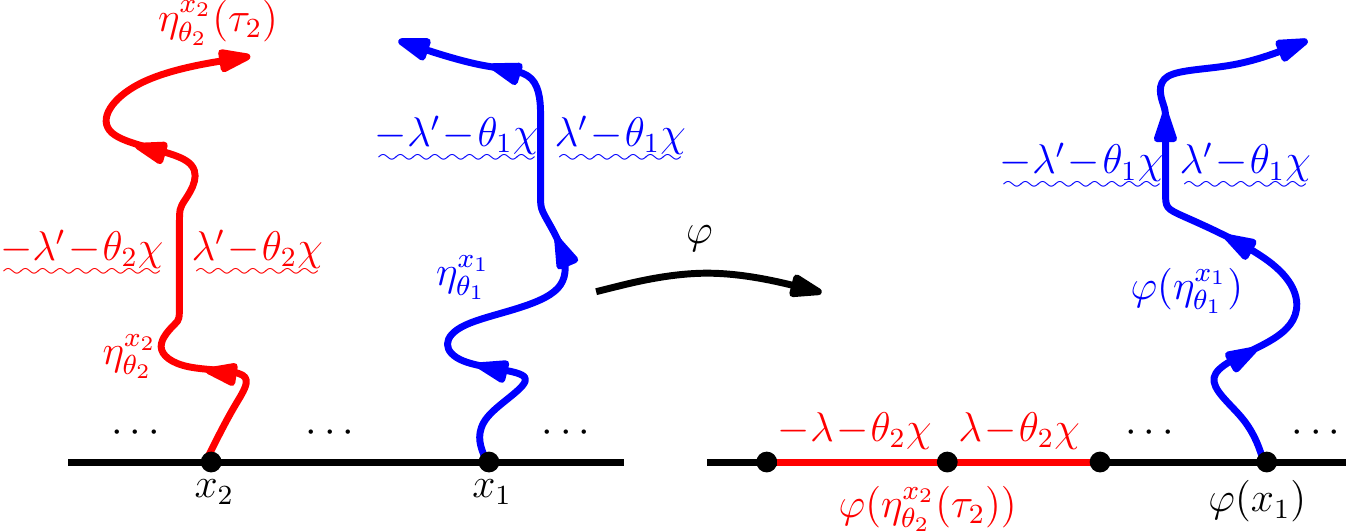}
\end{center}
\caption{\label{fig::conditional_law} Assume that we have the same setup as in Figure~\ref{fig::two_flowlines} and that $\tau_2$ is a stopping time for $\eta_{\theta_2}^{x_2}$.  Then we can compute the conditional law of $\eta_{\theta_1}^{x_1}$ given $\eta_{\theta_2}^{x_2}|_{[0,\tau_2]}$.  Let $\varphi$ be a conformal map which takes the unbounded connected component of $\h \setminus \eta_{\theta_2}^{x_2}([0,\tau_2])$ to $\h$ and let $h_2 = h \circ \varphi^{-1} - \chi \arg (\varphi^{-1})'$.  Then $\varphi(\eta_{\theta_1}^{x_1})$ is the flow line of $h_2$ starting from $\varphi(x_1)$ with angle $\theta_1$ and we can read off its conditional law from the boundary data of $h_2$ as in Figure~\ref{fig::gff_boundary_data_flow}.}
\end{figure}

\begin{figure}[ht!]
\begin{center}
\includegraphics[scale=0.85]{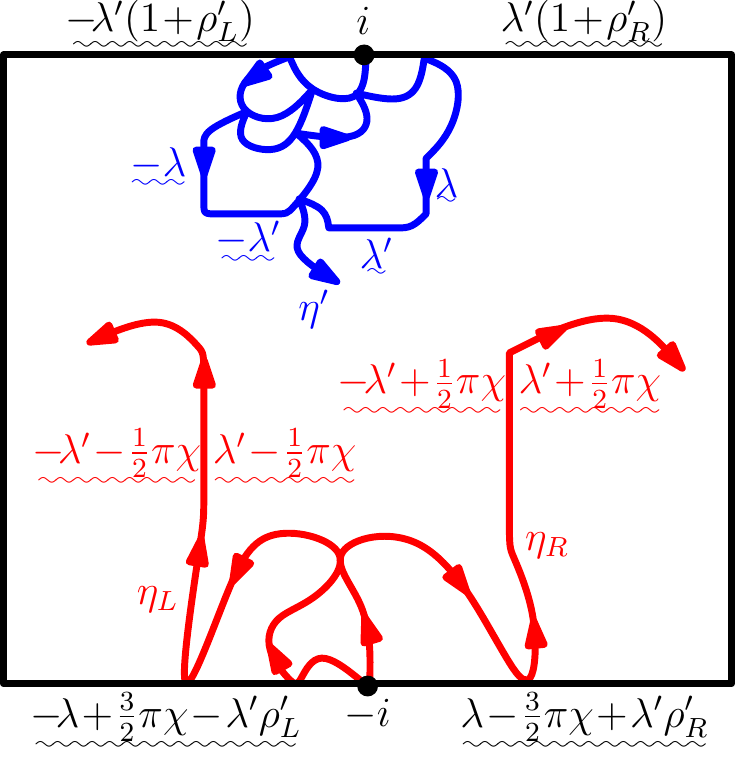}
\end{center}
\caption{
\label{fig::counterflowline_and_flowline}  Let $h$ be a GFF on $[-1,1]^2$ with the illustrated boundary data.  Then the counterflow line $\eta'$ of $h$ from $i$ to $-i$ is an $\SLE_{\kappa'}(\rho_L';\rho_R')$ process ($\kappa'>4$) with force points located at $(i)^-,(i)^+$ (immediately to the left and right of $i$).  The left (resp.\ right) boundary $\eta_L$ (resp.\ $\eta_R$) of $\eta'$ is given by the flow line of $h$ with angle $\tfrac{\pi}{2}$ (resp.\ $-\tfrac{\pi}{2}$) starting from $-i$ and targeted at $i$; these paths can be drawn if $\rho_L',\rho_R' \geq \tfrac{\kappa'}{2}-4$.  Explicitly, $\eta_L$ (resp.\ $\eta_R$) is an $\SLE_\kappa(\kappa-4+\tfrac{\kappa}{4}\rho_L';\tfrac{\kappa}{2}-2+\tfrac{\kappa}{4}\rho_R')$ (resp.\ $\SLE_\kappa(\tfrac{\kappa}{2}-2+\tfrac{\kappa}{4}\rho_L';\kappa-4+\tfrac{\kappa}{4}\rho_R')$) process in $[-1,1]^2$ from $-i$ to $i$ with force points located at $(-i)^-,(-i)^+$ ($\kappa=16/\kappa' \in (0,4)$).  The cut-set of $\eta'$ is given by $\eta_L \cap \eta_R$ and $\eta' \cap \partial ([-1,1]^2) =(\eta_L \cup \eta_R) \cap \partial ([-1,1])^2$.  The same holds if $[-1,1]^2$ is replaced by a proper, simply-connected domain and the boundary data of the GFF is transformed according to \eqref{eqn::ac_eq_rel}.  Finally, if $\rho_L',\rho_R' \geq \tfrac{\kappa'}{2}-4$, then conditional law of $\eta'$ given $\eta_L$ and $\eta_R$ is independently that of an $\SLE_{\kappa'}(\tfrac{\kappa'}{2}-4;\tfrac{\kappa'}{2}-4)$ in each of the bubbles of $[-1,1]^2 \setminus (\eta_L \cup \eta_R)$ which lie to the right of $\eta_L$ and to the left of $\eta_R$.
}
\end{figure}

Let
\begin{equation}
\label{eqn::constants}
\chi=\frac{2}{\sqrt{\kappa}}-\frac{\sqrt{\kappa}}{2}, \quad \lambda=\frac{\pi}{\sqrt{\kappa}},\quad\text{and}\quad \lambda'=\frac{\pi}{\sqrt{\kappa'}} = \frac{\pi}{4}\sqrt{\kappa} = \lambda-\frac{\pi}{2} \chi.
\end{equation}
Suppose that $\eta$ is an $\SLE_{\kappa}(\ul{\rho}_L;\ul{\rho}_R)$ process in $\h$ from $0$ to $\infty$ with force points $(\ul{x}_L;\ul{x}_R)$, let $(g_t)$ be the associated Loewner flow, $W$ its driving function, and $f_t = g_t - W_t$.  Let $h$ be a GFF on $\h$ with zero boundary values.  It is shown in \cite{She_SLE_lectures, dubedat2009, MakarovSmirnov09,SchrammShe10,HagendorfBauerBernard10,IzyurovKytola10, \SheWeld} that there exists a coupling $(\eta,h)$ such that the following is true. Suppose $\tau$ is any stopping time for $\eta$.  Let $\phi_t^0$ be the function which is harmonic in $\h$ with boundary values (recall~\eqref{eqn::rho_bar})
\begin{equation*}
\left\{\begin{array}{ccc}
       -\lambda(1+\overline{\rho}_{j,L}) & \rm{if} & x \in [f_t(x_{j+1,L}),f_t(x_{j,L})) \\
       \lambda(1+\overline{\rho}_{j,R}) & \rm{if} & x \in (f_t(x_{j,R}), f_t(x_{j+1,R})]. \\
     \end{array}\right.
\end{equation*}
Let
\begin{equation*}
\phi_t(z)=\phi^0_t(f_t(z))-\chi \arg f'_t(z).
\end{equation*}
Then the conditional law of $(h+\phi_0)|_{\h\setminus K_{\tau}}$ given $K_{\tau}$ is equal to the law of $h\circ f_{\tau}+\phi_{\tau}$.  In this coupling, $\eta$ is almost surely determined by $h$ \cite{SchrammShe10,dubedat2009,IG1}.  For $\kappa \in (0,4)$, $\eta$ has the interpretation as being the flow line of the (formal) vector field $e^{i(h+\phi_0)/\chi}$ \cite{\SheWeld} starting from $0$; we will refer to $\eta$ simply as a flow line of $h+\phi_0$.  See Figure~\ref{fig::gff_boundary_data_flow} for an illustration of the boundary data.  The notation \uwave{$x$} is used to indicate that the boundary data for the field is given by $x+\chi \cdot {\rm winding}$ where ``winding'' refers to the winding of the path or domain boundary.  For curves or domain boundaries which are not smooth, it is not possible to make sense of the winding along the curve or domain boundary.  However, the harmonic extension of the winding does make sense.  This notation as well as this point are explained in detail in \cite[Figures~1.9~and~1.10]{IG1}.  When $\kappa=4$, $\eta$ has the interpretation of being the level line of $h+\phi_0$ \cite{SchrammShe10}.  Finally, when $\kappa' > 4$, $\eta'$ has the interpretation of being a ``tree of flow lines'' which travel in the opposite direction of $\eta'$ \cite{IG1,IG4}.  For this reason, $\eta'$ is referred to as a \emph{counterflow line} of $h+\phi_0$ in this case.

If $h$ were a smooth function, $\eta$ a flow line of the vector field $e^{ih/\chi}$, and $\varphi$ a conformal map, then $\varphi(\eta)$ is a flow line of $e^{i\wt{h}/\chi}$ where
\begin{equation}
\label{eqn::ac_eq_rel}
\wt{h} = h \circ \varphi^{-1} - \chi \arg (\varphi^{-1})';
\end{equation}
see \cite[Figure 1.6]{IG1}.  The same is true when $h$ is a GFF and this formula determines the boundary data for coupling the GFF with an $\SLE_\kappa(\ul{\rho}_L;\ul{\rho}_R)$ process on a domain other than $\h$.  See also \cite[Figure 1.9]{IG1}.  $\SLE_\kappa$ flow lines and $\SLE_{\kappa'}$, $\kappa'=16/\kappa \in (4,\infty)$, counterflow lines can be coupled with the same GFF.  In order for both paths to transform in the correct way under the application of a conformal map, one thinks of the flow lines as being coupled with $h$ as described above and the counterflow lines as being coupled with $-h$.  This is because $\chi(\kappa')=-\chi(\kappa)$; see the discussion after the statement of \cite[Theorem~1.1]{IG1}.  This is why the signs of the boundary data in Figure~\ref{fig::gff_boundary_data_cfl} are reversed in comparison to that in Figure~\ref{fig::gff_boundary_data_flow}.

The theory of how the flow lines, level lines, and counterflow lines of the GFF interact with each other and the domain boundary is developed in \cite{IG1,IG4}.  See, in particular, \cite[Theorem 1.5]{IG1}.  The important facts for this article are as follows.  Suppose that $h$ is a GFF on $\h$ with piecewise constant boundary data. For each $\theta \in \R$ and $x \in \partial\h$, let $\eta^x_{\theta}$ be the flow line of $h$ starting at $x$ with angle $\theta$ (i.e., the flow line of $h+\theta\chi$ starting at $x$). If $\theta_1<\theta_2$ and $x_1\ge x_2$ then $\eta^{x_1}_{\theta_1}$ almost surely stays to the right of $\eta^{x_2}_{\theta_2}$. If $\theta_1=\theta_2,$ then $\eta^{x_1}_{\theta_1}$ may intersect $\eta^{x_2}_{\theta_2}$ and, upon intersecting, the two flow lines merge and never separate thereafter.  See Figure~\ref{fig::two_flowlines}.  Finally, if $\theta_2+\pi>\theta_1>\theta_2,$ then $\eta^{x_1}_{\theta_1}$ may intersect $\eta^{x_2}_{\theta_2}$ and, upon intersecting, crosses and possibly subsequently bounces off of $\eta^{x_2}_{\theta_2}$ but never crosses back.  It is possible to compute the conditional law of one flow line given the realization of several others; see Figure~\ref{fig::conditional_law}.  For simplicity, we use $\eta_{\theta}$ to indicate $\eta^x_{\theta}$ when $x=0$.  If $\eta'$ is a counterflow line coupled with the GFF, then its outer boundary is described in terms of a pair of flow lines starting from the terminal point of $\eta'$ \cite{\DubedatDuality,dubedat2009,IG1,IG4}; see Figure~\ref{fig::counterflowline_and_flowline}.

We are now going to use the $\SLE$/GFF coupling to collect several useful lemmas regarding the behavior of $\SLE_\kappa(\ul{\rho})$ processes.

\begin{lemma}
\label{lem::sle_kappa_rho_cont_in_force_points}
Fix $\kappa > 0$. Suppose that $(x_{n,L})$ (resp.\ $(x_{n,R})$) is a sequence of negative (resp.\ positive) real numbers converging to $x_L \leq 0^-$ (resp.\ $x_R \geq 0^+$) as $n \to \infty$.  For each $n$, suppose that $(W^n, V^{n,L},V^{n,R})$ is the driving triple for an $\SLE_\kappa(\rho_L;\rho_R)$ process in $\h$ with force points located at $(x_{n,L} \leq 0 \leq x_{n,R})$.  Then $(W^{n,L},V^{n,L},V^{n,R})$ converges weakly in law with respect to the local uniform topology to the driving triple $(W,V^L,V^R)$ of an $\SLE_\kappa(\rho_L;\rho_R)$ process with force points located at $(x_L \leq 0 \leq x_R)$ as $n \to \infty$.  The same likewise holds in the setting of multi-force-point $\SLE_\kappa(\ul{\rho})$ processes.
\end{lemma}
\begin{proof}
See \cite[Section 2]{IG1}.
\end{proof}

\begin{figure}[ht!]
\begin{center}
\includegraphics[scale=0.85]{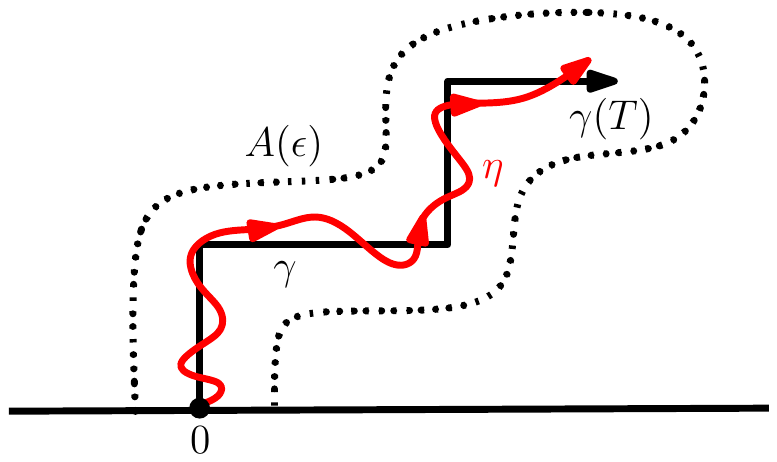}
\end{center}
\caption{\label{fig::sle_kappa_rho_close_to_curve}
Suppose that $\eta$ is an $\SLE_\kappa(\ul{\rho}_L;\ul{\rho}_R)$ process in $\h$ from $0$ to $\infty$ with $x_{1,L}=0^-$ and $x_{1,R}=0^+$ with $\rho_{1,L},\rho_{1,R} > -2$ and fix any deterministic curve $\gamma \colon [0,T] \to \h$.  For each $\epsilon > 0$, let $A(\epsilon)$ be the $\epsilon$  neighborhood of $\gamma$.  We show in Lemma~\ref{lem::sle_kappa_rho_close_to_curve} that with positive probability, $\eta$ gets within distance $\epsilon$ of $\gamma(T)$ before leaving $A(\epsilon)$.}
\end{figure}

\begin{lemma}
\label{lem::sle_kappa_rho_close_to_curve}
Fix $\kappa > 0$. Suppose that $\eta$ is an $\SLE_\kappa(\ul{\rho}_L;\ul{\rho}_R)$ process in $\h$ from $0$ to $\infty$ with force points located at $(\ul{x}_L;\ul{x}_R)$ with $x_{1,L} = 0^-$ and $x_{1,R} = 0^+$ (possibly by taking $\rho_{1,q} = 0$ for $q \in \{L,R\}$).  Assume that $\rho_{1,L}, \rho_{1,R} > -2$.  Suppose that $\gamma \colon [0,T] \to \R$ is any deterministic simple curve in $\ol{\h}$ starting from $0$ and otherwise does not hit $\partial \h$.  Fix $\epsilon > 0$, let $A(\epsilon)$ be the $\epsilon$ neighborhood of $\gamma([0,T])$, and define stopping times
\[ \sigma_1 = \inf\{t \geq 0: |\eta(t) - \gamma(T)| \leq \epsilon\}\quad\text{and}\quad \sigma_2 = \inf\{t \geq 0 : \eta(t) \notin A(\epsilon)\}.\]
Then $\p[ \sigma_1 < \sigma_2]  >0$.
\end{lemma}
\begin{proof}
See Figure \ref{fig::sle_kappa_rho_close_to_curve} for an illustration.  We will use the terminology ``flow line,'' but the proof holds for $\kappa > 0$. By running $\eta$ for a very small amount of time and using that $\p[W_t = V_t^{1,L}] = \p[ W_t = V_t^{1,R}] = 0$ for all $t > 0$ before the continuation threshold is reached \cite[Section 2]{IG1} and then conformally mapping back, we may assume without loss of generality that $\rho_{1,L} = \rho_{1,R} = 0$.  Let $U$ be a Jordan domain which contains $\gamma([0,T])$ and is contained in $A(\epsilon)$.  Assume, moreover, that $\partial U \cap [x_{2,L},x_{2,R}]$ is an interval, say $[y_L,y_R]$, which contains $0$.  Suppose $\kappa \in (0,4)$ and let $h$ be a GFF on $\h$ whose boundary data has been chosen so that its flow line $\eta$ from $0$ is an $\SLE_\kappa(\ul{\rho}_L;\ul{\rho}_R)$ process as in the statement of the lemma.  Pick a point $x_0 \in \partial U$ with $|\gamma(T) - x_0| \leq \epsilon$.  Let $\wt{h}$ be a GFF on $U$ whose boundary conditions are chosen so that its flow line $\wt{\eta}$ starting from $0$ is an $\SLE_\kappa$ process from $0$ to $x_0$.  Let $\wt{\sigma}_1 = \inf\{t \geq 0: |\wt{\eta}(t) - \gamma(T)| \leq \epsilon\}$.  Since $\wt{\eta}|_{(0,\wt{\sigma}_1]}$ almost surely does not hit $\partial U$, it follows that $\wt{X} \equiv \dist(\wt{\eta}|_{[0,\wt{\sigma}_1]}, \partial U \setminus [y_L,y_R]) > 0$ almost surely.  For each $\delta > 0$, let $U_\delta = \{x \in U : \dist(x,\partial U \setminus [y_L,y_R]) > \delta\}$.  Then the laws of $h|_{U_\delta}$ and $\wt{h}|_{U_\delta}$ are mutually absolutely continuous \cite[Proposition~3.2]{IG1}.  Thus the result follows since we can pick $\delta > 0$ sufficiently small so that $\p[ \wt{X} > \delta] > 0$.  This proves the result for $\kappa \in (0,4)$.  For $\kappa' > 4$, one chooses the boundary data for $\wt{h}$ so that the counterflow line is an $\SLE_{\kappa'}(\tfrac{\kappa'}{2}-2;\tfrac{\kappa'}{2}-2)$ process (recall Lemma~\ref{lem::sle_kappa_rho_boundary_interaction}).
\end{proof}

\begin{lemma}
\label{lem::sle_kappa_rho_exit_disk}
Fix $\kappa > 0$. Suppose that $\eta$ is an $\SLE_\kappa(\rho_L;\rho_R)$ process in $\h$ from $0$ to $\infty$ with force points located at $(x_L \leq 0 \leq x_R)$ and with $\rho_R > -2$.  Let $\gamma \colon [0,1] \to \ol{\h}$ be the unit segment connecting $0$ to $i$.  Fix $\epsilon > 0$ and define stopping times $\sigma_1$, $\sigma_2$ as in Lemma~\ref{lem::sle_kappa_rho_close_to_curve}.  For each $x_0^L < 0$ there exists $p_0 = p_0(x_0^L,\epsilon) > 0$ such that for every $x_L \in (-\infty,x_0^L]$ and $x_R \geq 0$, we have that
\begin{equation}
\label{eqn::close_to_line}
 \p[\sigma_1 < \sigma_2] \geq p_0.
\end{equation}
If $\rho_L > -2$, then there exists $p_0 = p_0(\epsilon)$ such that \eqref{eqn::close_to_line} holds for $x_0^L = 0^-$.
\end{lemma}
\begin{proof}
We know that this event has positive probability for each fixed choice of $x_L,x_R$ as above by Lemma~\ref{lem::sle_kappa_rho_close_to_curve}.  Therefore the result follows from Lemma~\ref{lem::sle_kappa_rho_cont_in_force_points} and the results of \cite[Section~4.7]{lawler2005}.
\end{proof}

\begin{figure}[ht!]
\begin{center}
\includegraphics[scale=0.85]{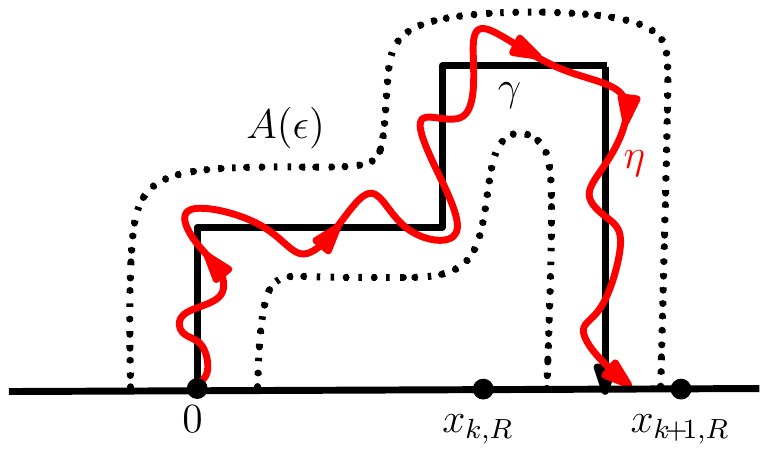}
\end{center}
\caption{\label{fig::sle_kappa_rho_hit_boundary_segment}
Suppose that $\eta$ is an $\SLE_\kappa(\ul{\rho}_L;\ul{\rho}_R)$ process in $\h$ from $0$ to $\infty$ with $x_{1,L}=0^-$ and $x_{1,R}=0^+$ with $\rho_{1,L},\rho_{1,R} > -2$ and fix any deterministic curve $\gamma \colon [0,T] \to \h$ which connects $0$ to $[x_{k,R},x_{k+1,R}]$ where $k$ is such that $\sum_{j=1}^k \rho_{j,R} \in (\tfrac{\kappa}{2}-4,\tfrac{\kappa}{2}-2)$.  For each $\epsilon > 0$, let $A(\epsilon)$ be the $\epsilon$  neighborhood of $\gamma$.  We show in Lemma~\ref{lem::sle_kappa_rho_hit_boundary_segment} that with positive probability, $\eta$ hits $[x_{k,R},x_{k+1,R}]$ before leaving $A(\epsilon)$.}
\end{figure}

\begin{lemma}
\label{lem::sle_kappa_rho_hit_boundary_segment}
Fix $\kappa > 0$. Suppose that $\eta$ is an $\SLE_\kappa(\ul{\rho}_L;\ul{\rho}_R)$ process in $\h$ from $0$ to $\infty$ with force points located at $(\ul{x}_L;\ul{x}_R)$ with $x_{1,L} = 0^-$ and $x_{1,R} = 0^+$ (possibly by taking $\rho_{1,q} = 0$ for $q \in \{L,R\}$).  Assume that $\rho_{1,L}, \rho_{1,R} > -2$.  Fix $k \in \N$ such that $\rho = \sum_{j=1}^k \rho_{j,R} \in (\tfrac{\kappa}{2}-4,\tfrac{\kappa}{2}-2)$ and $\epsilon > 0$.  There exists $p_1 > 0$ depending only on $\kappa,\max_{i,q} |\rho_{i,q}|$, $\rho$, and $\epsilon$ such that if $|x_{2,q}| \geq \epsilon$ for $q \in \{L,R\}$, $x_{k+1,R} - x_{k,R} \geq \epsilon$, and $x_{k,R} \leq \epsilon^{-1}$ then the following is true.  Suppose that $\gamma$ is a simple curve starting from $0$, terminating in $[x_{k,R},x_{k+1,R}]$, and otherwise does not hit $\partial \h$.  Let $A(\epsilon)$ be the $\epsilon$ neighborhood of $\gamma([0,T])$ and let
\[ \sigma_1 = \inf\{t \geq 0 : \eta(t) \in (x_{k,R}, x_{k+1,R})\} \quad\text{and}\quad \sigma_2 = \inf\{t \geq 0 : \eta(t) \notin A(\epsilon)\}.\]
Then $\p[\sigma_1 < \sigma_2] \geq p_1$.
\end{lemma}
\begin{proof}
See Figure \ref{fig::sle_kappa_rho_hit_boundary_segment} for an illustration.  We will use the terminology ``flow line,'' but the proof holds for $\kappa > 0$.  Arguing as in the proof of Lemma~\ref{lem::sle_kappa_rho_close_to_curve}, we may assume without loss of generality that $\rho_{1,L} = \rho_{1,R} = 0$.  Let $U$ be a Jordan domain which contains $\gamma$ and is contained in $A(\epsilon)$.  Assume, moreover, that $\partial U \cap [x_{2,L},x_{2,R}]$ is an interval which contains $0$ and $\partial U \cap [x_{k,R},x_{k+1,R}]$ is also an interval, say $[y_L,y_R]$.  Suppose $\kappa \in (0,4)$.  Let $h$ be a GFF on $\h$ whose boundary data has been chosen so that its flow line $\eta$ from $0$ is an $\SLE_\kappa(\ul{\rho}_L;\ul{\rho}_R)$ process as in the statement of the lemma.  Let $\wt{h}$ be a GFF on $U$ whose boundary conditions are chosen so that its flow line $\wt{\eta}$ starting from $0$ and targeted at $y_R$ is an $\SLE_\kappa(\rho)$ process with a single force point located at $y_L$ with $\rho$ as in the statement of the lemma.  Let $\wt{\sigma}_1$ be the first time that $\wt{\eta}$ hits $[y_L,y_R]$.  Since $\wt{\eta}|_{(0,\wt{\sigma}_1]}$ almost surely does not hit $\partial U \setminus [y_L,y_R]$, it follows that
\[ \dist(\wt{\eta}|_{[0,\wt{\tau}]}, \partial U \setminus ([x_{2,L},x_{2,R}] \cup [y_L,y_R])) > 0\]
almost surely.  Since $\wt{\eta}$ almost surely hits $[y_L,y_R]$, the assertion follows using the same absolute continuity argument for GFFs as in the proof of Lemma~\ref{lem::sle_kappa_rho_close_to_curve}.  As in the proof of Lemma~\ref{lem::sle_kappa_rho_close_to_curve}, one proves the result for $\kappa' > 4$ by taking the boundary conditions for $\wt{h}$ on $U$ so that the counterflow line starting from $0$ is an $\SLE_{\kappa'}(\tfrac{\kappa'}{2}-2;\tfrac{\kappa'}{2}-2,\rho - (\tfrac{\kappa'}{2}-2))$ process.
\end{proof}

\begin{lemma}
\label{lem::sle_kappa_rho_boundary_hitting}
Fix $\kappa  > 0$. Suppose that $\eta$ is an $\SLE_\kappa(\rho_L;\rho_R)$ process in $\h$ from $0$ to $\infty$ with force points located at $(x_L \leq 0 \leq x_R)$ with $\rho_L  \in (\tfrac{\kappa}{2}-4, \tfrac{\kappa}{2}-2)$ and $\rho_R > -2$.  For each $x_0^L \in (-1,0)$ there exists $p_2 = p_2(x_0^L) \in [0,1)$ such that the following is true.  Fix $x_L \in [x_0^L,0]$ and define stopping times
\[ \sigma_1 = \inf\{t \geq 0 : |\eta(t)| = 1\}\quad\text{and}\quad \tau_0^L = \inf\{t \geq 0: \eta(t) \in (-\infty,x_L]\}.\]
Then we have that
\[ \p[\sigma_1 \leq \tau_0^L] \leq p_2.\]
\end{lemma}
\begin{proof}
See Figure \ref{fig::sle_kappa_rho_boundary_hitting}. Lemma~\ref{lem::sle_kappa_rho_hit_boundary_segment} implies that this event has probability strictly smaller than 1 for each fixed choice of $x_L,x_R$ as above.  Therefore the result follows from Lemma~\ref{lem::sle_kappa_rho_cont_in_force_points}.
\end{proof}

\begin{figure}[ht!]
\begin{center}
\includegraphics[scale=0.85]{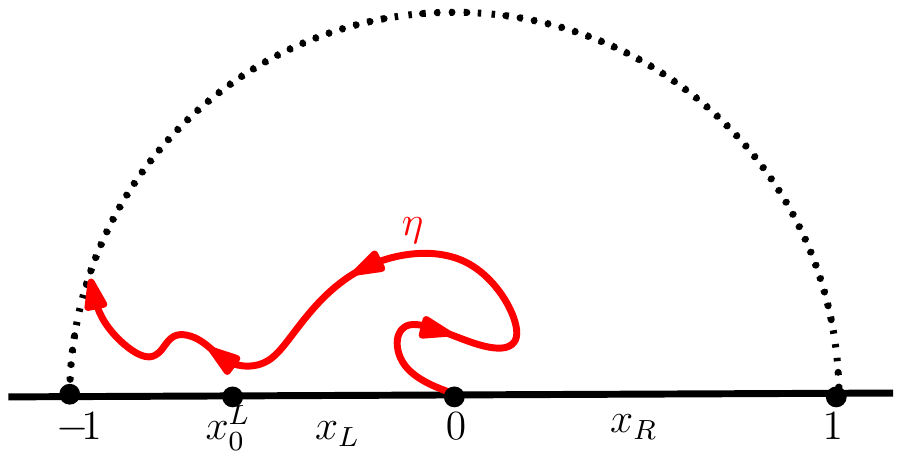}
\end{center}
\caption{\label{fig::sle_kappa_rho_boundary_hitting}
Suppose that $\eta$ is an $\SLE_\kappa(\rho_L;\rho_R)$ process in $\h$ starting from $0$ to $\infty$ with force points located at $x_L \leq 0 \leq x_R$ with $\rho_L \in (\tfrac{\kappa}{2}-4,\tfrac{\kappa}{2}-2)$ and $\rho_R > -2$.  We show in Lemma~\ref{lem::sle_kappa_rho_boundary_hitting} that for each choice of $x_0^L \in (-1,0)$ there exists $p_2 = p_2(x_0^L) \in [0,1)$ such that the probability that $\eta$ hits $\partial B(0,1)$ before hitting $(-\infty,x_L]$ is at most $p_2$ uniformly in $x_L \in [x_0^L,0]$.
}

\end{figure}

\subsection{Radon-Nikodym Derivative}
\label{subsec::radon_nikodym}

Following \cite[Lemma~13]{\DubedatDuality}, we will now describe the Radon-Nikodym derivative between $\SLE_\kappa(\ul{\rho})$ processes arising from a change of domains and the locations and weights of the force points.  Let $c=(D, z_0, \underline{x}_L, \underline{x}_R, z_{\infty})$ be a configuration consisting of a Jordan domain $D$ in $\C$ with $\ell+r+2$ marked points on $\partial D$. An $\SLE_{\kappa}(\ul{\rho}_L;\ul{\rho}_R)$ process $\eta$ with configuration $c$ is given by the image of an $\SLE_{\kappa}(\ul{\rho}_L;\ul{\rho}_R)$ process $\wt{\eta}$ in $\h$ under a conformal transformation $\varphi$ taking $\h$ to $D$ with $\varphi(0)=z_0$, $\varphi(\infty)=z_\infty$, and which takes the force points of $\wt{\eta}$ to those of $\eta$.

Suppose that $c=(D,z_0,\ul{x}_L,\ul{x}_R,z_{\infty})$ and $\wt{c}=(\wt{D},z_0,\wt{\ul{x}}_L,\wt{\ul{x}}_R,\wt{z}_{\infty})$ are two configurations such that $\wt{D}$ agrees with $D$ in a neighborhood $U$ of $z_0$.  Let $\mu_c^U$ denote the law of an $\SLE_\kappa(\ul{\rho}_L;\ul{\rho}_R)$ process in $c$ stopped at the first time $\tau$ that it exits $U$ and define $\mu_{\wt{c}}^U$ analogously.  Let
\[ \rho_{\infty}=\kappa-6-\sum_{i,q}\rho_{i,q}\]
and
\begin{equation}\label{eqn::total_mass}
\begin{split}
Z(c)&=\LH_D(z_0,z_{\infty})^{-\frac{\rho_{\infty}}{2\kappa}} \times \prod_{i,q} \LH_D(z_0,x_{i,q})^{-\frac{\rho_{i,q}}{2\kappa}} \\
&\times \prod_{(i,q) \neq (i',q')}\LH_D(x_{i,q},x_{i',q'})^{-\frac{\rho_{i,q}\rho_{i',q'}}{4\kappa}}\times \prod_{i,q}\LH_D(x_{i,q},z_{\infty})^{-\frac{\rho_{i,q}\rho_{\infty}}{4\kappa}}
\end{split}
\end{equation}
where $\LH_D$ is the Poisson excursion kernel of the domain $D$.  We also let
\begin{equation*}
\xi=\frac{(6-\kappa)(8-3\kappa)}{2\kappa},
\end{equation*}
\begin{equation*}
c_{\tau}=(D\setminus K_{\tau}, \eta(\tau),\ul{x}^\tau_L,\ul{x}^\tau_R,z_{\infty}),
\end{equation*}
\begin{equation*}
m(D;K,K')=\mu^{\rm loop}\left( \ell : \ell \subseteq D,\ \ell \cap K\neq \emptyset,\ \ell \cap K'\neq \emptyset \right),
\end{equation*}
where $K_{\tau}$ is the compact hull associated with $\eta([0,\tau])$ and $\mu^{\rm loop}$ the Brownian loop measure on unrooted loops in $\C$ (see \cite{LawlerWernerBrownianLoopsoup} for more on the Brownian loop measure).  Also, $x_{i,q}^\tau = x_{i,q}$ if $x_{i,q}$ is not swallowed by time $\tau$, otherwise $x^\tau_{i,L}$ (resp.\ $x^{\tau}_{i,R}$) is the leftmost (resp.\ rightmost) point of $\partial K_\tau\cap\partial D$ in the clockwise (resp.\ counterclockwise) arc on $\partial D$ from $z_0$ to $z_\infty$, .

The following result is proved in \cite[Lemma 13]{\DubedatDuality} in the case that $U$ is at a positive distance from the marked points of $c,\wt{c}$ other than $z_0$.  We are now going to use the $\SLE$/GFF coupling described in the previous section to extend the result to the case that $U$ is at a positive distance from the marked points of $c,\wt{c}$ which are different.

\begin{lemma}
\label{lem::change_of_domains}
Assume that we have the setup described just above.  Suppose that $U$ is at a positive distance from those marked points of $c,\wt{c}$ which differ.  The probability measures $\mu_{\wt{c}}^U$ and $\mu_c^U$ are mutually absolutely continuous and
\begin{equation}
\begin{split}
&\frac{d\mu^U_{\wt{c}}}{d\mu^U_{c}}(\eta) \\
=&\left(\frac{Z(\wt{c}_{\tau})/Z(\wt{c})}{Z(c_{\tau})/Z(c)}\right)\exp\big(-\xi m(D;K_{\tau}, D\setminus\wt{D})+\xi m(\wt{D} ; K_{\tau}, \wt{D}\setminus D) \big) \label{eqn::domain_change_rn}
\end{split}
\end{equation}
\end{lemma}
\begin{proof}
We are first going to prove the result in the case that $x_{1,L} \neq z_0 \neq x_{1,R}$.  We know that we can couple $\eta \sim \mu_c^U$ (resp.\ $\wt{\eta} \sim \mu_{\wt{c}}^U$) with a GFF $h$ (resp.\ $\wt{h}$) on $D$ (resp.\ $\wt{D}$) so that $\eta$ (resp.\ $\wt{\eta}$) is the flow line of $h$ (resp.\ $\wt{h}$) starting from $z_0$.  By our hypotheses, the boundary data of $h$ and $\wt{h}$ agree with each other in the boundary segments which are also contained in $\partial U$.  Consequently, the laws of $h|_{U}$ and $\wt{h}|_{U}$ are mutually absolutely continuous \cite[Proposition~3.2]{IG1}.  Since $\eta$ (resp.\ $\wt{\eta}$) is almost surely determined by $h$ (resp.\ $\wt{h}$) \cite[Theorem~1.2]{IG1}, it follows that $\mu_c^U$ and $\mu_{\wt{c}}^U$ are mutually absolutely continuous.  Thus, to complete the proof, we just need to identify $f(\eta) := (d\mu_{\wt{c}}^U/ d\mu_c^U)(\eta)$.  By \cite[Lemma 13]{\DubedatDuality}, we know that $f(\eta)$ is equal to the right side of \eqref{eqn::domain_change_rn} for paths $\eta$ which intersect the boundary only in the counterclockwise segment of $\partial D$ from $x_{1,L}$ to $x_{1,R}$ (and this only happens for $\kappa > 4$).  Therefore, to complete the proof, we need to show that the same equality holds for paths $\eta$ which intersect the other parts of the domain boundary.  Note that the right hand side of \eqref{eqn::domain_change_rn} is a continuous function of $\eta$ with respect to the uniform topology on paths.  Therefore, to complete the proof, it suffices to show that the Radon-Nikodym derivative $f(\eta)$ is also continuous with respect to the same topology.  Indeed, then the result follows since both functions are continuous and agree with each other on a dense set of paths.  We are going to prove that this is the case using that $\eta,\wt{\eta}$ are coupled with $h,\wt{h}$, respectively.

Let $\nu_c^U$ (resp.\ $\nu_{\wt{c}}^U$) denote the joint law of $(\eta,h|_U)$ (resp.\ $(\wt{\eta},\wt{h}|_U)$).  As explained above, $\nu_c^U$ and $\nu_{\wt{c}}^U$ are mutually absolutely continuous.  Moreover, the Radon-Nikodym derivative $d\nu_{\wt{c}}^U / d\nu_c^U$
is a function of $h$ alone since $h,\wt{h}$ almost surely determine $\eta,\wt{\eta}$, respectively.  Let $\nu_c^U(\cdot \giv \cdot)$ (resp.\ $\nu_{\wt{c}}^U(\cdot \giv \cdot)$) denote the conditional law of $h|_U$ given $\eta$ (resp.\ $\wt{h}|_U$ given $\wt{\eta}$).  Note that
\[ \eta \mapsto \frac{d\nu_{\wt{c}}^U(\cdot \giv \eta)}{d\nu_{c}^U(\cdot \giv \eta)}\]
is continuous in $\eta$ with respect to the uniform topology on continuous paths.  Let $\nu_{c,h}^U(\cdot)$ (resp.\ $\nu_{\wt{c},h}^U(\cdot )$) denote the law of $h|_U$ (resp.\ $\wt{h}|_U$).  Then we have that
\begin{align*}
 \frac{d\nu_{\wt{c},h}^U}{d\nu_{c,h}^U}(\cdot) = \frac{d \nu_{\wt{c}}^U}{d\nu_{c}^U}(\eta,\cdot) = \frac{d\nu_{\wt{c}}^U}{d\nu_{c}^U}(\cdot \giv \eta) \times \frac{d\mu_{\wt{c}}^U}{d\mu_{c}^U}(\eta) = \frac{d\nu_{\wt{c}}^U}{d\nu_{c}^U}(\cdot \giv \eta) \times f(\eta).
\end{align*}
Rearranging, we see that
\[ f(\eta) = \frac{d \nu_{\wt{c},h}^U(\cdot)}{d \nu_{c,h}^U(\cdot)} \times  \frac{d\nu_{c}^U(\cdot \giv \eta)}{d\nu_{\wt{c}}^U(\cdot \giv \eta)}\]
(the right side does not depend on the choice of $\cdot$ since the left side does not depend on $\cdot$).
This implies the desired result in the case that $x_{1,L} \neq z_0 \neq x_{1,R}$ since the latter factor on the right side is continuous in $\eta$, as we remarked above.  The result follows in the case that one or both of $x_{1,L},x_{1,R}$ agrees with $z_0$ since the laws converge as one or both of $x_{1,L},x_{1,R}$ converge to $z_0$ (Lemma~\ref{lem::sle_kappa_rho_cont_in_force_points}).
\end{proof}

\begin{lemma}
\label{lem::change_of_domains_bound_RN}
Assume that we have the same setup as in Lemma \ref{lem::change_of_domains} with $D=\h$, $\wt{D}\subseteq \h$, $U \subseteq \h$ bounded, and $z_0=0$. Fix $\zeta>0$ and suppose that the distance between $U$ and $\h\setminus \wt{D}$ is at least $\zeta$, the force points of $c,\wt{c}$ in $\ol{U}$ are identical, the corresponding weights are also equal, and the force points which are outside of $U$ are at distance at least $\zeta$ from $U$.  There exists a constant $C \geq 1$ depending on $U$, $\zeta$, $\kappa$, and the weights of the force points such that
\[\frac{1}{C} \leq \frac{d\mu^U_{\wt{c}}}{d\mu^U_c} \leq C.\]
\end{lemma}
\begin{proof}
Note that $0\le m(\h; K_{\tau},\h\setminus \wt{D})\le m(\h; U,\h\setminus U^{\zeta})$ where $U^\zeta$ is the $\zeta$-neighborhood of $U$. Moreover, we have that $m(\h; U,\h\setminus U^{\zeta})$ is bounded from above by a finite constant depending on $U$ and $\zeta$ since the mass according to $\mu^{\rm loop}$ of the loops which are contained in $\h$, intersect $U$, and have diameter at least $\zeta$ is finite \cite[Corollary~4.6]{LawlerBrownianLoopMeasure}.  Consequently, by Lemma~\ref{lem::change_of_domains}, we only need to bound the quantity $\tfrac{Z(\wt{c}_{\tau})/Z(\wt{c})}{Z(c_{\tau})/Z(c)}$.

Recall from \eqref{eqn::total_mass} that the terms in $\tfrac{Z(\wt{c}_{\tau})/Z(\wt{c})}{Z(c_{\tau})/Z(c)}$ are ratios of terms of the form $\LH_X(u,v)$ where $X$ is one of $\h$, $\h_\tau$, $\wt{D}$, $\wt{D}_\tau$ and $u,v$ are two marked points on the boundary of $X$. 
We will complete the proof by considering several cases depending on the location of the marked points.

\smallskip
\noindent{\bf Case 1.}  At least one marked point is outside of $U^\zeta$.  This is the case handled in the proof of \cite[Lemma~14]{\DubedatDuality}.
\smallskip

\noindent{\bf Case 2.} Both marked points $u,v$ are contained in $\overline{U}$ and $u\neq v$. It is enough to bound from above and below the ratios:
\[ A = \frac{\LH_{\wt{D}}(x,y)}{\LH_{\h}(x,y)}\ \quad\text{and}\quad B = \frac{\LH_{\wt{D}_\tau}(x^\tau,y^\tau)}{\LH_{\h_\tau}(x^\tau,y^\tau)}\]
where $x,y \in \partial U\cap\R$ are distinct and $x^\tau,y^\tau \in \partial\h_\tau\cap\overline{U}$ are distinct. 

We can bound $A$ as follows.  Let $\varphi \colon \wt{D} \to \h$ be the unique conformal transformation with $\varphi(x)=x$, $\varphi(y)=y$, and $\varphi'(x)=1$. Then $A =|\varphi'(y)|$ which, by \cite[Proposition~4.1]{LawlerSchrammWernerConformalRestrictionChordal}, is equal to the mass of those Brownian excursions in $\h$ connecting $x$ and $y$ which avoid $\h \setminus\wt{D}$.  We will write $q(\h,x,y,\h\setminus\wt{D})$ for this quantity.  Since this is given by a probability, we have that $|\varphi'(y)| \leq 1$ and it follows that $|\varphi'(y)|$ is bounded from below by $q(\h, x, y, U^{\zeta}) > 0$. This lower bound is a positive continuous function in $x,y \in \partial U\cap \partial \h$ hence yields a uniform lower bound.  Consequently, $A$ is bounded from both above and below.

Similarly, $B$ is equal to the mass $q(\h \setminus K_\tau,x^\tau,y^\tau,\h \setminus\wt{D})$
of those Brownian excursions in $\h \setminus K_\tau$ which connect $x^\tau$ and
$y^\tau$ and avoid $\h \setminus\wt{D}$. As before, this quantity is bounded
from above by $1$.  We will now establish the lower bound. Let $g$ be the conformal map from $\h\setminus K_\tau$ onto $\h$ which sends the triple $(x^\tau,y^\tau,\infty)$ to $(0,1,\infty)$. Note that $g$ can be extended to $\C\setminus (K_\tau \cup \bar{K}_\tau)$ by Schwarz reflection where $\bar{K}_\tau=\{z\in\C: \bar{z}\in K_\tau\}$. We will view $g$ as such an extension. Then it is clear that
\begin{align*}
q(\h \setminus K_\tau,x^\tau,y^\tau,\h \setminus\wt{D})&\ge q(\h \setminus K_\tau,x^\tau,y^\tau,\h \setminus U^{\zeta})\\
&=q(\h, 0,1,\h\setminus g(U^{\zeta})).
 \end{align*}
Note that $q(\h, 0,1,\h\setminus g(U^{\zeta}))$ is a continuous functional on compact hulls $K$ inside $\overline{U}$ equipped with the Hausdorff metric.  Indeed, suppose that $(K_n)$ is a sequence of compact hulls inside $\overline{U}$ converging towards $K$ in the Hausdorff metric and, for each $n$, let $g_n$ be the corresponding conformal map. Then $g_n$ converges to $g$ uniformly away from $K\cup\bar{K}$. In particular, $g_n(U^{\zeta})$ converges to $g(U^{\zeta})$ in Hausdorff metric. Let $\phi_n$ (resp.\ $\phi$) be the conformal map from $\h\setminus g_n(U^\zeta)$ (resp.\ $\h\setminus g(U^\zeta)$) onto $\h$ which fixes $0$, $1$ and has derivative $1$ at $1$.  Then $\phi_n'(0)$ converges to $\phi'(0)$. Thus $q(\h, 0,1,\h\setminus g_n(U^{\zeta}))=\phi_n'(0)$ converges to $q(\h, 0,1,\h\setminus g(U^{\zeta}))=\phi'(0)$ which explains the continuity of $q(\h,0,1,\h\setminus g(U^{\zeta}))$ in $K$. Since the set of compact hulls inside $\overline{U}$ endowed with Hausdorff metric is compact, there exists $q_0>0$ depending only on $U$ and $\zeta$ such that 
\[ q(\h \setminus K_\tau,x^\tau,y^\tau,\h \setminus\wt{D})\ge q(\h, 0,1,\h\setminus g(U^{\zeta}))\ge q_0.\]
\smallskip

\noindent{\bf Case 3.}  A single marked point $u$ contained in $\overline{U}$.  The ratios which involve terms of the form $\LH_X(u,u)$ are interpretted using limits hence are uniformly bounded by the argument of Case 2.
\smallskip

\end{proof}

\subsection{Estimates for conformal maps}
\label{subsec::distort}

For a proper simply connected domain $D$ and $w\in D$, let $\confrad(w; D)$ denote the conformal radius of $D$ with respect to~$w$, i.e., $\confrad(w; D)\equiv f'(0)$ for $f$ the unique conformal map $\D\to D$ with $f(0)=w$ and $f'(0)>0$. Let $\rad(w;D)\equiv\inf\set{r:\ball{r}{w}\supseteq D}$ denote the out-radius of $D$ with respect to $w$. By the Schwarz lemma and the Koebe one-quarter theorem,
\begin{equation}
\label{eqn::ir.cr.or}
\dist(w,\pd D)
\le\confrad(w;D)
\le[4\,\dist(w,\pd D)] \wedge \rad(w;D).
\end{equation}
Further (see e.g.\ \cite[Theorem~1.3]{\Pommerenke})
\begin{align}
\label{eqn::distortion}
\frac{\abs{\ze}}{(1+\abs{\ze})^2}
&\leq \frac{\abs{f(\ze)-w}}{\confrad(w;D)}
\leq \frac{\abs{\ze}}{(1-\abs{\ze})^2}
\end{align}
As a consequence,
\begin{equation}
\label{eqn::growth}
\frac{\abs{\ze}}{4}
\le \frac{\abs{f(\ze)-w}}{\confrad(w;D)}
\le4\abs{\ze}
\end{equation}
where the right-hand inequality above holds for $\abs{\ze}\le1/2$.

Finally, we state the Beurling estimate \cite[Theorem~3.76]{lawler2005} which we will frequently use in conjunction with the conformal invariance of Brownian motion.

\begin{theorem}[Beurling Estimate]
\label{thm::beurling}
Suppose that $B$ is a Brownian motion in $\C$ and $\tau_\D =\inf\{t \geq 0: B(t) \in \partial \D\}$.  There exists a constant $c < \infty$ such that if $\gamma \colon [0,1] \to \C$ is a curve with $\gamma(0) = 0$ and $|\gamma(1)| = 1$, $z \in \D$, and $\p^z$ is the law of $B$ when started at $z$, then
\[ \p^z[ B([0,\tau_\D]) \cap \gamma([0,1]) = \emptyset] \leq c |z|^{1/2}.\]
\end{theorem}

\section{The intersection of $\SLE_\kappa(\rho)$ with the boundary}
\label{sec::sle_kappa_rho_boundary}

\subsection{The upper bound}
\label{subsec::upper_bound}

The main result of this section is the following theorem, which in turn implies Theorem~\ref{thm::one_point}.

\begin{theorem}
\label{thm::one_point_exponent_general}
Fix $\kappa >0$, $\rho_{1,R} > -2$, and $\rho_{2,R} \in \R$ such that $\rho_{1,R} + \rho_{2,R} > \tfrac{\kappa}{2}-4$.  Fix $x_R \in [0^+,1)$ and let $\eta$ be an $\SLE_{\kappa}(\rho_{1,R},\rho_{2,R})$ process with force points $(x_R,1)$.  Let
\begin{equation}
\label{eqn::alpha_definition}
\alpha=\frac{1}{\kappa}(\rho_{1,R}+2)\left(\rho_{1,R}+\rho_{2,R}+4-\frac{\kappa}{2}\right).
\end{equation}
For each $\epsilon > 0$, let $\tau_\epsilon = \inf\{t \geq 0 : \eta(t) \in \partial B(1,\epsilon)\}$ and, for each $r > 0$, let $\sigma_r = \inf\{t \geq 0 : \eta(t) \in \partial (r \D)\}$.  For each $\delta \in [0,1)$ and $r \geq 2$ fixed, let
\begin{equation}
\label{eqn::e_definition}
E_\eps^{\delta,r}=\{\tau_\eps<\sigma_r,\ \im(\eta(\tau_\eps))\ge \delta\eps\}.
\end{equation}
We have that
\begin{equation}
\label{eqn::onepointestimate}
 \p[E_\eps^{\delta,r}]=\eps^{\alpha+o(1)}\quad\text{as}\quad \epsilon \to 0.
\end{equation}
\end{theorem}

The $o(1)$ in the exponent of \eqref{eqn::onepointestimate} tends to $0$ as $\epsilon \to 0$ and depends only on $\kappa$, $\delta$, $x_R$, and the weights $\rho_{1,R}$, $\rho_{2,R}$.  The $o(1)$, however, is uniform in $r \geq 2$.  Taking $\rho_{1,R} > (-2) \vee (\tfrac{\kappa}{2}-4)$ and $\rho_{2,R} = 0$, we have that
\begin{equation}
\label{eqn::onepoint_constant_rho}
\alpha=\frac{1}{\kappa}(\rho+2)\left(\rho+4-\frac{\kappa}{2}\right).
\end{equation}
Thus Theorem~\ref{thm::one_point_exponent_general} leads to the upper bound of Theorem~\ref{thm::boundary_dimension}.  We begin with the following lemma which contains the same statement as Theorem~\ref{thm::one_point_exponent_general} except is restricted to the case that $\delta \in (0,1)$ and, in particular, is not applicable for $\delta=0$.

\begin{lemma}
\label{lem::lower_bound_one_point_exponent}
Assume that we have the same setup and notation as in Theorem~\ref{thm::one_point_exponent_general}. Then for each $\delta\in(0,1)$ and $r \geq 2$ fixed, we have that
\[ \p[E_\eps^{\delta,r}]\asymp \eps^{\alpha}\]
where the constants in $\asymp$ depend only on $\kappa$, $\delta$, $x_R$, and the weights $\rho_{1,R}$,~$\rho_{2,R}$.
\end{lemma}
\begin{proof}
For $\eta$, the $\SLE_{\kappa}(\rho_{1,R},\rho_{2,R})$ process with force points $(x_R,1)$, let $(g_t)$ be the associated Loewner evolution and let $V_t^R$ denote the evolution of $x_R$. From \eqref{eqn::martingalebetweensles} we know that
\[M_t=\left(\frac{g_t(1)-V^R_t}{g'_t(1)}\right)^{-\alpha}\left(\frac{g_t(1)-W_t}{g_t(1)-V^R_t}\right)^{-\frac{2}{\kappa}(\rho_{1,R}+\rho_{2,R}+4-\kappa/2)}\]
is a local martingale and the law of $\eta$ reweighted by $M$ is that of an $\SLE_{\kappa}(\rho_{1,R},\wt{\rho}_{2,R})$ process where $\wt{\rho}_{2,R}= -2\rho_{1,R}-\rho_{2,R}-8+\kappa$.
We write $K=K_{\tau_\eps}$ and $\ol{K}=\{\ol{z}: z\in K\}$. Let $G$ be the extension of $g_{\tau_\eps}$ to $\C\setminus (K\cup\ol{K})$ which is obtained by Schwarz reflection. By \eqref{eqn::ir.cr.or}, we have
\begin{equation}
\label{eqn::distance}
G'(x)\dist(x,K)\asymp \dist(G(x), G(K\cup\ol{K})).
\end{equation}
Observe that $G(K \cup \ol{K}) = [O_{\tau_\epsilon}^L,O_{\tau_{\epsilon}}^R]$ where $O_t^L$ (resp.\ $O_t^R$) is the image of the leftmost (resp.\ rightmost) point of $K_t \cap \R$ under $g_t$.  Note that \eqref{eqn::distance} implies
\[ \eps g'_{\tau_\eps}(1)\asymp g_{\tau_\eps}(1)-O_{\tau_\eps}^R.\]

It is clear that $g_t(1)-W_t\ge g_t(1)-O_t^R \ge g_t(1)-V_t^R$.  On the event $E_{\eps}^{\delta,r}$, we run a Brownian motion started from the midpoint of the line segment $[1,\eta(\tau_{\eps})]$.  Then this Brownian motion has uniformly positive (though $\delta$-dependent) probability to exit $\h\setminus K$ through each of the left side of $K$, the right side of $K$, the interval $[x_R,1]$, and the interval $(1,\infty)$.  Consequently, by the conformal invariance of Brownian motion,
\[ g_{\tau_\eps}(1)-W_{\tau_\eps}\asymp g_{\tau_\eps}(1)-O_{\tau_\eps}^R \asymp g_{\tau_\eps}(1)-V_{\tau_\eps}^R \quad \text{on}\quad E_{\eps}^{\delta, r}.\]

These facts imply that $M_{\tau_{\eps}}\asymp \eps^{-\alpha}$ on $E_{\eps}^{\delta,r}$ where the constants in $\asymp$ depend only on $\kappa$, $\delta$, $x_R$, and the weights $\rho_{1,R}$, $\rho_{2,R}$.  Thus
\[\p[E_{\eps}^{\delta,r}]\asymp \eps^{\alpha}\E[M_{\tau_{\eps}}\one_{E_{\eps}^{\delta,r}}]=\eps^{\alpha}\p^\star[E_{\eps}^{\delta,r}]\]
where $\p^\star$ is the law of $\eta$ weighted by the martingale $M$. As we remarked earlier, $\p^\star$ is the law of an $\SLE_{\kappa}(\rho_{1,R}, \wt{\rho}_{2,R})$ with force points $(x_R,1)$.

\begin{figure}[ht!]
\begin{center}
\includegraphics[scale=0.85]{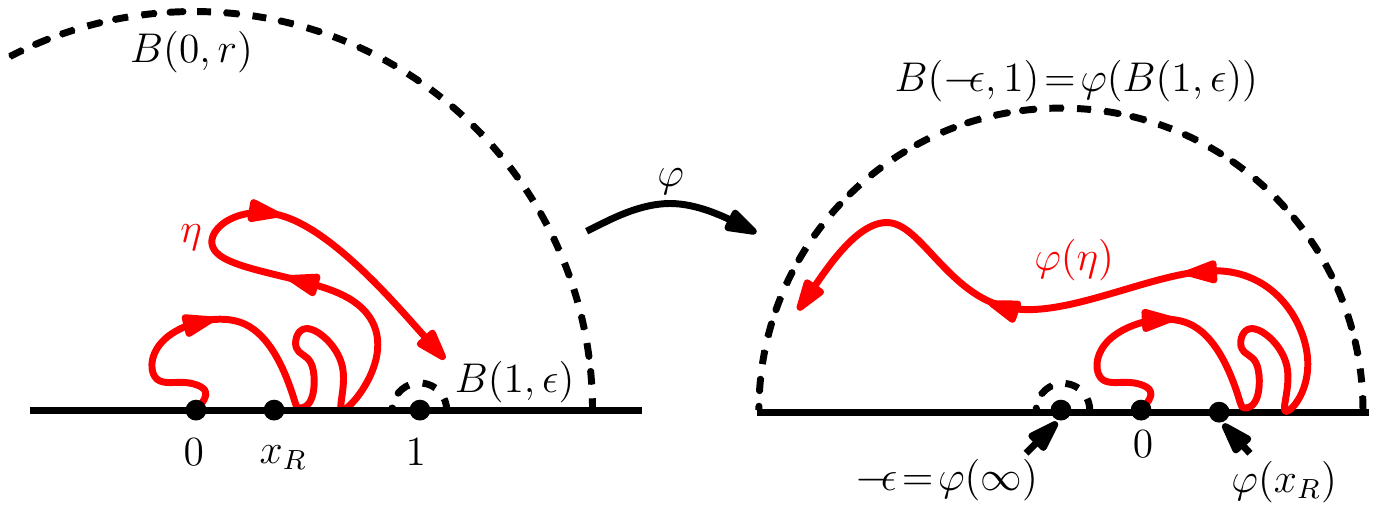}
\end{center}
\caption{\label{fig::change_coordinates} The image of an $\SLE_{\kappa}(\rho_{1,R},\rho_{2,R})$ process in $\h$ from $0$ to $\infty$ with force points $(x_R,1)$ under $\varphi(z)=\eps z/(1-z)$ has the same law as an $\SLE_{\kappa}(\rho_L;\rho_R)$ process in $\h$ from $0$ to $\infty$ with force points $(-\eps; \eps x_R/(1-x_R))$ where $\rho_R=\rho_{1,R}$ and $\rho_L=\kappa-6-(\rho_{1,R}+\rho_{2,R})$.}
\end{figure}

We now perform a coordinate change using the M\"obius transformation $\varphi(z)=\eps z/(1-z)$.  Then the law of the image of a path distributed according to $\p^\star$ under $\varphi$ is equal to that of an $\SLE_{\kappa}(2+\rho_{1,R}+\rho_{2,R}; \rho_{1,R})$ process in $\h$ from $0$ to $\infty$ with force points $(-\eps;\eps x_R/(1-x_R))$ (see Figure~\ref{fig::change_coordinates}). Note that $2+\rho_{1,R}+\rho_{2,R} \geq \tfrac{\kappa}{2}-2$ by the hypotheses of the lemma.  Let $\eta^\star$ be an $\SLE_{\kappa}(2+\rho_{1,R}+\rho_{2,R}; \rho_{1,R})$ process in $\h$ from $0$ to $\infty$ with force points $(-\eps;\eps x_R/(1-x_R))$.   In particular, by Lemma~\ref{lem::sle_kappa_rho_boundary_interaction}, $\eta^\star$ almost surely does not hit $(-\infty,-\eps)$.  Under the coordinate change, the event $E_{\eps}^{\delta,r}$ becomes $\{ \sigma^\star_{1,\eps}< \xi^\star_{\eps,r},\ \im(\eta^\star(\sigma^\star_{1,\eps}))\ge \delta\}$ where $\sigma^\star_{1,\eps}$ is the first time that $\eta^\star$ hits $\partial B(-\eps,1)$, $\xi^\star_{\eps,r}$ is the first time that $\eta^\star$ hits $\partial B(-\eps r^2 / (r^2-1),\eps r/(r^2-1))$. By Lemma~\ref{lem::sle_kappa_rho_exit_disk}, the probability of the event $\{ \sigma^\star_{1,\eps}< \xi^\star_{\eps,r},\ \im(\eta^\star(\sigma^\star_{1,\eps}))\ge \delta\}$ is bounded from below by a positive constant depending only on $\kappa$, $\delta$, $\rho_{1,R}$, and $\rho_{2,R}$. Thus $\p^\star[E_{\eps}^{\delta,r}]\asymp 1$ which implies $\p[E_{\eps}^{\delta,r}]\asymp \eps^{\alpha}$ and the constants in $\asymp$ depend only on $\kappa$, $\delta$, $x_R$, and the weights $\rho_{1,R}$, $\rho_{2,R}$.
\end{proof}

\begin{corollary}
\label{cor::lower_bound_one_point_exponent}
Fix $\kappa>0$, $\rho_L>-2,\rho_{1,R}>-2$ and $\rho_{2,R}\in\R$ such that $\rho_{1,R} + \rho_{2,R} > \tfrac{\kappa}{2}-4$.  Fix $x_L\le 0, x_R \in [0^+,1)$ and let $\eta$ be an $\SLE_{\kappa}(\rho_L;\rho_{1,R},\rho_{2,R})$ process with force points $(x_L;x_R,1)$.  Let $E_{\epsilon}^{\delta,r}$ be the event as in Theorem \ref{thm::one_point_exponent_general}, then for each $\delta\in(0,1)$ and $r \geq 2$ fixed, we have that
\[ \p[E_\eps^{\delta,r}]\asymp \eps^{\alpha}\]
where the constants in $\asymp$ depend only on $\kappa$, $\delta$, $r$, $x_L$, $x_R$, and the weights $\rho_L, \rho_{1,R}$,~$\rho_{2,R}$.
\end{corollary}
\begin{proof}
Let $(g_t)$ be the Loewner evolution associated with $\eta$ and let $V_t^L,V_t^R$ denote the evolution of $x_L,x_R$, respectively, under $g_t$. From \eqref{eqn::martingalebetweensles} we know that
\begin{align*}M_t=&\left(\frac{g_t(1)-V^R_t}{g'_t(1)}\right)^{-\alpha}\times\left(\frac{g_t(1)-W_t}{g_t(1)-V^R_t}\right)^{-\frac{2}{\kappa}(\rho_{1,R}+\rho_{2,R}+4-\kappa/2)}\\
&\times(g_t(1)-V_t^L)^{-\frac{\rho_L}{\kappa}(\rho_{1,R}+\rho_{2,R}+4-\kappa/2)}\end{align*}
is a local martingale which yields that the law of $\eta$ reweighted by $M$ is that of an $\SLE_{\kappa}(\rho_L;\rho_{1,R},\wt{\rho}_{2,R})$ process where $\wt{\rho}_{2,R}= -2\rho_{1,R}-\rho_{2,R}-8+\kappa$. Note that, by similar analysis in Lemma \ref{lem::tubeestimate}, the term $g_{\tau_\eps}(1)-V_{\tau_\eps}^L$ is bounded both from below and above by positive finite constants depending only on $r$ on the event $E_{\eps}^{\delta,r}$. The rest of the analysis in the proof of Lemma \ref{lem::lower_bound_one_point_exponent} applies similarly in this setting.
\end{proof}

Throughout the rest of this subsection, we let:
\begin{equation}
\label{eqn::t_definition}
\T =  \R \times (0,1).
\end{equation}

\begin{lemma}
\label{lem::tubeestimate}
Let $\eta$ be a continuous curve in $\ol{\h}$ starting from $0$ with continuous Loewner driving function $W$ and let $(g_t)$ be the corresponding family of conformal maps.  For each $t \geq 0$, let $O_t^L$ (resp.\ $O_t^R$) be the leftmost (resp.\ rightmost) point of $g_t(\eta([0,t]))$ in $\R$.  There exists a universal constant $C \geq 1$ such that the following is true.  Fix $\vartheta > 0$ and let $\sigma$ be the first time that $\eta$ exits $\vartheta \T$.  Then
\begin{equation}
\label{eqn::tube_lower}
 |W_{\sigma}-O_{\sigma}^q| \geq \frac{\vartheta}{C}\quad \text{for} \quad q\in \{L,R\}.
\end{equation}
Let $\zeta$ be the first time that $\eta$ exits $\D \cap\vartheta\T$.  Then
\begin{equation}
\label{eqn::tube_upper}
  |W_t - O_t^q| \leq C \vartheta \quad \text{for}\quad q\in \{L,R\} \quad\text{and all}\quad t \in [0,\zeta].
\end{equation}
Finally, if $\eta$ exits $\D \cap\vartheta\T$ through the right side of $\partial \D \cap \vartheta \T$, then
\begin{equation}
\label{eqn::tube_left_lower}
|W_{\zeta}-O_{\zeta}^L|\ge \frac{1}{C}.
\end{equation}
\end{lemma}
\begin{proof}
For $z \in \C$, we let $\p^{z}$ denote the law of a Brownian motion $B$ in $\C$ started at $z$.  By \cite[Remark~3.50]{lawler2005} we have that
\[ |W_{\sigma}-O_{\sigma}^L|=\lim_{y\to\infty}y\p^{yi}\left[ B\mbox{ exits }\HH\setminus\eta[0,\sigma]\mbox{ on the left side of } \eta([0,\sigma]) \right].\]
Let $\tau$ be the exit time of $B$ from $\h \setminus \vartheta \T$ and let $I=[\eta(\sigma)-\vartheta,\eta(\sigma)]$.  Then
\begin{align}
|W_{\sigma}-O_{\sigma}^L| \geq& \lim_{y\to\infty}y\p^{yi}\left[ B_\tau \in I \right] \notag\\
&\times \p^{yi}\left[ B \mbox{ exits } \HH\setminus\eta([0,\sigma]) \mbox{ on the left side of }\eta([0,\sigma]) \giv B_{\tau}\in I\right]. \label{eqn::left_side_lbd}
\end{align}
We have,
\begin{align}
    \lim_{y \to\infty} y\p^{yi}\left[ B_\tau \in I \right]
 =&\lim_{y \to \infty} \int_{I-\vartheta i}  \frac{1}{\pi} \frac{y(y-\vartheta)}{w^2+(y-\vartheta)^2} dw \notag \\
 =&\int_{I-\vartheta i}  \frac{1}{\pi} dw =\frac{\vartheta}{\pi} \label{eqn::bm_hit_interval}
\end{align}
(recall the form of the Poisson kernel on $\h$, see e.g.\ \cite[Exercise~2.23]{lawler2005}).  It is easy to see that there exists a universal constant $p_0 > 0$ such that for any $z \in I$,
\begin{equation}
\label{eqn::bm_hit_interval_exit_left_side}
\p^z\left[ B \text{ exits } \HH\setminus\eta[0,\sigma] \text{ on the left side of }\eta([0,\sigma]) \right] \geq p_0.
\end{equation}
Combining \eqref{eqn::left_side_lbd} with \eqref{eqn::bm_hit_interval} and \eqref{eqn::bm_hit_interval_exit_left_side} gives \eqref{eqn::tube_lower}.  The bounds \eqref{eqn::tube_upper} and \eqref{eqn::tube_left_lower} are proved similarly.
\end{proof}

\begin{lemma}
\label{lem::tubelemma}
Fix $\kappa>0$, $\rho_L \in (\tfrac{\kappa}{2}-4,\tfrac{\kappa}{2}-2)$, and $\rho_R>-2$. Let $\eta$ be an $\SLE_\kappa(\rho_L;\rho_R)$ process with force points $(-\eps;x_R)$ for $x_R \geq 0^+$ and $\epsilon > 0$. Let $\sigma_1=\inf\{t\ge 0: \eta(t)\in\partial \D\}$. Define, for $u\ge 0$, $T_u^L=\inf\{t\ge 0: W_t-V_t^L=u\}$, where $V_t^L$ denotes the evolution of $x^L$. Let $p_2 = p_2(\tfrac{1}{2})$ be the constant from Lemma~\ref{lem::sle_kappa_rho_boundary_hitting}.  There exists constants $\epsilon_0 > 0$, $\vartheta_0>0$, and $C > 0$ such that for all $\eps \in (0,\eps_0)$ and $\vartheta\in(0,\vartheta_0)$ we have
\[ \p[\sigma_1<T_0^L \wedge T^L_{\vartheta}]\le p_2^{1/(C \vartheta)}.\]
\end{lemma}

\begin{proof}
Let $E_{\vartheta}=\{\sigma_1<T_0^L \wedge T^L_{\vartheta}\}$.  By definition, we have that
\begin{equation}
\label{eqn::w_v_bound2}
|W_t-V_t^L|<\vartheta\quad\text{for all}\quad t \in [0,\sigma_1]\quad\text{on}\quad E_\vartheta.
\end{equation}
By \eqref{eqn::tube_lower} of Lemma~\ref{lem::tubeestimate} there exists a constant $C_1 > 0$ such that $\eta([0,\sigma_1]) \subseteq C_1 \vartheta \T$.  Moreover, $\eta$ exits $\D\cap \big( C_1 \vartheta \T \big)$ on its left side for all $\vartheta > 0$ small enough because a Brownian motion argument (analogous to \eqref{eqn::tube_left_lower}) implies there exists a constant $C_2 > 0$ such that $|W_{\sigma_1}-V^L_{\sigma_1}| \geq C_2$ on the event that $\eta$ exits through the right side, contradicting \eqref{eqn::w_v_bound2}.

Suppose $C > 0$; we will set its value later in the proof.  For each $1 \leq k \leq \tfrac{1}{C \vartheta}$, we let
\[ L_k=\{z\in\HH: \re(z)=-kC\vartheta\}\quad \text{ and } \quad \zeta_k = \inf\{t \geq 0 : \eta(t) \in L_k\}.\]
On $E_\vartheta$, we have that $\zeta_1<\zeta_2< \cdots <\sigma_1<T_0^L$.  For each $k$, let $F_k = \{\zeta_k < T_{\vartheta}^L\}$ and let $\CF_k$ be the $\sigma$-algebra generated by $\eta|_{[0,\zeta_k]}$.  To complete the proof, we will show that
\[ \p[ \zeta_{k+1} < T_0^L \giv \CF_k] \one_{F_k} \leq p_2 \one_{F_k} \quad \text{ for each } \quad 1\le k \leq \frac{1}{C\vartheta}\]
where $p_2=p_2(\tfrac{1}{2})$ is the constant from Lemma~\ref{lem::sle_kappa_rho_boundary_hitting}.  To see this, we just need to show that $g_{\zeta_k}(\eta|_{[\zeta_k,\zeta_{k+1}]})$ satisfies the hypotheses of Lemma~\ref{lem::sle_kappa_rho_boundary_hitting} and that with
\[ \wt{L}_{k+1} = \frac{g_{\zeta_k}(L_{k+1}) - W_{\zeta_k}}{W_{\zeta_k} - V_{\zeta_k}^L}\]
we have that $\wt{L}_{k+1} \cap 2\D = \emptyset$ on $F_k$.

Therefore it suffices to prove
\begin{equation}
\label{eqn::dist_ratio_bound}
 \frac{\dist(W_{\zeta_k}, g_{\zeta_k}(L_{k+1}))}{W_{\zeta_k}-V_{\zeta_k}^L} \to \infty\quad\text{on}\quad F_k\quad\text{as}\quad C \to \infty.
\end{equation}
Let $B$ be a Brownian motion starting from $z_k^\vartheta = \eta(\zeta_k)-\vartheta$ and let $H_{k+1} = \{z \in \h : \re(z) \geq -(k+1)C \vartheta\}$ be the subset of $\HH$ which is to the right of $L_{k+1}$ (see Figure~\ref{fig::tube_lemma}). The probability that $B$ exits $H_{k+1} \setminus \eta([0,\zeta_k])$ through the right side of $\eta([0,\zeta_k])$ (blue) is $\gtrsim 1$, through $(-(k+1)C\vartheta, -kC\vartheta)$ (green) is $\gtrsim 1$, and through $L_{k+1}$ (orange) is $\lesssim 1/C$ (since this probability is less than the probability that the Brownian motion exits $\{z\in\C: -(k+1)C\vartheta<\re(z)<-kC\vartheta\}$ through $L_{k+1}$ which is less than $1/C$). Let
\[ \wt{z}_k^\vartheta \equiv \wt{x}_k^\vartheta + \wt{y}_k^\vartheta i \equiv \frac{g_{\zeta_k}(z_k^\vartheta) - W_{\zeta_k}}{W_{\zeta_k} - V_{\zeta_k}^L}\quad.\]

\begin{figure}[ht!]
\begin{center}
\includegraphics[scale=0.85]{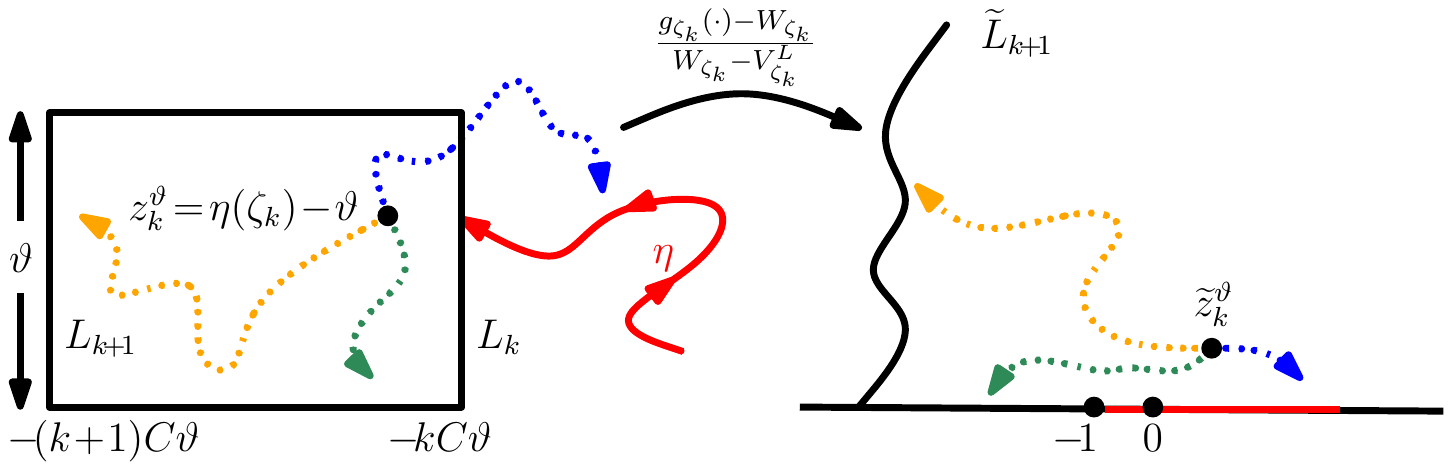}
\end{center}
\caption{\label{fig::tube_lemma}
Illustration of the justification of \eqref{eqn::dist_ratio_bound} in the proof of Lemma~\ref{lem::tubelemma}.}
\end{figure}

By the conformal invariance of Brownian motion, we have that 
\begin{equation}
\label{eqn::dist_ratio_bound2}
\frac{\dist(\wt{z}_k^\vartheta,\wt{L}_{k+1})}{\wt{y}_k^\vartheta} \gtrsim C.
\end{equation}
Indeed, the probability of a Brownian motion started from $\wt{z}_k^\vartheta$ to exit $\wt{H}_{k+1}:=(g_{\zeta_k}(H_{k+1}) - W_{\zeta_k})/(W_{\zeta_k} - V_{\zeta_k}^L)$ through $\wt{L}_{k+1}$ is bounded from below by a positive universal constant times the probability that a Brownian motion starting from $\wt{z}_k^\vartheta$ exits $B(\wt{z}_k^\vartheta,\wt{d})\cap\h$, $\wt{d}=\dist(\wt{z}_k^\vartheta,\wt{L}_{k+1})$, through $\partial B(\wt{z}_k^{\vartheta},\wt{d}) \cap \h$. This latter probability is bounded from below by a positive universal constant times $\wt{y}_k^\vartheta/\wt{d}$. Thus $1/C\gtrsim \wt{y}_k^\vartheta/\wt{d}$, as desired.

The conformal invariance of Brownian motion and the estimates above also imply that $\sin(\arg(\wt{z}_k^\vartheta)) \asymp 1$, hence $|\wt{z}_k^\vartheta| \asymp |\wt{y}_k^\vartheta|$.  Combining this with \eqref{eqn::dist_ratio_bound2} implies that
\[ \frac{\dist(\wt{z}_k^\vartheta,\wt{L}_{k+1})}{|\wt{z}_k^\vartheta|} \gtrsim C.\]
Thus, by the triangle inequality,
\[ \dist(\wt{L}_{k+1},0) \gtrsim C | \wt{z}_k^\vartheta|\]
(provided $C$ is large enough).  Since $|\wt{z}_k^\vartheta| \asymp 1$, this proves \eqref{eqn::dist_ratio_bound}, hence the lemma.
\end{proof}

\begin{proof}[Proof of Theorem~\ref{thm::one_point_exponent_general}]
Lemma~\ref{lem::lower_bound_one_point_exponent} implies the lower bound in \eqref{eqn::onepointestimate} because we can take, e.g., $\delta=\tfrac{1}{2}$. In order to prove the upper bound, it is sufficient to show
\[ \p[\tau_{\eps}<\infty] \le \eps^{\alpha+o(1)} \quad\text{as}\quad \eps \to 0.\]

We are first going to perform a change of coordinates.  Let $\varphi \colon \HH \to \HH$ be the M\"obius transformation
$z \mapsto \varphi(z) :=\eps z/(1-z)$.
Fix $\wt{x}^R \in [0^+,1)$ and let $\wt{\eta}$ be an $\SLE_\kappa(\rho_{1,R},\rho_{2,R})$ process with force points located at $(\wt{x}^R, 1)$ as in Theorem~\ref{thm::one_point_exponent_general}.  Then the law of $\eta=\varphi(\wt{\eta})$ is that of an $\SLE_{\kappa}(\rho_L;\rho_R)$ process with force points $(-\eps; x_R)$ where $x_R = \eps \wt{x}^R / (1-\wt{x}^R)$ and
\begin{equation}
\label{eqn::changeinparameter}
\rho_L=\kappa-6-\big(\rho_{1,R} + \rho_{2,R}\big) \quad \text{ and } \quad \rho_R=\rho_{1,R}.
\end{equation}
Let $\sigma_1$ be the first time that $\eta$ hits $\partial \D$ and let $V_t^L,V_t^R$ denote the evoltuion of $x_L,x_R$ under $g_t$, respectively. For $u\ge 0$, define $T^L_u=\inf\{t\ge 0: W_t-V^L_t=u\}$ (as in the statement of Lemma~\ref{lem::tubelemma}). Then it is sufficient to prove $\p[\sigma_1<T_0^L]\le \eps^{\alpha+o(1)}$. Note that the exponent $\alpha$ comes from the sum of the exponent of $|V_t^L - V_t^R|$ and the exponent of $|W_t - V_t^L|$ in the left martingale $M^L$ from \eqref{eqn::left_conditioning} with these weights. For $u\ge 0$, define $\tau^L_u=\inf\{t\ge 0: M^L_t=u \}$. Note that $\tau^L_0=T^L_0$. Fix $\beta\in(0,1)$ and set $\vartheta=\eps^\beta$. For $u>0$, we have the bound
\begin{equation}
\label{eqn::two_terms}
\p[\sigma_1<\tau_0^L]\le \p[\tau^L_u<\tau_0^L]+\p[\sigma_1<\tau_0^L<\tau^L_u].
\end{equation}

We claim that exists constants $C_1 > 0$ and $\gamma > 0$ depending only on $\rho_L$, $\rho_R$, and $\kappa$ such that
\begin{equation}
\label{eqn::w_v_bound}
 |W_t-V_t^L|^{\gamma} \leq C_1 M_t^L \quad \text{for all} \quad t \in [0,\sigma_1].
\end{equation}
Since $\rho_{1,R} + \rho_{2,R} > \tfrac{\kappa}{2}-4$ it follows that $\rho_L < \tfrac{\kappa}{2}-2$.  Therefore the sign of the exponent of $|V_t^L - V_t^R|$ in the definition of $M_t^L$ is the same as the sign of $\rho_R$.  If $\rho_R \geq 0$, then the exponent has a positive sign.  In this case, $M_t^L \geq |W_t - V_t^L|^\alpha$ so that we can take $\gamma=\alpha$.  Now suppose that $\rho_R < 0$.  By \eqref{eqn::tube_upper} of Lemma~\ref{lem::tubeestimate} we know that there exists a constant $C_2 > 0$ such that
\begin{equation}
\label{eqn::exit_left_right_difference}
 |V_t^L - V_t^R| \leq C_2\quad\text{for all}\quad t \in [0,\sigma_1].
\end{equation}
Thus, in this case, there exists a constant $C_3 > 0$ such that $M_t^L \geq C_3 |W_t - V_t^L|^{(\kappa-4-2\rho_L)/\kappa}$.  Therefore we can take $\gamma = (\kappa-4-2\rho_L) / \kappa$.  This proves the claimed bound in \eqref{eqn::w_v_bound}.

Set $u=\vartheta^\gamma/C_1$.  To bound the second term on the right side of \eqref{eqn::two_terms}, we first note by \eqref{eqn::w_v_bound} that
\begin{equation}
\label{eqn::pre_application_tube}
 \p[\sigma_1<\tau_0^L<\tau^L_u]\le \p[\sigma_1<T^L_0 \wedge T^L_{\vartheta}].
\end{equation}
By Lemma~\ref{lem::tubelemma}, we know that
\begin{equation}
\label{eqn::application_tube}
\p[\sigma_1<T_0^L \wedge T^L_{\vartheta}]\le p_2^{1/(C\vartheta)}.
\end{equation}

We will now bound the first term on the right side of \eqref{eqn::two_terms}.  Since $\tau^L_0,\tau^L_u$ are stopping times for the martingale $M^L$ and $M_{\tau_0 \wedge \tau_u} = u \p[\tau_u^L < \tau_0^L]$, we have that
\begin{equation}
\label{eqn::application_mart}
\p[\tau_u^L<\tau_0^L] = \frac{1}{u} \E[ M^L_{\tau_0 \wedge \tau_u}] = \frac{M^L_0}{u} =\frac{\eps^{\alpha}}{u(1-\wt{x}^R)^{(\kappa-4-2\rho_L) \rho_R /(2\kappa)}}.
\end{equation}
Combining \eqref{eqn::two_terms} with \eqref{eqn::application_tube} and \eqref{eqn::application_mart} we get that $\p[\sigma_1<T^L_0]\le \eps^{\alpha+o(1)}$, as desired.
\end{proof}

Recall that (see for example \cite[Section 4]{BM}) the $\beta$-Hausdorff measure of a set $A \subseteq \R$ is defined as
\[\LH^{\beta}(A)=\lim_{\eps \to 0^+} \LH^{\beta}_{\eps}(A)\]
where
\[\LH^{\beta}_{\eps}(A):=\inf\left\{\sum_j |I_j|^{\beta} \ :\  A \subseteq \cup_j I_j\ \text{ and }\ |I_j|\le \eps\ \text{ for all }\ j\right\}.\]

\begin{proof}[Proof of Theorem~\ref{thm::boundary_dimension} for $\kappa \in (0,4)$, upper bound]
Fix $\kappa\in (0,4),\rho\in (-2,\frac{\kappa}{2}-2).$ Let $\eta$ be an $\SLE_{\kappa}(\rho)$ process with a single force point located at $0^+$.  Let $\alpha \in (0,1)$ be as in \eqref{eqn::onepoint_constant_rho}.  Fix $0 < x < y$.  We are going to prove the result by showing that
\begin{equation}
\label{eqn::hd_ubd}
\dimH(\eta\cap [x,y])\le 1-\alpha \quad\text{almost surely}.
\end{equation}
For each $k \in \Z$ and $n \in \N$ we let $I_{k,n} = [k2^{-n},(k+1)2^{-n}]$ and let $z_{k,n}$ be the center of $I_{k,n}$.  Let $\CI_n$ be the set of $k$ such that $I_{k,n} \subseteq [x/2,2y]$ and let $E_{k,n}$ be the event that $\eta$ gets within distance $2^{1-n}$ of $z_{k,n}$.  Therefore there exists $n_0 =n_0(x,y)$ such that for every $n \geq n_0$ we have that $\{ I_{k,n} : k \in \CI_n,\ E_{k,n} \text{ occurs}\}$ is a cover of $\eta \cap [x,y]$.

Fix $\zeta > 0$.  Theorem~\ref{thm::one_point_exponent_general} implies that there exists a constant $C_1 > 0$ (independent of $n$) and $n_1 = n_1(\zeta)$ such that
\[ \p[E_{k,n}] \le C_1 2^{-(\alpha-\zeta)n} \quad\text{for each}\quad n \geq n_1 \quad\text{and}\quad k \in \CI_n.\]
Consequently, there exists a constant $C_2 > 0$ such that
\[ \E \left[ \LH^{\beta}_{2^{-n}}(\eta\cap [x,y]) \right]\le \E\left[\sum_{k \in\CI_n} 2^{-\beta n}\one_{E_{k,n}}\right]\le C_2 2^{-\beta n} \times  2^n \times 2^{-(\alpha-\zeta)n}.\]
By Fatou's lemma,
\begin{align*}
 \E\left[ \LH^{1-\alpha+2\zeta}(\eta\cap [x,y]) \right] &\le \liminf_n \E\left[ \LH^{1-\alpha+2\zeta}_{2^{-n}}(\eta\cap [x,y]) \right]\\
 &\le \liminf_n C_2 2^{-n \zeta} =0.
\end{align*}
This implies that $\LH^{1-\alpha+2\zeta}(\eta\cap [x,y])=0$ almost surely.  This proves \eqref{eqn::hd_ubd} which completes the proof of the upper bound.
\end{proof}

\subsection{The lower bound}
\label{subsec::boundary_intersection_lower_bound}

Throughout, we fix $\kappa\in (0,4)$ and $\rho \in (-2,\tfrac{\kappa}{2}-2)$ and let $h$ be a GFF on $\HH$ with boundary data $-\lambda$ on $\R_-$ and $\lambda(1+\rho)$ on $\R_+$.  (Recall the values in \eqref{eqn::constants} as well as Figure~\ref{fig::gff_boundary_data_flow}.)  For each $x \geq 0$, we let $\eta^x$ be the flow line of $h$ starting from $x$ and let $\eta = \eta^0$.  Note that $\eta$ is an $\SLE_\kappa(\rho)$ process in $\h$ from $0$ to $\infty$ with a single force point located at~$0^+$, i.e., has configuration $(\h,0,0^+,\infty)$ (recall the notation of Section~\ref{subsec::radon_nikodym}).  By Lemma~\ref{lem::sle_kappa_rho_boundary_interaction}, it follows that $\eta$ can hit $(0,\infty)$.  For each $x > 0$, $\eta^x$ is an $\SLE_\kappa(2+\rho,-2-\rho;\rho)$ process with configuration $(\h,x,(0,x^-),(x^+),\infty)$.  By Lemma~\ref{lem::sle_kappa_rho_boundary_interaction}, it follows that $\eta^x$ can hit $(x,\infty)$ and, if $\rho > -\kappa/2$, then $\eta^x$ can also hit $(0,x)$.  Fix $\delta \in (0,1)$, $a > \log 8$, and let
\[ \eps_n = e^{-an} \quad\text{for each}\quad n \in \N.\]
We will eventually take limits as $a \to \infty$ and $\delta \to 0^+$.  For $U \subseteq \HH$, we let
\begin{equation}
\label{eqn::hitting_time}
 \sigma^x(U)=\inf\{t \geq 0: \eta^x(t) \in \ol{U}\}.
\end{equation}
We will omit the superscript in \eqref{eqn::hitting_time} if $x =0$.  For $k \in \N$ and $x \in [1,\infty)$, we let
\[ x_k =
\begin{cases}
x-\tfrac{1}{4}\eps_k \quad&\text{if}\quad\quad k \geq 2 \quad\quad\text{and}\\
0 \quad&\text{if}\quad \quad k=1.
\end{cases}\]
We also let
\begin{equation}
\label{eqn::hitting_time_ball}
\sigma_m^x = \sigma^{x_m}(B(x,\eps_{m+1})).
\end{equation}
Let $E_k^1(x)$ be the event that
\begin{enumerate}[(i)]
\item $\sigma_k^x < \infty$ and $\im(\eta^{x_k}(\sigma_k^x)) \geq \delta \eps_{k+1}$ and
\item $\eta^{x_k}$ hits $B(x,\eps_{k+1})$ before exiting $B(x,\tfrac{1}{2}\eps_k)$.
\end{enumerate}

\begin{figure}[ht!]
\begin{center}
\includegraphics[scale=0.85]{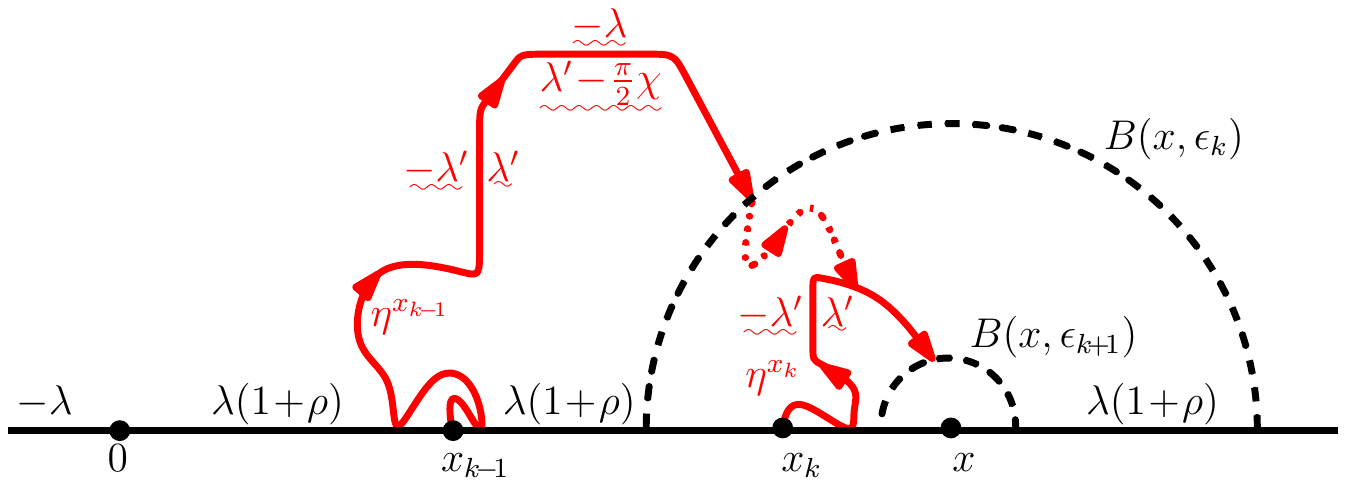}
\end{center}
\caption{\label{fig::two_flowlines_merge}
On $E_{k-1}^1(x)$, $\eta^{x_{k-1}}$ hits $B(x,\eps_k)$ and does so for the first time above the horizontal line through $i \delta \eps_k$.  Given that $E_k^1(x)$ has occurred, $E_k^2(x)$ is the event that $\eta^{x_{k-1}}$ merges with $\eta^{x_k}$ before the path leaves the annulus $B(x,\tfrac{1}{2} \eps_{k-1}) \setminus B(x,\eps_{k+1})$.  Also indicated is the boundary data for $h$ along $\partial \h$ as well as along the paths $\eta^{x_{k-1}}$ and $\eta^{x_k}$.
}
\end{figure}
\noindent We let $E_k^2(x)$ be the event that
$\eta^{x_{k-1}}|_{[\sigma_{k-1}^x,\infty)}$ merges with $\eta^{x_k}|_{[0,\sigma_k^x]}$ before exiting the annulus $B(x,\tfrac{1}{2} \eps_{k-1}) \setminus B(x,\eps_{k+1})$ (see Figure~\ref{fig::two_flowlines_merge}).  Finally, we let $E_k(x) = E_k^1(x) \cap E_k^2(x)$,
\[E^{m,n}(x)= E_{m+1}^1(x) \cap \bigcap_{k=m+2}^n E_k(x), \quad\text{and}\quad E^n(x) = E^{0,n}(x).\]
The following is the main input into the proof of the lower bound.

\begin{proposition}
\label{prop::two_point_estimate}
For each $\delta \in (0,1)$, there exists a constant $c(\delta) > 0$ such that for all $x,y \in [1,2]$ and $m \in \N$ such that $\frac{1}{2}\eps_{m+1}\le |x-y|<\frac{1}{2}\eps_m$ we have
\[ \p[E^n(x), E^n(y)] \leq c(\delta)^{-m} \eps_m^{-\alpha} \p[E^n(x)]\p[E^n(y)].\]
\end{proposition}

The main steps in the proof of Proposition~\ref{prop::two_point_estimate} are contained in the following three lemmas.

\begin{lemma}
\label{lem::approx_ind1}
For each $x \geq 1$ and $m,n \in \N$ with $m \leq n$, we have that
\begin{equation}
\p[ E^{m,n}(x), E^m(x)] \asymp \p[ E^{m,n}(x)] \p[ E^m(x)] \label{eqn::given_one}
\end{equation}
If, moreover, $y \geq 1$ and $\tfrac{1}{2}\eps_{m+2} < |x-y| \leq \tfrac{1}{2} \eps_{m+1}$, then we have that
\begin{equation*}
\p[E^{m+1,n}(x), E^{m+1,n}(y), E^m(x)] \asymp \p[E^{m+1,n}(x)] \p[ E^{m+1,n}(y)] \p[ E^m(x)].
\end{equation*}
In each of the above, the constants in $\asymp$ depend only on $\delta$, $\kappa$ and $\rho$.
\end{lemma}
\begin{proof}
We begin by proving \eqref{eqn::given_one} which is equivalent to
\[ \p[E^{m,n}(x) \giv E^m(x)]\asymp \p[E^{m,n}(x)].\]
Recall that $\eta^{x_{m+1}}$ is an $\SLE_{\kappa}(2+\rho,-2-\rho;\rho)$ process with configuration
\[c=(\HH,x_{m+1}, (0,x_{m+1}^-),(x_{m+1}^+),\infty).\]
Let $\omega=\eta(\sigma(B(x,\eps_m)))$, let $H$ be the closure of the complement of the unbounded connected component of $\h \setminus \cup_{j=1}^m \eta^{x_j}([0,\sigma_j^x])$, and let $v$ be the rightmost point of $H \cap \R$ (see Figure~\ref{fig::given_one}).  The conditional law of $\eta^{x_{m+1}}$ given $\eta^{x_1}|_{[0,\sigma_1^x]},\ldots, \eta^{x_m}|_{[0,\sigma_m^x]}$ on $E^m(x)$ is that of an $\SLE_{\kappa}(2,\rho,-2-\rho;\rho)$ process in
\[\wt{c}=(\HH\setminus H, x_{m+1}, (\omega, v,x_{m+1}^-), (x_{m+1}^+),\infty)\]
(recall Figure~\ref{fig::conditional_law}.)

\begin{figure}[ht!]
\begin{center}
\includegraphics[scale=0.85]{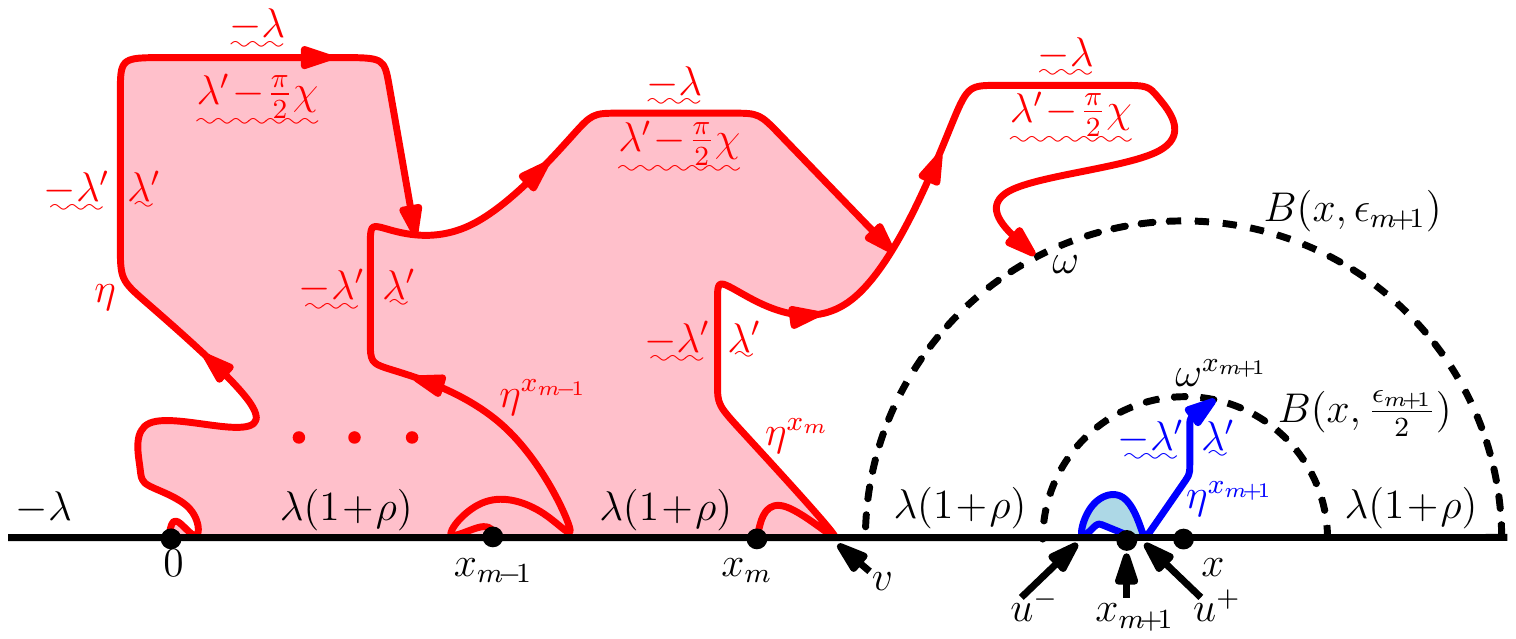}
\end{center}
\caption{\label{fig::given_one}
Let $H$ (shown in red) be the closure of the complement of the unbounded connected component of $\h \setminus \cup_{j=1}^m \eta^{x_j}([0,\sigma_j^x])$ and let $K$ (shown in blue) be the closure of the complement of the unbounded connected component of $\h \setminus \eta^{x_{m+1}}([0,\tau])$ where $\tau$ is the first time that $\eta^{x_{m+1}}$ leaves $U=B(x,\tfrac{\eps_{m+1}}{2})$.  Then $\dist(H,K)\gtrsim \diam(U)$.}
\end{figure}

Let $U = B(x,\tfrac{1}{2}\eps_{m+1})$, $\tau = \sigma^{x_{m+1}}(\h \setminus U)$, $K$ be the closure of the complement of the unbounded connected component of $\h \setminus \eta^{x_{m+1}}([0,\tau])$, $\omega^{x_{m+1}}=\eta^{x_{m+1}}(\tau)$, and let $u^-,u^+$ be the leftmost (resp.\ rightmost) point of $K \cap \R$.  By Lemma~\ref{lem::change_of_domains}, we have that
\[\frac{d\mu_{\wt{c}}^U}{d\mu_{c}^U}=\frac{Z(\wt{c}_{\tau})/Z(\wt{c})}{Z(c_{\tau})/Z(c)}\exp(-\xi m(\HH; H,K))\]

where
\begin{align*}
c_{\tau} &= (\HH\setminus K,\omega^{x_{m+1}},(0,u^-),(u^+),\infty),\\
\wt{c}_{\tau} &= (\HH\setminus (H \cup K), \omega^{x_{m+1}}, (\omega,v,u^-),(u^+),\infty).
\end{align*}
Note that $H \subseteq \h \setminus B(x,\tfrac{3}{4}\eps_{m+1})$, $K \subseteq \ol{B(x,\tfrac{1}{2}\eps_{m+1})}$, and $\diam(U)=\eps_{m+1}$.  Consequently,
\[\frac{\dist(H,K)}{\diam(U)}\gtrsim 1. \]
Therefore Lemma \ref{lem::change_of_domains_bound_RN} implies there exists $C_1 \geq 1$ so that
\begin{equation}
\label{eqn::radon_bound}
\frac{1}{C_1} \leq \frac{d\mu^U_{\wt{c}}}{d\mu^U_{c}} \leq C_1.
\end{equation}
This proves \eqref{eqn::given_one} in the case that $n=m+1$.  We now suppose that $n \geq m+2$.  Given $\eta^{x_{m+1}}|_{[0,\tau]}$, we similarly have that the Radon-Nikodym derivative between the conditional law of $\eta^{x_n}$ stopped upon exiting the connected component of $B(x,\tfrac{1}{2}\eps_n) \setminus \eta^{x_{m+1}}([0,\tau])$ with $x_n$ on its boundary with respect to the law in which we additionally condition on $H$ on $E_m(x)$ is bounded from above and below by $C_1$ and $C_1^{-1}$, respectively, possibly by increasing the value of $C_1 > 1$ (see Figure~\ref{fig::one_point_approx_ind}).  Moreover, conditional on both of the paths $\eta^{x_{m+1}}|_{[0,\sigma^{x_{m+1}}(B(x,\eps_{n+1}))]}$ and $\eta^{x_n}|_{[0,\sigma_n^x]}$ as well as the event that they have merged before exiting $U$, the joint law of $\eta^{x_j}|_{[0,\sigma_j^x]}$ for $j=m+2,\ldots,n-1$ is independent of $\eta^{x_k}|_{[0,\sigma_k^x]}$ for $k=1,\ldots,m$ (see Figure~\ref{fig::one_point_approx_ind}).  This proves \eqref{eqn::given_one}.

The second part of the lemma is proved similarly.
\end{proof}

\begin{figure}[ht!]
\begin{center}
\includegraphics[scale=0.85]{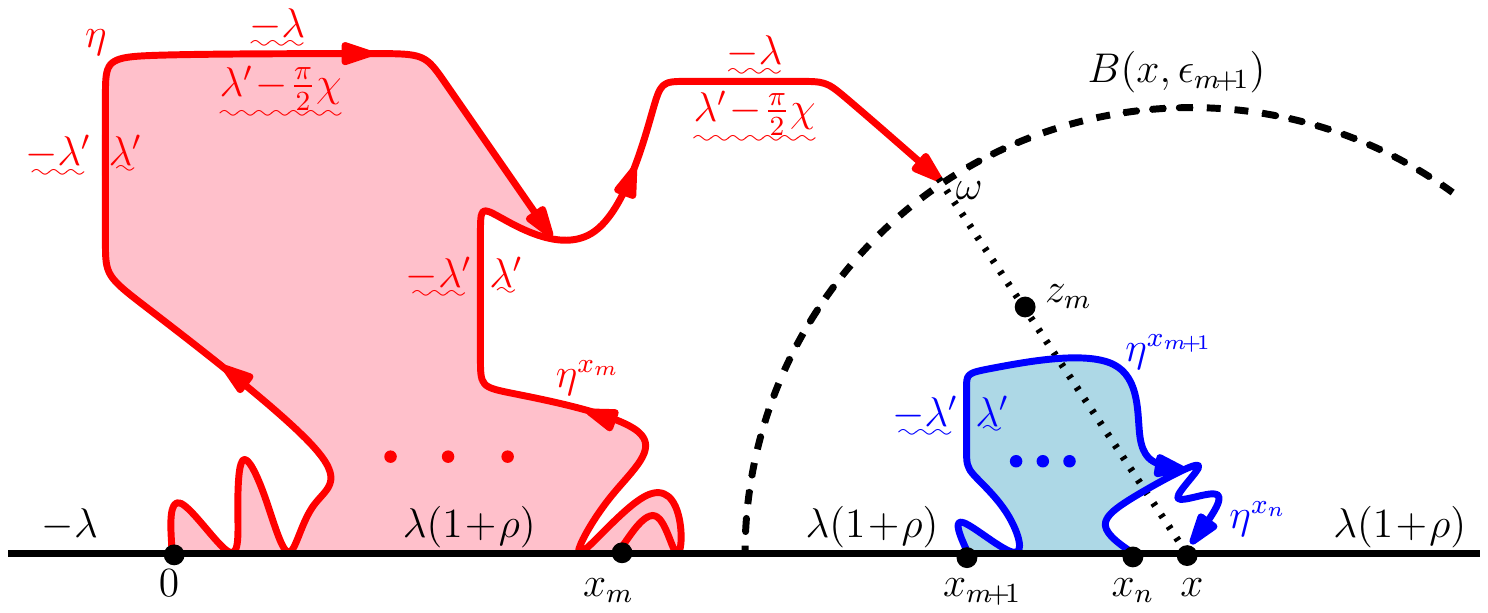}
\end{center}
\caption{\label{fig::one_point_approx_ind}
Assume that we are working on $E^m(x) \cap E^{m,n}(x)$.  Let $H$ (shown in red) be the closure of the complement of the unbounded connected component of $\h \setminus \cup_{j=1}^m \eta^{x_j}([0,\sigma_j^x])$ and let $K$ (shown in blue) be the closure of the complement of the unbounded connected component of $\h \setminus \cup_{j=m+1}^n \eta^{x_j}([0,\sigma^x_j])$.  Let $z_m$ be the point that lies at distance $\delta\eps_{m+1}$ from $\omega$ along the line connecting $\omega$ to $x$.  Then a Brownian motion starting from $z_m$ has positive probability to exit $\h\setminus(H\cup K)$ through each of the left side of $H$, the right side of $H$, and the left side of $K$.}
\end{figure}

\begin{lemma}
\label{lem::approx_ind2}
For each $x \geq 1$ and $m,n \in \N$ with $m \leq n$ we have that
\begin{equation}
\label{eqn::one_point_approx_ind}
 \p[E^n(x)]\asymp \p[E^m(x)] \p[E^{m,n}(x)]
\end{equation}
where the constants depend only on $\delta$, $\kappa$, and $\rho$.
\end{lemma}
\begin{proof}
The upper bound follows from \eqref{eqn::given_one} of Lemma~\ref{lem::approx_ind1}.  To complete the proof of the lemma, it suffices to show that
\[ \p[ E_{m+1}^2(x)\giv E^m(x), E^{m,n}(x)] \asymp 1.\]
Throughout, we assume that we are working on $E^m(x) \cap E^{m,n}(x)$.  To see this, we let $H$ (resp.\ $K$)  be the closure of the complement of the unbounded connected component of $\h \setminus \cup_{j=1}^m \eta^{x_j}([0,\sigma^x_j])$ (resp.\ $\h \setminus \cup_{j=m+1}^n \eta^{x_j}([0,\sigma^x_j])$).  Let $\omega = \eta^{x_m}(\sigma^x_m)$ and let $z_m$ be the point which lies at distance $\delta \eps_{m+1}$ from $\omega$ along the line segment connecting $\omega$ to $x$ (see Figure~\ref{fig::one_point_approx_ind}).  Note that the probability that a Brownian motion starting from $z_m$ exits $\h \setminus (H \cup K)$ in the left (resp.\ right) side of $H$ is $\asymp 1$ (though this probability decays as $\delta \downarrow 0$) and likewise for the left side of $K$.  Let $\varphi \colon \h \setminus (H \cup K) \to \h$ be the conformal map which takes $z_m$ to $i$ and $\omega$ to $0$.  Let $x_L$ (resp.\ $x_R$) be the image of the leftmost (resp.\ rightmost) point of $H \cap \R$ under $\varphi$.  The conformal invariance of Brownian motion implies that there exists $\eps > 0$ depending only on $\delta$ such that $|x_q| \geq \eps$ for $q \in \{L,R\}$.  Let $y_L$ (resp.\ $y$) be the image of the leftmost point of $K \cap \R$ (resp.\ $\eta^{x_{m+1}}(\sigma_{m+1}^x)$) under $\varphi$.  By shrinking $\eps > 0$ if necessary (but still depending only on $\delta$), it is likewise true that $y-y_L \geq \eps$ and $y_L \leq \eps^{-1}$.  Consequently, it follows from Lemma~\ref{lem::sle_kappa_rho_hit_boundary_segment} that $\eta^{x_m}|_{[\sigma_m^x,\infty)}$ has a positive chance (depending only on $\delta$, $\kappa$, and $\rho$) of hitting (hence merging into) the left side of $\eta^{x_{m+1}}|_{[0,\sigma_{m+1}^x)}$ before leaving $B(x,\tfrac{1}{2}\eps_{m}) \setminus B(x,\eps_{m+2})$.
\end{proof}

\begin{lemma}
\label{lem::two_point_lower_bound}
For each $\delta \in (0,1)$ there exists a constant $c(\delta) > 0$ such that the following is true.  For each $x \geq 1$, we have that
\[\p[E^m(x)] \geq c(\delta)^m \times \eps_m^{\alpha}.\]
\end{lemma}
\begin{proof}
By \eqref{eqn::given_one} of Lemma~\ref{lem::approx_ind1}, we know that
\[ \p[ E_k^{1}(x) \giv E^{k-1}(x)] \asymp \p[ E_k^1(x)].\]
Therefore we just have to show that there exists a constant $c(\delta) > 0$ such that
\begin{align}
 \p[ E_k^1(x)] \geq c(\delta) \left( \frac{\eps_{k+1}}{\eps_k} \right)^{\alpha} &= c(\delta) e^{-a \alpha} \quad\text{and} \label{eqn::ek1_giv_ek}\\
 \p[ E_k^2(x) \giv E^{k-1}(x), E_k^1(x)] &\asymp 1. \label{eqn::ek2_giv_ek}
\end{align}

Note that \eqref{eqn::ek2_giv_ek} follows from Lemma~\ref{lem::sle_kappa_rho_hit_boundary_segment} using the same argument as in the proof of Lemma~\ref{lem::approx_ind2}.
We know that $\eta^{x_k}$ is an SLE$_{\kappa}(2+\rho,-2-\rho;\rho)$ process within the configuration $c=(\h,x_k,(0,x_k^-), (x_k^+),\infty)$. Consequently, \eqref{eqn::ek1_giv_ek} follows by combining Corollary~\ref{cor::lower_bound_one_point_exponent} and Lemma~\ref{lem::change_of_domains_bound_RN}.  The latter is used to get that the Radon-Nikodym derivative between the law of an $\SLE_{\kappa}(2+\rho,-2-\rho;\rho)$ process with configuration $(\h,x_k,(0,x_k^-), (x_k^+),\infty)$ and the law of an $\SLE_{\kappa}(-2-\rho;\rho)$ process with configuration $(\h,x_k,(x_k^-),(x_k^+),\infty)$, where each path is stopped upon exiting $B(x,
\tfrac{\eps_k}{2})$, is bounded both from below and above by universal positive and finite constants.
\end{proof}

\begin{proof}[Proof of Proposition~\ref{prop::two_point_estimate}]
We have that,
{\allowdisplaybreaks
\begin{align*}
\p[E^n(x),E^n(y)] \leq& \p[E^n(x),E^{m,n}(y)]\\
\lesssim& \p[E^{m}(x)] \p[E^{m+1,n}(x)] \p[E^{m+1,n}(y)]  \quad\quad\text{(Lemma~\ref{lem::approx_ind1})}\\
=& \frac{\p[E^{m}(x)] \p[E^m(y)] }{\p[E^m(y)] }\p[E^{m+1,n}(x)]\p[E^{m+1,n}(y)]\\
\lesssim& \frac{\p[E^n(x)]\p[E^n(y)]}{c(\delta)^m \eps_m^{\alpha}}\quad\quad\text{(Lemma~\ref{lem::approx_ind2} and Lemma~\ref{lem::two_point_lower_bound})}
\end{align*}}
\end{proof}

\begin{proof}[Proof of Theorem~\ref{thm::boundary_dimension}]
We are first going to give the lower bound for $\kappa \in (0,4)$ and then explain how to extract the dimension result for $\kappa' > 4$ from the result for $\kappa \in (0,4)$.  For each $\beta \in \R$ and Borel measure $\mu$, let
\[ I_{\beta}(\mu):=\int\int \frac{\mu(dz)\mu(dw)}{|z-w|^{\beta}}\]
be the $\beta$-energy of $\mu$.  To prove the lower bound, we will show that, for each $\zeta>0$, there exists a nonzero Borel measure supported on $\eta \cap [1,2]$ that has finite $(1-\alpha-2\zeta)$-energy.

Fix $n \in \N$.  We divide $[1,2]$ into $\eps_n^{-1}$ intervals of equal length $\eps_n$ and let $z_{j,n} = (j-\tfrac{1}{2}) \eps_n + 1$ be the center of the $j$th such interval for $j=1,\ldots,\eps_n^{-1}$.  Let $\CC_n$ be the subset of $\CD_n = \{z_{j,n} : j=1,\ldots,\eps_n^{-1}\}$ for which $E^n(z)$ occurs. Let $I_n(z) = [z-\tfrac{\eps_n}{2},z+\tfrac{\eps_n}{2}]$ be the interval with center $z$ and length $\eps_n$. Finally, we let
\[ \CC=\bigcap_{k\ge 1}\overline{\bigcup_{n\ge k}\bigcup_{z\in\CC_n} I_n(z)}.\]
It is easy to see that
\[ \CC \subseteq \eta\bigcap\R_+.\]
Let $\mu_n$ be the measure on $[1,2]$ defined by
\[ \mu_n(A)=\int_A \sum_{z\in\CD_n}\frac{\one_{E^n(z)}}{\p[E^n(z)]}\one_{I_n(z)}(z')dz'\quad\text{for}\quad A\subseteq [1,2] \quad\text{Borel}.\]
Then $\E[ \mu_n([1,2])]=1$.  Moreover, we have that
{\allowdisplaybreaks
\begin{align*}
\E&[\mu_n([1,2])^2]=\eps_n^2\sum_{z,w\in\CD_n}\frac{\p[ E^n(z)\cap E^n(w)] }{\p[E^n(z)]\p[E^n(w)]}\\
&=\eps_n^2 \sum_{\substack{z,w\in\CD_n \\ z\neq w}}\frac{\p[ E^n(z)\cap E^n(w)]}{\p[E^n(z)]\p[E^n(w)]}+\eps_n^2\sum_{z\in\CD_n}\frac{1}{\p[E^n(z)]}\\
&\lesssim \eps_n^2\sum_{\substack{z,w\in\CD_n \\ z\neq w}} |z-w|^{-\alpha-\zeta}+\eps_n^2\sum_{z\in\CD_n}\eps_n^{-1+\alpha-\zeta}\quad\text{(Proposition \ref{prop::two_point_estimate} and Lemma~\ref{lem::two_point_lower_bound})}\\
&\lesssim 1
\end{align*}}
provided we choose $n$ and $a$ large enough.  Set $\beta=1-\alpha-2\zeta$.  We also have that
{\allowdisplaybreaks
\begin{align*}
\E&[I_{\beta}(\mu_n)]=\sum_{\substack{z,w\in\CD_n \\ z\neq w}}\frac{\p[ E^n(z)\cap E^n(w)]}{\p[ E^n(z)]\p[E^n(w)]}\iint\limits_{I_n(z)\times I_n(w)}\frac{dz'dw'}{|z'-w'|^{\beta}}\\
&=\sum_{\substack{z,w\in\CD_n \\ z\neq w}}\frac{\p[ E^n(z)\cap E^n(w)]}{\p[E^n(z)]\p[E^n(w)]}\iint\limits_{I_n(z)\times I_n(w)}\frac{dz'dw'}{|z'-w'|^{\beta}}\\
&\quad\quad\quad\quad+\sum_{z\in\CD_n}\frac{1}{\p[ E^n(z)] }\iint\limits_{I_n(z)\times I_n(z)}\frac{dz'dw'}{|z'-w'|^\beta}\\
&\lesssim \sum_{\substack{z,w\in\CD_n \\ z\neq w}}\frac{\p[ E^n(z)\cap E^n(w)]}{\p[E^n(z)]\p[E^n(w)]}\frac{\eps_n^2}{|z-w|^{\beta}}+\sum_{z\in\CD_n}\frac{1}{\p[ E^n(z)]}\eps_n^{2-\beta}\\
&\lesssim \sum_{\substack{z,w\in\CD_n \\ z\neq w}} |z-w|^{-\alpha-\zeta}\eps_n^2|z-w|^{-\beta}+\sum_{z\in\CD_n} \eps_n^{-1+\alpha-\zeta}\eps_n^{2-\beta}
 \lesssim 1.
\end{align*}}
Consequently, the sequence $(\mu_n)$ has a subsequence $(\mu_{n_k})$ that converges weakly to some nonzero measure $\mu$. It is clear that $\mu$ is supported on $\CC$ and has finite $(1-\alpha-2\zeta)$-energy. From \cite[Theorem 4.27]{BM}, we know that
\[ \p\left[ \dimH(\eta\bigcap\R_+)\ge 1-\alpha-2\zeta \right]>0.\]
Since $\eta$ is conformally invariant, by 0-1 law (see \cite{beffara2008}), we have that
\[ \p\left[ \dimH(\eta\bigcap\R_+)\ge 1-\alpha-2\zeta\right]=1\]
for any $\zeta>0$.  This proves the lower bound for $\kappa \in (0,4)$.

\medbreak

It is left to prove the result for $\kappa'>4$.  Fix $\rho' \in (\frac{\kappa'}{2}-4, \frac{\kappa'}{2}-2)$. Consider a GFF $h$ on $[-1,1]^2$ with the boundary values as depicted in Figure~\ref{fig::counterflowline_and_flowline} with $\rho_R' = \rho'$ and $\rho_L'=0$, and let $\eta'$ be the counterflow line of $h$ from $i$ to $-i$. Then $\eta'$ is an $\SLE_{\kappa'}(\rho')$ process with a single force point located at $(i)^+$, i.e., immediately to the right of $i$.  As explained in Figure~\ref{fig::counterflowline_and_flowline}, the right boundary of $\eta'$ is equal to the flow line $\eta_R$ of $h$ with angle $-\tfrac{\pi}{2}$ starting from $-i$.  In particular, $\eta_R$ is an $\SLE_{\kappa}(\frac{\kappa}{2}-2;\kappa-4+\tfrac{\kappa}{4}\rho')$ process with force points $((-i)^-;(-i)^+)$ where $\kappa=\tfrac{16}{\kappa'} \in (0,4)$. The intersection of $\eta'$ with the counterlcockwise segment $\CS$ of $\partial ([-1,1]^2)$ from $-i$ to $i$ coincides with $\eta_R \cap \CS$.  Consequently, it follows that the dimension of $\eta'\cap \CS$ is given by
\[ 1-\frac{1}{\kappa}\left( \kappa-2+\frac{\kappa}{4}\rho' \right)\left(\frac{\kappa}{2}+\frac{\kappa}{4}\rho' \right)=1-\frac{1}{\kappa'}\left(\rho'+2 \right)\left(\rho'+4-\frac{\kappa'}{2}\right).\]

\end{proof}

\section{The intersection of flow lines}
\label{sec::flow_line_intersection}

In this section, we will prove Theorem~\ref{thm::two_flowline_dimension}.  We begin in Section~\ref{subsec::derivative} by proving an estimate for the derivative of the Loewner map associated with an $\SLE_\kappa(\ul{\rho})$ process when it gets close to a given point.  Next, in Section~\ref{subsec::hitting} we will prove the one point estimate which we will use in Section~\ref{subsec::intersection_upper_bound} to prove the upper bound.  Finally in Section~\ref{subsec::intersection_lower_bound} we will complete the proof by establishing the lower bound.

\subsection{Derivative estimate}
\label{subsec::derivative}

Recall from Section~\ref{subsec::distort} that for a point $w$ in a simply connected domain $U$, $\confrad(w;U)$ denotes the conformal radius of $U$ as viewed from $w$.  Fix $\kappa\in(0,4)$, let $\eta$ be an ordinary $\SLE_\kappa$ process in $\h$ from $0$ to $\infty$ and, for each $t$, let $\h_t$ denote the unbounded connected component of $\h \setminus \eta([0,t])$.  We use the notation of \cite[Section~6.1]{lawler2009}.  We let
\[ Z_t = Z_t(z) = X_t + i Y_t = g_t(z) - W_t.\]
For $z \in \h$, we let
\begin{equation}
\label{eqn::processes}
\Delta_t = |g_t'(z)|,\quad \Upsilon_t = \frac{Y_t}{|g_t'(z)|},\quad \Theta_t = \arg Z_t,\quad\text{and}\quad S_t = \sin \Theta_t.
\end{equation}
We note that $\Upsilon_t = \tfrac{1}{2} \confrad(z; \h_t) \asymp \dist(z,\partial \h_t)$.  For each $r \in \R$, we also let
\begin{equation}
\label{eqn::lambda_xi_def}
 \nu =\nu(r)= \frac{r^2}{4} \kappa + r\left(1-\frac{\kappa}{4}\right)\quad\text{and}\quad \xi =\xi(r)= \frac{r^2}{8} \kappa.
\end{equation}
(In the notation of \cite{lawler2009}, $a=2/\kappa$.)  Then we have that \cite[Proposition~6.1]{lawler2009}:
\begin{equation}
\label{eqn::mg_def}
 M_t  = M_t(z) = |Z_t|^r Y_t^\xi \Delta_t^\nu = S_t^{-r} \Upsilon_t^{\xi+r} \Delta_t^{\nu+r}
\end{equation}
is a local martingale.  This martingale also appears in \cite[Theorem~6]{\SW}, though it is expressed there in a slightly different form.  (The martingale in \eqref{eqn::martingalebetweensles} is of the same type, though there we have not included the interior force points.)  For each $\eps > 0$ and $R > 0$, we let
\begin{equation}
\label{eqn::tau_sigm_def}
\begin{split}
\tau_{\eps} &=\inf\{t\ge 0: \Upsilon_t= \tfrac{1}{2}\eps\} = \inf\{t \geq 0: \confrad(z;\h_t) = \epsilon\} \quad\text{and}\\
\sigma_R &=\inf\{t\ge 0: |\eta(t)|=R\}.
\end{split}
\end{equation}

\begin{lemma}
\label{lem::sle_cr_martingale}
Fix $r<\frac{1}{2}-\frac{4}{\kappa}$, $\delta \in (0,\tfrac{\pi}{2})$, and $z \in \h$ such that $\arg(z) \in (\delta,\pi-\delta)$.  Let $\p^\star$ be the law of $\eta$ weighted by $M$.  We have that,
\begin{equation}
\label{eqn::radial_finite}
\p^\star[\tau_{\eps}<\infty]=1
\end{equation}
and
\begin{equation}
\label{eqn::radial_sin_finite}
\E^\star[S_{\tau_{\eps}}^r]\asymp 1
\end{equation}
where the constants depend only on $\delta$, $\kappa$, and $r$.  We also have that
\begin{equation}
\label{eqn::radial_angle}
\p^\star[\Theta_{\tau_{\eps}}\in (\delta,\pi-\delta)]\asymp 1
\end{equation}
where constants depend only on $\delta$, $\kappa$, and $r$.  Finally, we have that
\begin{equation}
\label{eqn::radial_exiting_small}
\p^\star[\sigma_R\le \tau_{\eps}]\to 0 \quad\text{as}\quad R\to \infty
\end{equation}
uniformly over $\eps > 0$.
\end{lemma}
\begin{proof}
Note that \eqref{eqn::radial_finite} and \eqref{eqn::radial_sin_finite} are proved in \cite[Equation (6.9)]{lawler2009}, so we will not repeat the arguments here.  Following \cite{lawler2009}, we define the radial parametrization (i.e., by $\log$ conformal radius) $u(t)$ by
\[ \wh{\Upsilon}_t=\Upsilon_{u(t)}=e^{-4t/\kappa}\]
and write $\wh{\eta}(t)=\eta(u(t))$ and $\wh{\Theta}_t=\Theta_{u(t)}$. Then $\wh{\Theta}_t$ satisfies the SDE (see \cite[Section 6.3]{lawler2009})
\begin{equation}
\label{eqn::evolution_radial_parameter}
d\wh{\Theta}_t=\left(1-\frac{4}{\kappa}-r\right)\cot\big(\wh{\Theta}_t\big) dt+ d\wh{W}_t
\end{equation}
where $\wh{W}$ is a $\p^\star$-Brownian motion.  The process $\wh{\Theta}$ almost surely does not hit $\{0,\pi\}$ (see \cite[Lemma 1.27]{lawler2005}) and the density with respect to Lebesgue measure on $[0,\pi]$ for the stationary distribution for \eqref{eqn::evolution_radial_parameter} is given by
\[ f(\theta)=c(\sin\theta)^{2\left(1-\tfrac{4}{\kappa}-r\right)}\]
where $c > 0$ is a normalizing constant (see \cite[Lemma~1.28]{lawler2005}).  Moreover, as $t \to \infty$, the law of $\wh{\Theta}_t$ converges to the stationary distribution with respect to the total variation norm.

We can use this to extract \eqref{eqn::radial_angle} as follows.  Fix $0 < T < \infty$.  We first note that by the Girsanov theorem the law of $\wh{\Theta}|_{[0,T]}$ stopped upon leaving $(\tfrac{\delta}{2},\pi-\tfrac{\delta}{2})$ is mutually absolutely continuous with respect to that of $B|_{[0,T]}$ where $B$ is a Brownian motion starting from $\wh{\Theta}_0$, also stopped upon leaving $(\tfrac{\delta}{2},\pi-\tfrac{\delta}{2})$.  Fix $ 0 \leq t \leq T$.  Then a Brownian motion starting from $\wh{\Theta}_0 \in [\delta,\pi-\delta]$ has a uniformly positive chance of staying in $(\tfrac{\delta}{2},\pi-\tfrac{\delta}{2})$ during the time interval $[0,t]$ and then being in $(\delta,\pi-\delta)$ at time $t$.  Therefore it is easy to see that \eqref{eqn::radial_angle} holds for all $0 \leq t \leq T$.

The lower bound, however, that comes from this estimate decays as $T$ increases.  We are now going to explain how we make our choice of $T$ as well as get a uniform lower bound for $t \geq T$.  We suppose that $\wh{\Theta}^1,\wh{\Theta}^2$ are solutions of \eqref{eqn::evolution_radial_parameter} where $\wh{\Theta}_0^1 = \delta$ and $\wh{\Theta}_0^2 = \pi-\delta$.  We assume further that the Brownian motions driving $\wh{\Theta}$, $\wh{\Theta}^1$, and $\wh{\Theta}^2$ are independent of each other until the time that any two of the processes meet, after which we take the Brownian motions for the pair to be the same.  This gives us a coupling $(\wh{\Theta}^1,\wh{\Theta},\wh{\Theta}^2)$ such that $\wh{\Theta}_t^1 \leq \wh{\Theta}_t \leq \wh{\Theta}_t^2$ for all $t \geq 0$ almost surely.  Note that after $\wh{\Theta}^1$ first hits $\wh{\Theta}^2$, all three processes stay together and never separate.  Let $q_\delta > 0$ be the mass that the stationary distribution puts on $(\delta,\pi-\delta)$.  We then take $T > 0$ sufficiently large so that:
\begin{enumerate}
\item For all $t \geq T$, the total variation distance between the law of $\wh{\Theta}_t^1$ and the stationary distribution is at most $\tfrac{q_\delta}{2}$.
\item Let $\xi = \inf\{t \geq 0 : \wh{\Theta}_t^1 = \wh{\Theta}_t^2\}$.  Then $\p[\xi \geq T] \leq \tfrac{q_\delta}{4}$.
\end{enumerate}
With this particular choice of $T$, we have that
\begin{align*}
  \p^\star[ \wh{\Theta}_t \in (\delta,\pi-&\delta)]
 \geq \p^\star[ \wh{\Theta}_t^1 \in (\delta,\pi-\delta)] - \p^\star[ \xi \geq T]\\
\geq& \frac{q_\delta}{2} - \frac{q_\delta}{4}
= \frac{q_\delta}{4} \quad\text{for all}\quad t \geq T.
\end{align*}
This proves \eqref{eqn::radial_angle}.

For \eqref{eqn::radial_exiting_small}, note that, under $\p^\star$, $\wh{\eta}$ has the same law as a radial $\SLE_{\kappa}(\rho)$ in $\h$ from $0$ to $z$ with a single boundary force point located at $\infty$ of weight $\rho =\kappa-6-r\kappa\ge \frac{\kappa}{2}-2$ (see \cite[Theorem~3 and Theorem~6]{MR2188260}). Define $\wh{\sigma}_R=\inf\{t\ge 0: |\wh{\eta}(t)|=R\}$.  Then
\[\p^\star[\sigma_R<\tau_{\eps}]\le \p^\star[\wh{\sigma}_R<\infty].\]
The endpoint continuity of the radial $\SLE_\kappa(\rho)$ processes with $\rho > -2$ \cite[Theorem~1.12]{IG4} implies that $\p^\star[\wh{\sigma}_R < \infty] \to 0$ as $R \to \infty$, as desired.
\end{proof}

We are now going to use Lemma~\ref{lem::sle_cr_martingale} to estimate the moments of $g_t'(z)$ at times when $\eta$ is close to $z$.  We will actually prove this for general $\SLE_\kappa(\ul{\rho})$ processes which is why we truncate on various events in the estimates proved below.

\begin{lemma}
\label{lem::derivative_hit_before_exit}
Fix $r < \tfrac{1}{2}-\tfrac{4}{\kappa}$ and $\delta \in (0,\tfrac{\pi}{2})$.  There exists $R_0 = R_0(r) > 0$ such that for all $R \geq R_0$ the following holds.  Suppose $\eta \sim \SLE_\kappa(\ul{\rho})$ in $\h$ from $0$ to $\infty$ where the force points lie outside of $2R \D$.  Fix $z \in \D \cap \h$ with $\arg(z) \in (\delta,\pi-\delta)$.  For each $\eps > 0$ and $R > 0$ we let $\tau_\eps$ and $\sigma_R$ be as in \eqref{eqn::tau_sigm_def}.  Then
\begin{equation}
\label{eqn::loewner_deriv_upper_bound_theta}
\E\left[\left|g'_{\tau_{\eps}}(z)\right|^{\nu+r}\one_{\{\tau_{\eps}<\sigma_R\}}\right]\asymp \eps^{-\xi-r} \quad\text{provided}\quad \confrad(z;\h) \geq \eps
\end{equation}
where the constants depend only on $\delta$, $\kappa$, and the weights $\ul{\rho}$ of the force points.  Fix a constant $C > 1$ and suppose that $\zeta_\eps$ is a stopping time for $\eta$ such that $\tau_{C \eps} \leq \zeta_\eps \leq \tau_{\eps/C}$.
Let
\begin{align}
 E_{\eps,R}^\delta &= \{\zeta_\eps < \sigma_R,\ \Theta_{\zeta_\eps} \in (\delta,\pi-\delta)\}.
\end{align}
Then we have that
\begin{align}
 \E\left[ \left|g_{\zeta_\eps}'(z)\right|^{\nu+r} \one_{E_{\eps,R}^\delta} \right] &\asymp \eps^{-\xi-r} \quad\text{provided}\quad \confrad(z;\h) \geq \eps \label{eqn::loewner_deriv_bound_theta}
\end{align}
where the constants depend only on $C$, $\delta$, $\kappa$, and the weights $\ul{\rho}$ of the force points.
\end{lemma}
\begin{proof}
It suffices to prove the result for an ordinary $\SLE_\kappa$ process since it is clear from the form of \eqref{eqn::martingalebetweensles} that the Radon-Nikodym derivative between the law of an $\SLE_\kappa$ and an $\SLE_\kappa(\ul{\rho})$ process whose force points lie outside of $2R\D$ stopped at time $\sigma_R$ is bounded from above and below by finite and positive constants which depend only on the total (absolute) weight of the force points and $\kappa$.

We are now going to prove the upper bound of \eqref{eqn::loewner_deriv_upper_bound_theta} and the lower bound of \eqref{eqn::loewner_deriv_bound_theta} with $\tau_\eps = \zeta_\eps$.  We have that,
\begin{align*}
\E\left[\left|g'_{\tau_{\eps}}(z)\right|^{\nu+r}\one_{\{\tau_{\eps}<\sigma_R\}}\right]&\le \E\left[\left|g'_{\tau_{\eps}}(z)\right|^{\nu+r}\one_{\{\tau_{\eps}<\infty\}}\right]\\
&\asymp \eps^{-\xi-r}\E[M_{\tau_{\eps}}S^r_{\tau_{\eps}}\one_{\{\tau_{\eps}<\infty\}}]\\
&=\eps^{-\xi-r}M_0\E^\star[S^r_{\tau_{\eps}}]\\
&\lesssim \eps^{-\xi-r}\quad\quad(\text{by \eqref{eqn::radial_sin_finite}}).
\end{align*}
This proves the upper bound of \eqref{eqn::loewner_deriv_upper_bound_theta}.  For the lower bound, we compute
\begin{align*}
\E\left[ \left|g_{\tau_\eps}'(z)\right|^{\nu+r} \one_{E_{\eps,R}^\delta} \right]&\asymp \eps^{-\xi-r}\E\left[M_{\tau_{\eps}}S^r_{\tau_{\eps}}\one_{E_{\eps,R}^\delta}\right]\\
&\ge \eps^{-\xi-r}\E\left[M_{\tau_{\eps}}\one_{E_{\eps,R}^\delta}\right]\\
&=\eps^{-\xi-r}M_0\p^\star[E_{\eps,R}^\delta].
\end{align*}
To bound $\p^\star[E_{\eps,R}^\delta]$, we have
\begin{align*}
\p^\star[E_{\eps,R}^\delta]
&=\p^\star[\tau_\eps < \sigma_R,\ \Theta_{\tau_\eps} \in (\delta,\pi-\delta)]\\
&\geq \p^\star[\Theta_{\tau_\eps} \in (\delta,\pi-\delta)]-\p^\star[\sigma_R<\tau_{\eps}].
\end{align*}
From \eqref{eqn::radial_angle}, we know that $\p^\star[\Theta_{\tau_\eps} \in (\delta,\pi-\delta)]$ is bounded from below uniformly in $\epsilon > 0$. From \eqref{eqn::radial_exiting_small}, we know that $\p^\star[\sigma_R<\tau_{\eps}]$ converges to zero as $R\to\infty$ uniformly over $\eps > 0$. These show that $\p^\star[E_{\eps,R}^\delta]$ is bounded from below which proves the lower bound for \eqref{eqn::loewner_deriv_bound_theta}.  The upper bound in the case that we replace $\tau_\epsilon$ with $\zeta_\epsilon$ is proved similarly.  For the lower bound, it is not difficult to see that
\[ \p^\star[ \Theta_t \in (\delta,\pi-\delta) \text{ for all } t \in [\tau_{C \eps},\tau_{\eps/C}] \giv \Theta_{\tau_{C \eps}} \in (\delta,\pi-\delta)] > 0\]
uniformly in $\eps > 0$ and
\[ \p^\star[ \sigma_R \leq \zeta_\eps] \leq \p^\star[ \sigma_R \leq \tau_{\eps/C}] \to 0 \quad\text{as}\quad R \to \infty\]
uniformly in $\eps >0$. 
\end{proof}

\subsection{Hitting probabilities}
\label{subsec::hitting}

\begin{figure}[ht!]
\begin{center}
\includegraphics[scale=0.85]{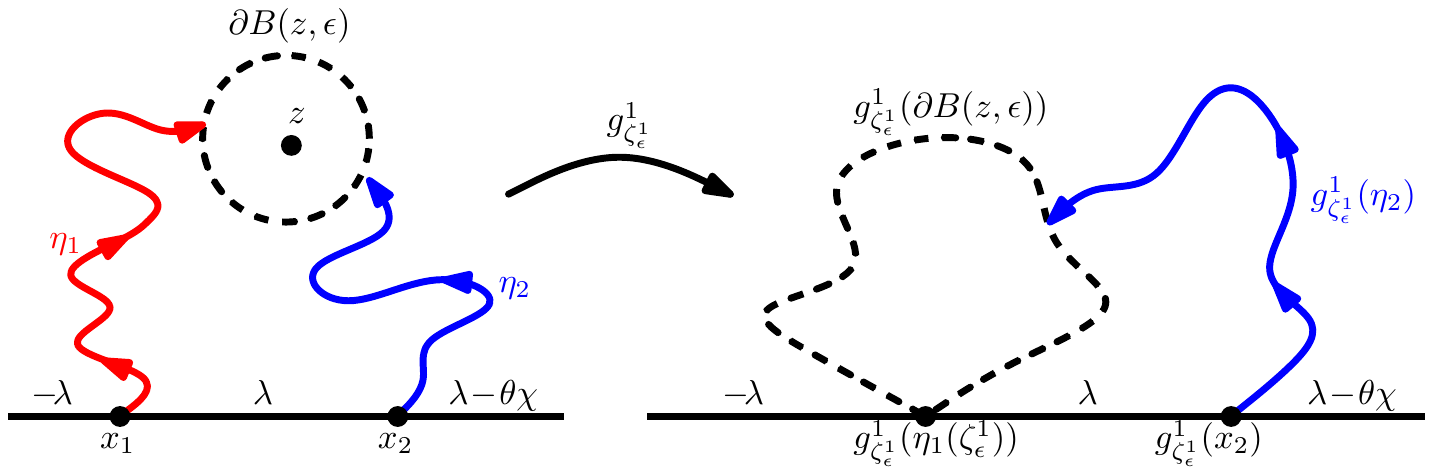}
\end{center}
\caption{\label{fig::flow_line_intersection_one_point}
Illustration of the setup of Lemma~\ref{lem::one_point_both}, the one point estimate for the intersection dimension.  On the left side, $\eta_1$ (resp.\ $\eta_2$) is a flow line of a GFF on $\h$ with the indicated boundary data with angle $0$ (resp.\ $\theta \in (\pi-2\lambda/\chi,0)$) starting from $x_1$ (resp.\ $x_2 > x_1$).  Note that $\eta_1$ (resp.\ $\eta_2$) is an $\SLE_\kappa(-\theta \chi/\lambda)$ (resp.\ $\SLE_\kappa(2,-\theta \chi/\lambda -2)$) process.  The force point for $\eta_1$ is located at $x_2$ and the force points for $\eta_2$ are located at $x_1$ and $x_2^-$.  By Figure~\ref{fig::conditional_law}, the conditional law of $\eta_2$ given $\eta_1$ drawn up to any stopping time is also an $\SLE_\kappa(2,-\theta \chi/\lambda-2)$ process.  Shown is the event $G^{\delta}_\eps(z)$ that $\eta_1$ hits $\partial B(z,\eps)$, say for the first time at $\zeta_\eps^1$, before exiting $B(0,R_0)$ where $R_0 > 0$ is a large, fixed constant, the harmonic measure of the left (resp.\ right) side of $\eta_1$ stopped upon hitting $\partial B(z,\epsilon)$ is not too small, and that $\eta_2$ also hits $\partial B(z,\eps)$.  We estimate the probability of $G^{\delta}_\eps(z)$ by combining Lemma~\ref{lem::derivative_hit_before_exit} with Theorem~\ref{thm::one_point_exponent_general}.}
\end{figure}

\begin{figure}[ht!]
\begin{center}
\includegraphics[scale=0.85]{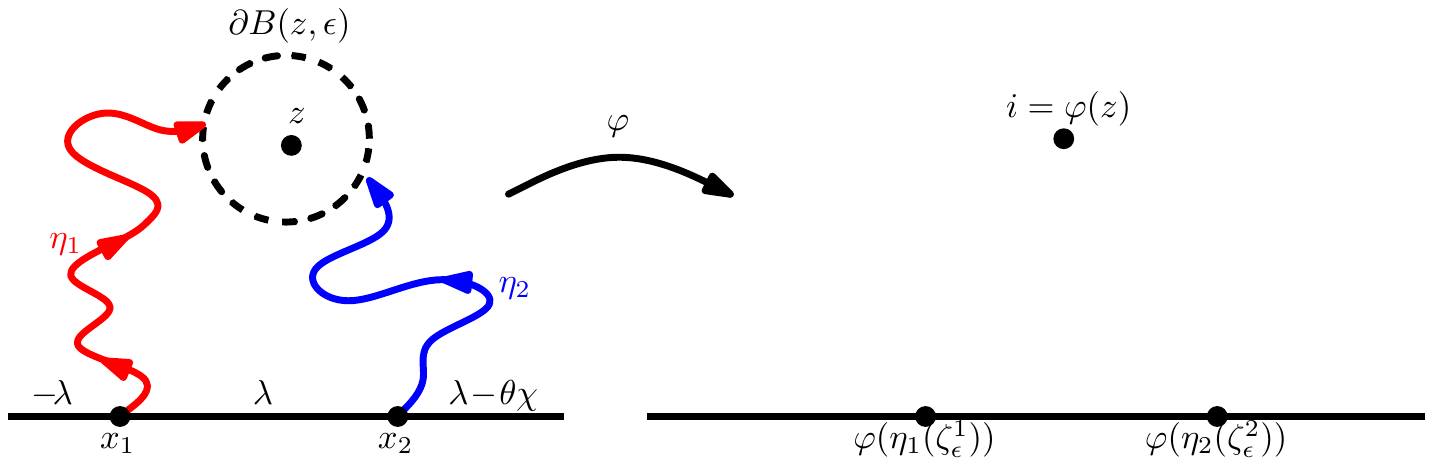}
\end{center}
\caption{\label{fig::flow_line_intersection_one_point2} (Continuation of Figure~\ref{fig::flow_line_intersection_one_point}.) Let $\zeta_\epsilon^1,\zeta_\epsilon^2$ be the times that $\eta_1,\eta_2$ hit $\partial B(z,\eps)$, respectively, and let $\varphi$ be the unique conformal map that uniformizes the unbounded connected component of $\h \setminus (\eta_1([0,\zeta_\epsilon^1]) \cup \eta_2([0,\zeta_\epsilon^2]))$ with $z$ sent to $i$ and $\infty$ fixed.  For the lower bound of Theorem~\ref{thm::two_flowline_dimension}, we will also need to estimate the probability of the event $H^{\delta}_\eps(z)$ that $G^{\delta}_\eps(z)$ occurs (as described in Figure~\ref{fig::flow_line_intersection_one_point}), that the diameter of $\eta_2([0,\zeta_\eps^2])$ is not too large, and that the images of $\eta_i(\zeta_\epsilon^i)$ for $i=1,2$ under $\varphi$ are not too far from $i$ as illustrated on the right.}
\end{figure}

Fix an angle $\theta \in (\pi-2\lambda/\chi,0)$.  This is the range so that GFF flow lines with angles $0,\theta$ are able to intersect each other where the flow line with angle $\theta$ stays to the right of the flow line with angle $0$ \cite[Theorem~1.5]{IG1}.  Let
\begin{equation}
\label{eqn::two_flow_line_exponent}
A = \frac{1}{2\kappa} \left(\rho+\frac{\kappa}{2} + 2\right)\left( \rho-\frac{\kappa}{2}+6 \right) \quad\text{where}\quad \rho = -\frac{\theta\chi}{\lambda} - 2.
\end{equation}

\begin{lemma}
\label{lem::one_point_both}
Fix $C > 2$, let $x_1 = 0$, and fix $x_2 \geq 2R_0$ where $R_0$ is the constant from Lemma~\ref{lem::derivative_hit_before_exit} with
\[ r = -\frac{2}{\kappa}\left(\rho+6-\frac{\kappa}{2} \right).\]
Let $h$ be a GFF on $\h$ with boundary data as illustrated in Figure~\ref{fig::flow_line_intersection_one_point}.  That is,
\begin{equation}
\label{eqn::h_bd_one_point}
 h|_{(-\infty,x_1)} \equiv -\lambda,\quad h|_{[x_1,x_2]} \equiv \lambda,\quad\text{and}\quad h|_{(x_2,\infty)} \equiv \lambda - \theta \chi.
\end{equation}
Let $\eta_1$ (resp.\ $\eta_2$) be the flow line of $h$ starting from $x_1$ (resp.\ $x_2$) with angle $0$ (resp.\ $\theta$).  Fix $\delta  \in (0,\tfrac{\pi}{2})$ and let $z \in \D \cap \h$ with $\arg(z) \in (\delta,\pi-\delta)$.  For $i=1,2$, let $\zeta_\epsilon^i$ be the first time that $\eta_i$ hits $\partial B(z,\epsilon)$ and let $\Theta_t^1$ be the process as in Lemma~\ref{lem::derivative_hit_before_exit} for $\eta_1$.
\begin{enumerate}[(i)]
\item Let $G^{\delta}_\eps(z)$ be the event that $\eta_1$ hits $\partial B(z,\eps)$ before hitting $\partial B(0,R_0)$, $\Theta_{\zeta_\eps^1}^1 \in (\delta,\pi-\delta)$, and that $\eta_2$ hits $\partial B(z,\eps)$.  Then we have that
\begin{align}
\label{eqn::one_point_both}
 \p[ G^{\delta}_\eps(z)] =\eps^{A+o(1)}
\end{align}
where the $o(1)$ term depends only on $\delta$, $\kappa$, $\theta$, and $x_2$.
\item On $G^{\delta}_\eps(z)$, let $\varphi$ be the unique conformal map which takes the unbounded connected component of $\h \setminus ( \eta_1([0,\zeta_\eps^1]) \cup \eta_2([0,\zeta_\eps^2]))$ to $\h$ sending $z$ to $i$ and fixing $\infty$.  There exists a constant $R_1 > 0$ such that with
\[ H^{\delta}_{\eps}(z) = G^{\delta}_\eps(z) \cap \{ \max_{i=1,2} |\varphi(\eta_i(\zeta_\eps^i))| \leq R_1,\ \ \eta_2([0,\zeta_\eps^2]) \subseteq B(0,10x_2)\}\]
we have that
\begin{align}
\label{eqn::one_point_both_good}
\p[ H^{\delta}_{\eps}(z)] \gtrsim \eps^A
\end{align}
where the constants depend only on $\delta$, $\kappa$, $\theta$, and $x_2$.
\end{enumerate}
The same likewise holds if $h$ is a GFF on $\h$ with piecewise constant boundary conditions which change values a finite number of times and in the interval $[-20 x_2,20 x_2]$ takes the form in \eqref{eqn::h_bd_one_point}.  In this case, the constants also depend on $\|h|_{\R}\|_\infty$.
\end{lemma}
\begin{proof}
For each $t \geq 0$, let $\h_t^1$ be the unbounded connected component of $\h \setminus \eta_1([0,t])$, let $\tau_\eps^1 = \inf\{t \geq 0 : \confrad(z;\h_t^1) = \eps\}$, $\sigma_{R_0}^1 = \inf\{t \geq 0: \eta_1(t) \notin B(0,R_0)\}$, and let $(g_t^1)$ be the Loewner evolution associated with $\eta_1$.  By \eqref{eqn::growth}, note that $\tau_{4\eps}^1 \leq \zeta_\eps^1$.  It then follows from Theorem~\ref{thm::one_point_exponent_general} that
\begin{align*}
\p[G^{\delta}_{\eps}(z) \giv \eta_1|_{[0,\tau_{4\eps}^1]}]\leq |(g^1_{\tau_{4\eps}^1})'(z)\eps|^{\alpha+o(1)}.
\end{align*}
Note that $r < 1-\tfrac{8}{\kappa}<\tfrac{1}{2}-\tfrac{4}{\kappa}$ since $\rho > -2$ and $\kappa \in (0,4)$.  With this choice of $r$, we have
\[ \nu + r = \alpha \quad\text{and}\quad \nu-\xi=A.\]
Thus, by \eqref{eqn::loewner_deriv_upper_bound_theta} of Lemma~\ref{lem::derivative_hit_before_exit}, we have that
\[\p[G^{\delta}_{\eps}(z)] \le  \E\left[ |(g^1_{\tau_{4\eps}^1})'(z)\eps|^{\alpha+o(1)} \one_{\{\tau_{4\eps}^1 \leq \sigma_{R_0}^1\}} \right] \le \eps^{A+o(1)}.\]
This gives the upper bound for \eqref{eqn::one_point_both}.

Let $E_{\eps,R_0}^\delta = \{ \zeta_\eps^1 < \sigma_{R_0}^1,\ \Theta_{\zeta_\eps^1}^1 \in (\delta,\pi-\delta)\}$.  On $E_{\eps,R_0}^\delta$ and $\{\zeta_\eps^2 < \infty\}$, we let $w_\eps = g_{\zeta_\eps^1}^1(\eta_2(\zeta_\eps^2))$ and $r_\eps = |(g^1_{\zeta_\eps^1})'(z)|\eps$.  From Lemma~\ref{lem::lower_bound_one_point_exponent}, we have that
\begin{align*}
\p\left[G_{\eps}^{\delta}(z) \giv\eta_1|_{[0,\zeta_{\eps}^1]} \right]\one_{E_{\eps,R_0}^{\delta}}\gtrsim r_\eps^{\alpha}\one_{E_{\eps,R_0}^{\delta}}.
\end{align*}
We see from \eqref{eqn::loewner_deriv_bound_theta} of Lemma~\ref{lem::derivative_hit_before_exit} that
$\p[ G^{\delta}_\eps(z)] \gtrsim \eps^A$.

We will now explain how to prove the result for $H^{\delta}_{\eps}(z)$ in place of $G^{\delta}_{\eps}(z)$.  First of all, we note that on $E_{\eps,R_0}^\delta$, it follows from \cite[Corollary~3.44]{lawler2005} that $|g_{\zeta_\eps^1}^1(w) - w| \leq 3 R_0$ for all $w\in \h^1_{\zeta^1_\eps}$.  Consequently,
\begin{equation}
\label{eqn::img_ball_contain}
 B(g_{\zeta_\eps^1}^1(z), 10 x_2 - 6 R_0) \subseteq g_{\zeta_\eps^1}^1(B(z, 10 x_2));
 \end{equation}
recall that $10 x_2 \geq 20 R_0$.  By Lemma~\ref{lem::lower_bound_one_point_exponent} and \eqref{eqn::img_ball_contain}, we have that,
\begin{align*}
 &\p\left[\zeta_\eps^2 < \infty,\ \eta_2([0,\zeta_\eps^2]) \subseteq B(z, 10x_2),\ \im(w_\eps)\ge \delta r_\eps \giv \eta_1|_{[0,\zeta_\eps^1]} \right]\one_{E^{\delta}_{\eps,R_0}}\\
&\quad\quad\quad\quad\gtrsim r_\eps^{\alpha}\one_{E^{\delta}_{\eps,R_0}}.
\end{align*}
On the event in the probability above, a Brownian motion starting from $z$ has a uniformly positive chance (depending on $\delta$) of hitting both the left side of $\eta_1([0,\zeta_\eps^1])$ and right side of $\eta_2([0,\zeta_\eps^2])$.  Consequently, the desired result follows by applying \eqref{eqn::loewner_deriv_bound_theta} from Lemma~\ref{lem::derivative_hit_before_exit}.

The final claim of the lemma follows from \eqref{eqn::martingalebetweensles} to compare the case with extra force points to the case without considered above.
\end{proof}

In order for Lemma~\ref{lem::one_point_both} to be useful, we need that as $\eta_1$ gets progressively closer to a given point $z$, it is unlikely that $\Theta^1 \notin (\delta,\pi-\delta)$ for some $\delta > 0$.  This is the purpose of the following estimate.

\begin{lemma}
\label{lem::hit_with_bad_angle}
Suppose that $\eta$ is an $\SLE_\kappa$ process in $\h$ from $0$ to $\infty$ with $\kappa \in (0,4)$.  Fix $z \in \h$ and let $n_z = -\log_2 \im(z)$ so that $n \geq n_z$ implies that $B(z,2^{-n}) \subseteq \h$.  Let $\Theta$ be the process as in \eqref{eqn::processes}.  For each $n$, let $\zeta_n$ be the first time that $\eta$ hits $\partial B(z,2^{-n})$ and, for each $\delta \in (0,\tfrac{\pi}{2})$, let $E_n^\delta = \{\zeta_n < \infty,\ \Theta_{\zeta_n} \notin (\delta,\pi-\delta)\}$.  There exists a function $p \colon (0,1) \to [0,1]$ with $p \downarrow 0$ as $\delta \downarrow 0$ such that for each $r \geq n_z$ we have that
\[ \p[ \cap_{m=n}^r E_m^\delta] \leq (p(\delta))^{r-n} \quad\text{for all}\quad n_z \leq n \leq r.\]
\end{lemma}
\begin{proof}
Since the $\SLE_\kappa$ processes are scale-invariant in law, almost surely transient, and do not intersect the boundary for $\kappa \in (0,4)$ \cite{\RohdeSchramm}, it follows that
\[ \lim_{s \to \infty} \p\big[ \eta \text{ hits } [s,s+2] \times [0,2] \big] = \lim_{s \to \infty} \p\big[ \eta \text{ hits } [1,1+\tfrac{2}{s}] \times [0,\tfrac{2}{s}] \big] = 0.\]
(For otherwise $\eta$ would intersect the boundary with positive probability.)  Consequently, it follows that there exists a function $q \colon (0,1) \to [0,1]$ with $q(\delta) \downarrow 0$ as $\delta \downarrow 0$ such that the following is true.  If $z \in \h$ with $\im(z) = 1$ and $\arg(z) \notin (\delta,\pi-\delta)$, then
\begin{equation}
\label{eqn::hit_ball_decay_theta}
\p[\eta\text{ hits } B(z,1)] \leq q(\delta).
\end{equation}
For each $n \geq n_z$, on the event $\{\zeta_n < \infty\}$, let $\varphi_n \colon \h \setminus \eta([0,\zeta_n]) \to \h$ be the unique conformal map with $\varphi_n(\eta(\zeta_n)) = 0$, $\varphi_n(\infty) = \infty$, and satisfies $\im(\varphi_n(z)) = 1$.  Note that $\varphi_n(B(z,2^{-n-3)})) \subseteq B(\varphi_n(z),1)$ by \cite[Corollary~3.25]{lawler2005}.  Therefore it follows from \eqref{eqn::hit_ball_decay_theta} that
\begin{equation}
\label{eqn::hit_ball_iterate_decay}
 \p[ E_{n+3}^\delta \giv \eta|_{[0,\zeta_{n}]}] \one_{E_n^\delta} \leq q(\delta) \one_{E_n^\delta}.
\end{equation}
Iterating \eqref{eqn::hit_ball_iterate_decay} and taking $p(\delta) = (q(\delta))^{1/3}$ proves the lemma.
\end{proof}

\begin{figure}[ht!]
\begin{center}
\includegraphics[scale=0.85]{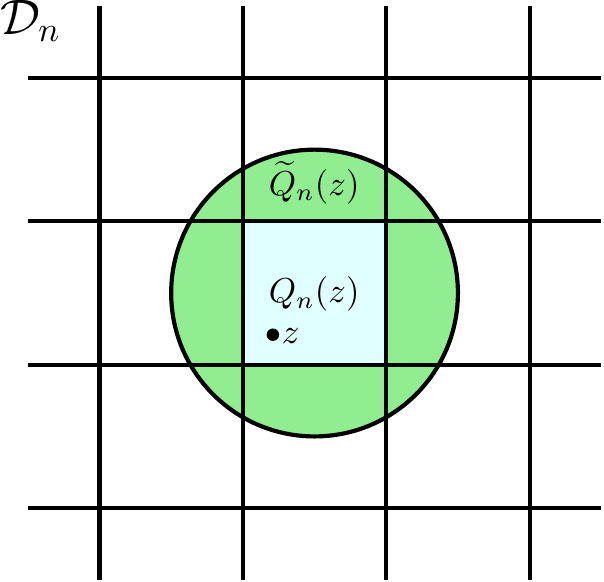}
\end{center}
\caption{
\label{fig::grids}
Shown in the illustration are $Q_n(z)$ and $\wt{Q}_n(z)$ for a given point $z \in \h$.
}
\end{figure}

For each $n \in \N$, we let $\CD_n$ be the set of squares with side length $2^{-n}$ which are contained in $\h$ and with corners in $2^{-n} \Z^2$.  For each $Q \in \CD_n$, let $z(Q)$ be the center of $Q$ and let $\wt{Q}_n(Q) = B(z(Q),2^{1-n})$.  For each $z \in \h$, let $Q_n(z)$ be the element of $\CD_n$ which contains $z$ and let $\wt{Q}_n(z) = \wt{Q}_n(Q_n(z))$.  See Figure~\ref{fig::grids} for an illustration.

\begin{lemma}
\label{lem::no_bad_points}
Suppose that $\eta$ is an $\SLE_\kappa$ process in $\h$ from $0$ to $\infty$ with $\kappa \in (0,4)$.  For each $z \in \h$, let $\Theta^z$ be the process from~\eqref{eqn::processes} (with respect to $z$) and let $\zeta_{z,n} = \inf\{t \geq 0: \eta(t) \in \partial \wt{Q}_n(z)\}$.  Let $\CS_n^\delta$ be the set of points $z \in \h$ such that $E_{z,n}^\delta = \{\zeta_{z,n} < \infty,\ \Theta_{\zeta_{z,n}}^z \notin (\delta,\pi-\delta)\}$ occurs and let $\CS^\delta = \cup_{n=1}^\infty \cap_{m = n}^\infty \CS_m^\delta$.  There exists $\delta_0 > 0$ such that for every $\delta \in (0,\delta_0)$ we have that $\CS^\delta = \emptyset$ almost surely.
\end{lemma}
\begin{proof}
Fix $z \in \h$ and let $n_z = -\log_2 \im(z)$.  Note that $\wt{Q}_n(z) \subseteq B(z,2^{2-n})$ so that $\wt{Q}_n(z) \subseteq \h$ provided $n \geq n_z+2$.  By Lemma~\ref{lem::hit_with_bad_angle}, we have that
\begin{equation}
\label{eqn::bad_square_center}
\p[ \cap_{m=n}^r E_{z,m}^\delta] \leq \big(p(\delta) \big)^{r-n} \quad\text{for all}\quad n_z+2 \leq n \leq r
\end{equation}
(where $p(\delta)$ is as in the statement of Lemma~\ref{lem::hit_with_bad_angle}).

Suppose that $Q \in \CD_m$ and suppose that $n \in \N$ with $n \leq m$.  Then the function $Q \to \R$ given by $w \mapsto \Theta_{\zeta_{w,n}}^w$ is positive and harmonic.  Consequently, it follows from the Harnack inequality \cite[Proposition~2.26]{lawler2005} that there exists a universal constant $K \geq 1$ (independent of $m,n$) such that the following is true.  If $E_{w,m}^\delta$ occurs for any $w \in Q$, then $E_{z(Q),m}^{K \delta}$ occurs.  Thus letting $E_{Q,m}^\delta = \cup_{w \in Q} E_{w,m}^\delta$ we have that
\begin{align}
\label{eqn::bad_square}
 \p[ \cap_{m=n}^r E_{Q,m}^\delta] \leq \p[ \cap_{m=n}^r E_{z(Q),m}^{K\delta}]\quad&\text{for any}\quad n_{z(Q)}+2 \leq n \leq r.\\
\intertext{Combining this with Lemma~\ref{lem::hit_with_bad_angle} implies that}
\label{eqn::bad_square2}
 \p[ \cap_{m=n}^r E_{Q,m}^\delta] \leq (p(K\delta))^{r-n} \quad&\text{for any}\quad n_{z(Q)}+2 \leq n \leq r.
 \end{align}
Fix $\omega \in (0,1)$ and let $n = -\log_2 \omega$.  For each $r \geq n+2$, let $\CV_r^{\omega,\delta}$ be the collection of squares $Q$ in $\CD_r$ with $Q \subseteq \{z \in \h : |z| < \tfrac{1}{\omega},\ \im(z) \geq \omega\}$ and for which $\cap_{m=n}^r E_{Q,m}^\delta$ occurs.  Then \eqref{eqn::bad_square2} implies that there exists a constant $C > 0$ such that
\begin{equation}
\label{eqn::bad_square_summation}
 \sum_{r = n}^\infty \E\big[ |\CV_r^{\omega,\delta}| \big] \leq \frac{C}{\omega^2} \sum_{r = n}^\infty 2^{2 r} \big(p(K\delta)\big)^{r-n}.
\end{equation}
Take $\delta_0  > 0$ so that $\delta \in (0,\delta_0)$ implies that $4 p(K \delta) < 1$.  Then for $\delta \in (0,\delta_0)$, the summation on the right side of \eqref{eqn::bad_square_summation} is finite.  This implies that for every $\omega \in (0,1)$, $\CV_r^{\omega,\delta} = \emptyset$ for all but finitely many $r$ almost surely.  This, in turn, implies the desired result since $\omega \in (0,1)$ was arbitrary and $\CV_r^{\omega,\delta}$ increases as $\omega$ decreases.
\end{proof}

\subsection{The upper bound}
\label{subsec::intersection_upper_bound}

Now that we have established Lemma~\ref{lem::one_point_both} and Lemma~\ref{lem::no_bad_points}, we can prove the upper bound in Theorem~\ref{thm::two_flowline_dimension}.

\begin{proposition}
\label{prop::dim_intersection_upper_bound}
Suppose that $h$ is a GFF on $\h$ with piecewise constant boundary conditions which change values a finite number of times.  Let $\eta_1$ (resp.\ $\eta_2$) be the flow line of $h$ starting from $x_1=0$ (resp.\ $x_2 > 0$) with angle $0$ (resp.\ $\theta \in (\pi - 2\lambda/\chi,0)$).  We have that
\[ \dimH(\eta_1 \cap \eta_2 \cap \h) \leq 2-A \quad \text{ almost surely}\]
where $A$ is as in \eqref{eqn::two_flow_line_exponent}.
\end{proposition}
\begin{proof}
We are going to prove the proposition assuming that the boundary data is as in Lemma~\ref{lem::one_point_both}.  This suffices by absolute continuity for GFFs.  Fix $0 < \eps < \tfrac{\delta}{2} < \delta < \tfrac{\pi}{4}$.  For each $t > 0$, we let $\h_t^1$ be the unbounded connected component of $\h \setminus \eta_1([0,t])$.  For each $z \in \h$, we let $\zeta_{z,\eps}^1 = \inf\{t \geq 0 : \eta^1(t) \in \partial B(z,\epsilon)\}$ and let $\Theta^{1,z}$ be the process as in \eqref{eqn::processes} for $\eta_1$ and $z$.  We let $I^{\eps,\delta}$ consist of those $z \in \eta_1 \cap \eta_2 \cap B(0,\delta^{-1})$ such that
\begin{enumerate}[(i)]
\item $\im(z) \geq \delta$.
\item $\Theta^{1,z}_t\in (2\delta,\pi-2\delta)$ for all $t \in [\zeta_{z,\eps/2}^1,\zeta_{z,2\eps}^1]$.
\item Let $\zeta_z^1$ be the first time that $\eta_1$ hits $z$ and $\sigma_{z,\delta}^1$ be the first time after $\zeta_{z,\eps}^1$ that $\eta_1$ hits $\partial B(z, \delta)$.  Then $\zeta_z^1 \leq \sigma_{z,\delta}^1$.
\end{enumerate}
By the transience, continuity, and simplicity of the $\SLE_\kappa(\ul{\rho})$ processes for $\kappa \in (0,4)$ (which almost surely do not hit the continuation threshold) \cite[Theorem~1.3]{IG1}, we have that $\eta_1 \cap \eta_2 \cap \h \subseteq \cup_{\eps\in \Q_+} \cup_{\delta\in \Q_+}  I^{\eps,\delta}$ almost surely. (If this were not true then we would be led to the contradiction that $\eta_1$ has double points with positive probability.)  We are going to prove the result by showing that for every $\eps,\delta > 0$,
\[ \dimH(I^{\eps,\delta}) \leq 2-A \quad \text{almost surely}.\]
It in fact suffices to show that this is the case for $0 < \eps < \tfrac{\delta}{2} < \delta < \delta_0$ where $\delta_0$ is as in Lemma~\ref{lem::no_bad_points}.  Let $\CD_n$ and $z(Q)$ be as before the statement of Lemma~\ref{lem::no_bad_points}.  We let $\CU_n^{\eps,\delta}$ consist of those $Q \in \CD_n$ which are hit by both $\eta_1$ and $\eta_2$, contained in $B(0,\delta^{-1})$, and:
\begin{enumerate}[(i)]
\item $\im(z(Q)) \geq \delta$.
\item $\Theta_{\zeta_{z(Q),\eps}^1}^{1,z(Q)}\in (\delta,\pi-\delta)$ and $\Theta_{\zeta_{z(Q),2^{-n}}}^{1,z(Q)} \in (\delta,\pi-\delta)$.
\item After $\zeta_{z(Q),\eps}^1$, $\eta_1$ hits $Q$ before $\sigma_{z(Q),\delta}^1$.
\end{enumerate}
We are now going to show that, for every $n \in \N$, $\CW_n^{\eps,\delta} = \cup_{m \geq n} \CU_m^{\eps,\delta}$ is a cover of $I^{\eps,\delta}$.  To see this, we fix $z \in I^{\eps,\delta}$ and let $(Q_k)$ be a sequence of squares in $\cup_{m \geq n} \CD_m$ such that $z \in Q_k$ for every $k$ and $|Q_k| \to 0$ as $k \to \infty$.  Let $z_k = z(Q_k)$.  Since $\zeta_{z_k,\eps}^1 \in [\zeta_{z,\eps/2}^1,\zeta_{z,2\eps}^1]$ for all $k$ large enough, there exists $K_0 =K_0(z)$ such that for all $k\ge K_0$, we have that $\Theta^{1,z_k}_{\zeta_{z_k,\eps}^1}\in (\delta,\pi-\delta)$.  Since $z \in Q_k$, we have that $\eta_1$ hits $Q_k$. If there exists a subsequence $(k_j)$ such that, for every $j$, $\eta_1$ hits $\partial B(z_{k_j},\delta)$ after hitting $\partial B(z_{k_j},\eps)$ and before hitting $Q_{k_j}$, we get a contradiction that $z \in I^{\eps,\delta}$.  Therefore there exists $K_1 = K_1(z)$ such that for every $k \geq K_1$, we have that, after hitting $\partial B(z_k,\eps)$, $\eta_1$ hits $Q_k$ before hitting $\partial B(z_k, \delta)$. Combing this with Lemma~\ref{lem::no_bad_points} implies that there exists a sequence $(k_j)$ such that $Q_{k_j}\in \CW_n^{\eps,\delta}$ for all $j$, which proves our claim.

By running $\eta_1$ until time $\zeta_{z,\eps}^1$ and then conformally mapping back, Lemma~\ref{lem::one_point_both} implies for $Q \in \CD_m$ with $Q \subseteq B(0,\delta^{-1})$ and $\im(z(Q)) \geq \delta$ that $\p[Q \in \CU_m^{\eps,\delta}] \le 2^{-m (A+o(1))}$ provided $m$ is large enough and $\eps > 0$ is small enough relative to $\delta > 0$.  (The purpose of choosing $\eps > 0$ smaller than $\delta > 0$ is so that the force points of $\eta_1$ are mapped far away from $\eta_1(\zeta_{z,\eps}^1)$ relative to the distance of $z$.)  Consequently, it follows that there exists $C=C(\eps,\delta) > 0$ such that for each $\xi > 0$, we have
\[ \E[ \LH^{2-A+2\xi}(I^{\eps,\delta})] \leq C \sum_{m=n}^\infty 2^{2m} \times 2^{-m (A-\xi)} \times 2^{-m(2-A+2\xi)} < \infty.\]
Since the above holds for every $n$, we therefore have that $\LH^{2-A+2\xi}(I^{\eps,\delta}) = 0$ almost surely.  Since $\xi > 0$ was arbitrary, we have that $\dimH(I^{\eps,\delta}) \leq 2-A$ almost surely, as desired.
\end{proof}

\subsection{The lower bound}
\label{subsec::intersection_lower_bound}

\begin{figure}[ht!]
\begin{center}
\includegraphics[scale=0.85]{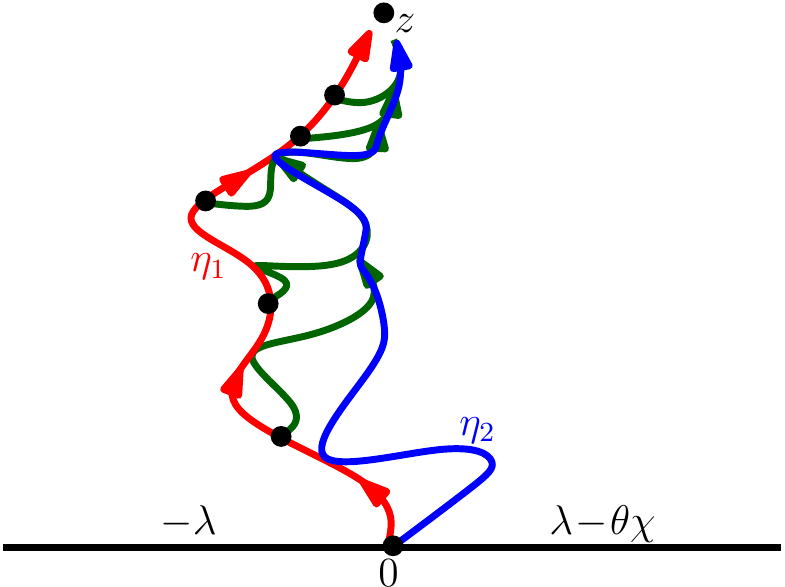}
\end{center}
\caption{\label{fig::perfect_illustration}
Suppose that $h$ is a GFF on $\h$ with the illustrated boundary data.  Let $\eta_1$ (resp.\ $\eta_2$) be the flow line of $h$ starting from $0$ with angle $0$ (resp.~$\theta \in (\pi-2\lambda/\chi,0)$).  Shown is an illustration of the construction of the event that a given point, say $z \in \h$, is a ``perfect point'' for the intersection of $\eta_1$ and $\eta_2$.  Each of the green flow lines has angle $\theta$ --- the same as that of $\eta_2$ --- and start at points along $\eta_1$ which get progressively closer to $z$.  The reason that we introduce the auxiliary green flow lines is that this is what gives us the approximate independence necessary for the two point estimate, see e.g.\ Figure~\ref{fig::two_point_approx_ind}.
}
\end{figure}

\begin{figure}[ht!]
\begin{center}
\includegraphics[scale=0.85]{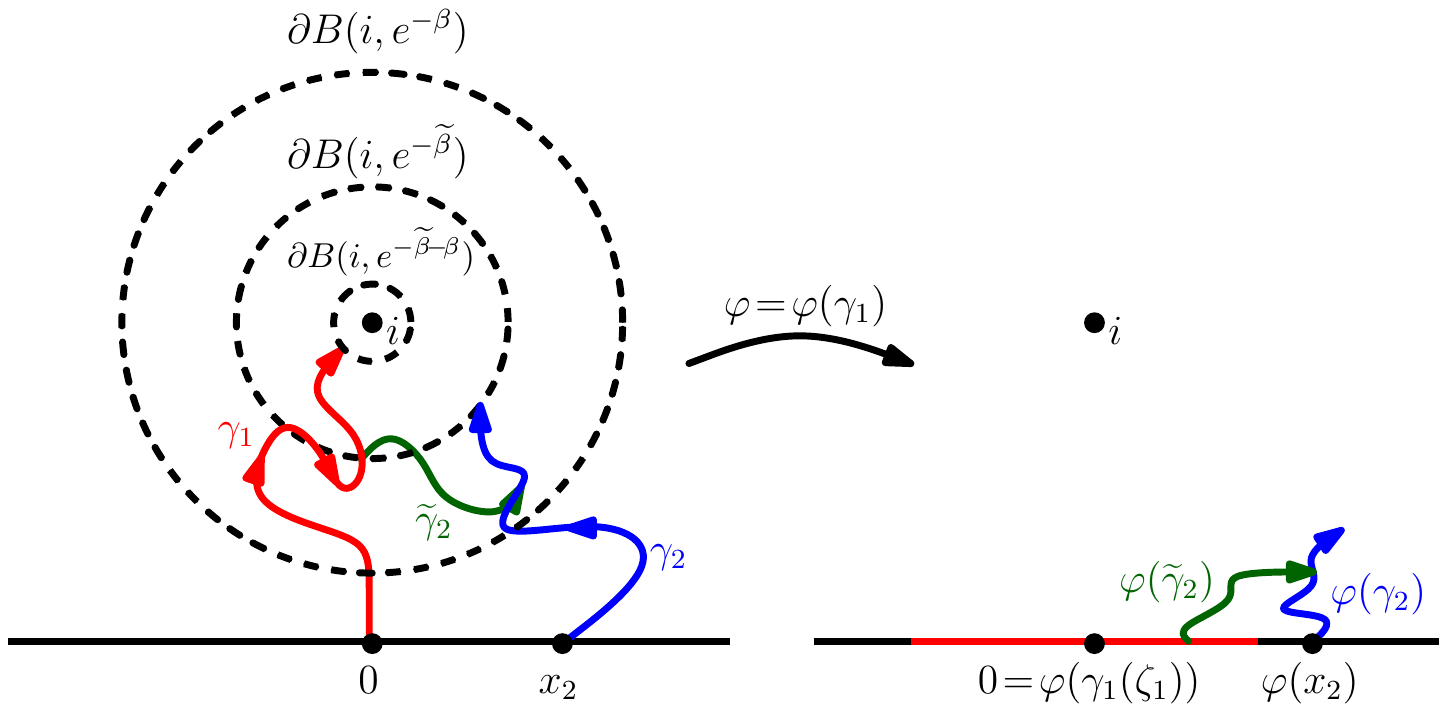}
\end{center}
\caption{\label{fig::perfect_def}
Suppose that $\gamma_1$, $\gamma_2$ are paths in $\h$ starting from $0,x_2 \in \R$, respectively, with $x_2 \in  [0,e^{\beta}]$.  Let $\wt{\zeta}_1$ be the first time that $\gamma_1$ hits $\partial B(i,e^{-\wt{\beta}})$ and let $\wt{\gamma}_2$ be a path starting from $\gamma_1(\wt{\zeta}_1)$.  Fix $u \in \R \setminus [0,x_2]$.  Then $E_{u}^{\beta,\wt{\beta}}(\gamma_1,\wt{\gamma}_2,\gamma_2)$ is the event that the following hold.  First, $\gamma_1$ hits $\partial B(i, e^{-\beta})$ before leaving the $e^{-2\beta}$ neighborhood of $[0,i]$.  Second, $\gamma_1$ (resp.\ $\gamma_2$) hits $\partial B(i,e^{-\wt{\beta}-\beta})$ (resp.\ $\partial B(i,e^{-\wt{\beta}})$) before leaving $B(i,e^{2\beta})$.  Let $\zeta_1,\zeta_2$ be the first hitting times for $\gamma_1$, $\gamma_2$, respectively, for these small circles.  Third, the first time $\wt{\zeta}_2$ that $\wt{\gamma}_2$ hits $\gamma_2$ is finite and $\wt{\gamma}_2([0,\wt{\zeta}_2])$ is disjoint from both $\partial B(i,\tfrac{1}{2} e^{-\wt{\beta}})$ and $\partial B(i,2e^{-\wt{\beta}})$.  Fourth, the three paths stopped at the aforementioned times do not separate $i$ from $u$.  Fifth, the probability that a Brownian motion starting from $i$ exits $\h \setminus (\gamma_1([0,\zeta_1]) \cup \wt{\gamma}_2([0,\wt{\zeta}_2]) \cup \gamma_2([0,\zeta_2]))$ in the left (resp.\ right) side of $\gamma_1$ is at least $\tfrac{1}{2} - e^{-\beta/4}$ and in the left (resp.\ right) side of $\wt{\gamma}_2([0,\wt{\zeta}_2])$ (resp.\ $\gamma_2([0,\zeta_2])$) is at least $e^{-\beta}$.  We take $H$ to be the connected component of $\h \setminus \gamma_1([0,\zeta_1])$ with $u$ on its boundary and let $\varphi = \varphi(\gamma_1)$ be the conformal transformation $H \to \h$ fixing $i$ and with $\varphi(\gamma_1(\zeta_1)) = 0$.  Then the image of (the right side of) $\gamma_1(\wt{\zeta}_1)$ under $\varphi$ is contained in $[0,e^\beta]$ and $\varphi(\wt{\gamma}_2([0,\wt{\zeta}_2])) \subseteq B(i,e^{\beta})$.}
\end{figure}

We are now going to prove the lower bound for Theorem~\ref{thm::two_flowline_dimension}.  As in the proof of Theorem~\ref{thm::boundary_dimension}, we will accomplish this by introducing a special class of points, so-called ``perfect points,'' which are contained in the intersection of two flow lines whose correlation structure is easy to control.  Fix $\wt{\beta} > \beta^2 > \beta > 1$; we will eventually send $\wt{\beta} \to \infty$ but we will take $\beta$ fixed and large.

\subsubsection{Definition of the events}
\label{subsubsec::flow_line_perfect}

We are going to define the perfect points as follows.  Suppose that $\gamma_1$ is a path in $\h$ starting from $0$ and $\gamma_2$ is a path starting from $x_2 \in [0,e^\beta]$.  Let $\wt{\zeta}_1$ be the first time that $\gamma_1$ hits $\partial B(i,e^{-\wt{\beta}})$ and suppose that $\wt{\gamma}_2$ is a path starting from $\gamma_1(\wt{\zeta}_1)$.  Fix $u \in \R \setminus [0,x_2]$.  We let $E_u^{\beta,\wt{\beta}}(\gamma_1,\wt{\gamma}_2,\gamma_2)$ be the event that the following hold (see Figure~\ref{fig::perfect_def} for an illustration):
\begin{enumerate}[(i)]
\item\label{it::initial_straight} $\gamma_1$ hits $\partial B(i, e^{-\beta})$ before leaving the $e^{-2\beta}$ neighborhood of $[0,i]$,
\item The first time $\zeta_1$ (resp.\ $\zeta_2$) that $\gamma_1$ (resp.\ $\gamma_2$) hits $\partial B(i,e^{-\wt{\beta}-\beta})$ (resp.\ $\partial B(i,e^{-\wt{\beta}})$) is finite and $\gamma_i([0,\zeta_i]) \subseteq B(i,e^{2\beta})$ for $i=1,2$.
\item The first time $\wt{\zeta}_2$ that $\wt{\gamma}_2$ hits $\gamma_2$ is finite and $\wt{\gamma}_2([0,\wt{\zeta}_2])$ does not intersect either $\partial B(i,\tfrac{1}{2} e^{-\wt{\beta}})$ or $\partial B(i,2e^{-\wt{\beta}})$.
\item The connected component of $\h \setminus (\gamma_1([0,\zeta_1]) \cup \wt{\gamma}_2([0,\wt{\zeta}_2]) \cup \gamma_2([0,\zeta_2]))$ which contains $i$ also contains $u$ on its boundary.
\item \label{it::e_beta_left_right}  The probability that a Brownian motion starting from $i$ exits $\h \setminus (\gamma_1([0,\zeta_1]) \cup \wt{\gamma}_2([0,\wt{\zeta}_2]) \cup \gamma_2([0,\zeta_2]))$ on the left (resp.\ right) side of $\gamma_1([0,\zeta_1])$ is at least $\tfrac{1}{2}-e^{-\beta/4}$ and the probability of exiting on the left (resp.\ right) side of $\wt{\gamma}_2([0,\wt{\zeta}_2])$ (resp.\ $\gamma_2([0,\zeta_2])$) is at least $e^{-\beta}$.  We take $H$ to be the connected component of $\h \setminus \gamma_1([0,\zeta_1])$ with $u$ on its boundary and let $\varphi=\varphi(\gamma_1)$ be the conformal transformation $H \to \h$ which fixes $i$ and with $\varphi(\gamma_1(\zeta_1)) = 0$.  Finally, the image of (the right side of) $\gamma_1(\wt{\zeta}_1)$ under $\varphi$ is contained in $[0,e^\beta]$ and $\varphi(\wt{\gamma}_2([0,\wt{\zeta}_2])) \subseteq B(i,e^{\beta})$.
\end{enumerate}

The purpose of Part~\eqref{it::initial_straight} above is that, by drawing a path up until hitting $\partial B(i,e^{-\beta})$ and then conformally mapping back, the resulting configuration of paths satisfies the hypotheses of Lemma~\ref{lem::one_point_both}.

\begin{lemma}
\label{lem::e_b_property}
Suppose that we have the same setup described just above.  There exists a constant $C_1 > 0$ such that the following is true.  On the event $E_u^{\beta,\wt{\beta}}(\gamma_1,\wt{\gamma}_2,\gamma_2)$, with $\varphi=\varphi(\gamma_1)$, for each $\alpha \in (0,1)$ we have that $B(i,C_1 e^{(1-\alpha) (\beta+\wt{\beta})/2}) \subseteq \varphi(B(i,e^{-\alpha (\beta+\wt{\beta})}))$.
\end{lemma}
\begin{proof}
Throughout, we shall suppose that $E_u^{\beta,\wt{\beta}}(\gamma_1,\wt{\gamma}_2,\gamma_2)$ occurs.  Fix $\alpha \in (0,1)$.  The probability that a Brownian motion starting from $i$ hits $\partial B(i,e^{-\alpha (\beta+\wt{\beta})})$ before hitting $\partial \h \cup \gamma_1([0,\zeta_1])$ is $O(e^{-(1-\alpha) (\beta+\wt{\beta})/2})$ by the \hyperref[thm::beurling]{Beurling estimate}.  By the conformal invariance of Brownian motion, the probability of the event $X$ that a Brownian motion starting from $i$ exits $\varphi(B(i,e^{-\alpha (\beta+\wt{\beta})}))$ in $\varphi(\partial B(i,e^{-\alpha (\beta+\wt{\beta})}))$ is also $O(e^{-(1-\alpha) (\beta+\wt{\beta})/2})$.  Let
\[ d = \dist(\varphi(\partial B(i,e^{-\alpha (\beta+\wt{\beta})})),i).\]
We claim $\p[X]\gtrsim d^{-1}$.  Indeed, $X_1 \cap X_2 \subseteq X$ where $X_1$ is the event that the Brownian motion exits $\partial B(0,d)$ before hitting $\partial \h$ at a point with argument in $[\tfrac{\pi}{4},\tfrac{3\pi}{4}]$ and $X_2$ is the event that it hits $\varphi(\partial B(i,e^{-\alpha (\beta+\wt{\beta})}))$ after hitting $\partial B(0,d)$ before hitting $\partial \h$.  It is easy to see that $\p[X_1] \gtrsim d^{-1}$ and $\p[X_2 \giv X_1]\gtrsim 1$.  Consequently, $e^{-(1-\alpha) (\beta+\wt{\beta})/2}\gtrsim d^{-1}$ hence $d \gtrsim e^{(1-\alpha) (\beta+\wt{\beta})/2}$, as desired.
\end{proof}

\subsubsection{Flow line estimates}
\label{subsubsec::flow_line_estimates}

Fix $\theta \in (\pi-2\lambda/\chi,0)$; recall that this is the range of angles so that a GFF flow line with angle $\theta$ can hit and bounce off of a GFF flow line with angle $0$ on its right side.  We will now use the events introduced in Section~\ref{subsubsec::flow_line_perfect} to define the perfect points.  Suppose that $h_1$ is a GFF on $\h$ with the following boundary data: 
suppose $x_{1,1} = x_{1,2} = 0$ and $u_1 \in \R \setminus \{0\}$.  If $u_1 < x_{1,1}=x_{1,2}=0$, the boundary data is
\begin{align*}
h|_{(-\infty,u_1]} \equiv \lambda+(2\pi-\theta)\chi, \quad h|_{(u_1,0]} \equiv -\lambda,\quad\text{and}\quad h|_{(0,\infty)} \equiv \lambda-\theta \chi.
\end{align*}
If $u_1 > x_{1,1}=x_{1,2}=0$, then the boundary data is
\begin{align*}
h|_{(-\infty,0]} \equiv -\lambda,\quad h|_{(0,u_1]} \equiv \lambda-\theta \chi,\quad\text{and}\quad h|_{(u_1,\infty)} \equiv -\lambda-2\pi\chi.
\end{align*}
These two possibilities correspond to the boundary data that arises when one takes a GFF with boundary conditions as in Figure~\ref{fig::flow_line_intersection_one_point} and Figure~\ref{fig::flow_line_intersection_one_point2} and then applies a change of coordinates which takes a given point $z \in \h$ to $i$.  In either case, we let $\eta_{1,1}$ (resp.\ $\eta_{1,2}$) be the flow line of $h_1$ starting from $x_{1,1}$ (resp.\ $x_{1,2}$) of angle $0$ (resp.~$\theta$).  We also let $\wt{\zeta}_{1,1}$ be the first time that $\eta_1$ hits $\partial B(i,e^{-\wt{\beta}})$ and let $\wt{\eta}_{1,2}$ be the flow line of $h_1$ starting from (the right side of) $\eta_{1,1}(\wt{\zeta}_{1,1})$ with angle $\theta$.

Let $E_1 = E_{u_1}^{\beta,\wt{\beta}}(\eta_{1,1},\wt{\eta}_{1,2},\eta_{1,2})$. Let $\zeta_{1,1}$ (resp.\ $\zeta_{1,2}$) be the first time that $\eta_{1,1}$ (resp.\ $\eta_{1,2}$) hits $\partial B(i,e^{-\wt{\beta}-\beta})$ (resp.\ $\partial B(i,e^{-\wt{\beta}})$) and let $\wt{\zeta}_{1,2}$ be the first time that $\wt{\eta}_{1,2}$ hits $\eta_{1,2}$.  Let $\varphi_1$ be the unique conformal map from the connected component of $\h \setminus \eta_{1,1}([0,\zeta_{1,1}])$ with $u_1$ on its boundary which fixes $i$ and sends the tip $\eta_{1,1}(\zeta_{1,1})$ to $0$.

Suppose that the events $E_j$ have been defined as well as paths $\eta_{j,1},\wt{\eta}_{j,2}, \eta_{j,2}$, GFFs $h_j$, and conformal transformations $\varphi_j$ for $1 \leq j \leq k$.  On the event that $\eta_{k,1}$ hits $\partial B(i,e^{-\beta-\wt{\beta}})$, we take $\eta_{k+1,1} = \varphi_k(\eta_{k,1})$ and $\eta_{k+1,2} = \varphi_k(\wt{\eta}_{k,2})$.  Note that $\eta_{k+1,1}$ is the flow line of the GFF $h_{k+1} = h_k \circ \varphi_k^{-1} - \chi \arg (\varphi_k^{-1})'$ starting from $0$. Similarly, $\eta_{k+1,2}$ is the flow line of $h_{k+1}$ starting from $x_{k+1,2} = \varphi_k(\eta_{k,1}(\wt{\zeta}_{k,1}))$ with angle $\theta$.  We let $\wt{\zeta}_{k+1,1}$ be the first time that $\eta_{k+1,1}$ hits $\partial B(i,e^{-\wt{\beta}})$ and let $\wt{\eta}_{k+1,2}$ be the flow line starting from (the right side of) $\eta_{k+1,1}(\wt{\zeta}_{k+1,1})$ with angle $\theta$ and let $u_{k+1} = \varphi_k(u_k)$.

On the event that $\eta_{k+1,1}$ hits $\partial B(i,e^{-\wt{\beta}-\beta})$, say for the first time at time $\zeta_{k+1,1}$, we let $\varphi_{k+1}$ be the conformal transformation which uniformizes the connected component of $\h \setminus \eta_{k+1,1}([0,\zeta_{k+1,1}])$ with $u_{k+1}$ on its boundary fixing $i$ and with $\varphi_{k+1}(\eta_{k+1,1}(\zeta_{k+1,1}))=0$.  We then define the event $E_{k+1}$ in terms of the paths $\eta_{k+1,1}$, $\wt{\eta}_{k+1,2}$, and $\eta_{k+1,2}$ analogously to $E_1$ as well as stopping times $\zeta_{k+1,2},\wt{\zeta}_{k+1,2}$.  For each $n \ge m$ we let
\begin{equation}
\label{eqn::perfect_definition}
E^{m,n} = \cap_{k=m+1}^n E_k \quad\text{and}\quad E^n = E^{0,n}.
\end{equation}

\begin{remark}
\label{rem::measurable}
$\ $

\begin{enumerate}[(i)]
\item Note that $E^{m,n}$ for $n > m \geq 1$ can occur even if only a subset of (or none of) $E^1,\ldots,E^m$ occur.
\item The conformal maps $\varphi_j$ are measurable with respect to $\eta_{1,1}$.  Note that each of the paths $\wt{\eta}_{k,2}$ is given by the conformal image of a flow line which starts at a point in the range of $\eta_{1,1}$.  The starting points of these flow lines are likewise measurable with respect to $\eta_{1,1}$.  These facts will be important when we establish the two point estimate for the lower bound of Theorem~\ref{thm::two_flowline_dimension} at the end of this subsection.
\end{enumerate}
\end{remark}

We will now work towards proving the one point estimate for the perfect point $i$.

\begin{proposition}
\label{prop::many_perfect}
There exists $\beta_0 > 1$ such that for all $\wt{\beta} > \beta^2 > \beta \geq \beta_0$ we have
\begin{equation}
\label{eqn::perfect_probability}
\p[E^n] \asymp e^{-\wt{\beta}(1 + O_\beta(1) o_{\wt{\beta}}(1)) n A}
\end{equation}
where $A$ is the constant from \eqref{eqn::two_flow_line_exponent} and the constants in the $\asymp$ of \eqref{eqn::perfect_probability} depend only on $u_1$, $\kappa$, and $\theta$.
\end{proposition}

In the statement of Proposition~\ref{prop::many_perfect}, we write $o_{\wt{\beta}}(1)$ to indicate a quantity which converges to $0$ as $\wt{\beta} \to \infty$ and $O_\beta(1)$ for a term which is bounded by some constant which depends only on $\beta$.  In particular, for $\beta$ fixed, $O_{\beta}(1) o_{\wt{\beta}}(1) \to 0$ as $\wt{\beta} \to \infty$.  The first step in the proof of Proposition~\ref{prop::many_perfect} is Lemma~\ref{lem::single_perfect}.  The second step, which allows one to iterate the estimate in \eqref{eqn::perfect_probability_one}, is Lemma~\ref{lem::perfect_conditional} and is stated and proved below.

\begin{lemma}
\label{lem::single_perfect}
There exists $\beta_0 > 1$ such that for all $\wt{\beta} > \beta^2 > \beta \geq \beta_0$ we have
\begin{equation}
\label{eqn::perfect_probability_one}
\p[E_1] \asymp e^{-\wt{\beta}(1 + O_\beta(1) o_{\wt{\beta}}(1)) A}
\end{equation}
where $A$ is the constant from \eqref{eqn::two_flow_line_exponent} and the constants in the $\asymp$ of \eqref{eqn::perfect_probability_one} depend only on $u_1$, $\kappa$, and $\theta$.
\end{lemma}
\begin{proof}
By Lemma~\ref{lem::sle_kappa_rho_close_to_curve}, we know that $\eta_{1,1}$ has a positive chance of being uniformly close to $[0,i]$ before hitting $\partial B(i,e^{-\beta})$.  Let $\tau$ be the first time that $\eta_{1,1}$ hits $\partial B(i,e^{-\beta})$ and let $g$ be the conformal transformation from the connected component of $\h \setminus \eta_{1,1}([0,\tau])$ containing $i$ which fixes $i$ and sends $\eta_{1,1}(\tau)$ to $0$.  By choosing $\beta_0$ sufficiently large, it is clear that $g(\eta_{1,1})$ and $g(\eta_{1,2})$ satisfy the hypotheses of \eqref{eqn::one_point_both_good} of Lemma~\ref{lem::one_point_both}.  From this, we deduce that the probability that $\eta_{1,1}$ and $\eta_{1,2}$ both hit $\partial B(i,2 e^{-\wt{\beta}})$ before leaving $B(i,e^{2\beta})$ and such that the harmonic measure of the left (resp.\ right) side of each of the paths stopped at this time as viewed from $i$ is bounded from below by some universal constant is equal to $e^{-\wt{\beta} (1 + O_{\beta}(1) o_{\wt{\beta}}(1)) A}$.  The rest of the lemma follows from repeated applications of Lemma~\ref{lem::sle_kappa_rho_close_to_curve} and Lemma~\ref{lem::sle_kappa_rho_hit_boundary_segment}.
\end{proof}

For each $z \in \h$, we let $\psi_z$ be the unique conformal transformation $\h \to \h$ taking $z$ to $i$ and fixing $0$.  For each $k \in \N$, we let $\eta_{k,i}^z$ for $i=1,2$ and $\wt{\eta}_{k,2}^z$ be the paths after applying the conformal map $\psi_z$ and we let $\zeta_{k,i}^z$, $\wt{\zeta}_{k,i}^z$ be the corresponding stopping times.  We define
\begin{equation}
\label{eqn::g_n_z_def}
 \begin{split}
 E^{m,n}(z) &= E^{m,n}(\eta_{1,1}^z, \wt{\eta}_{1,2}^z, \eta_{1,2}^z) \quad\text{and}\\
  E^n(z) &= E^{0,n}(z).
 \end{split}
\end{equation}
In other words, $E^{m,n}(z)$ and $E^n(z)$ are the events corresponding to $E^{m,n}$ and $E^n$ defined in \eqref{eqn::perfect_definition} but with respect to the flow lines of the GFF $h_1 \circ \psi_z^{-1} - \chi \arg (\psi_z^{-1})'$ starting from $0$.  Let $\varphi_{k,z}$ be the corresponding conformal maps.  We let
\begin{equation}
\label{eqn::psi_iterative}
 \varphi_z^{j,k} = \varphi_{j+1,z} \circ \cdots \circ \varphi_{k,z} \quad\text{for each}\quad 0 \leq j \leq k\quad\text{and}\quad \varphi_z^k = \varphi_z^{0,k}.
\end{equation}
We also let
\begin{align*}
 V_n(z) &= B(z, 2^{8n+4} \im(z) e^{-n(\beta+\wt{\beta})}) \quad\text{for each}\quad n \in \N \quad\text{and}\quad z \in \h.
\end{align*}

\begin{lemma}
\label{lem::q_size}
There exists $\beta_0 > 1$ such that for all $\wt{\beta} >\beta^2> \beta \geq \beta_0$, the following is true.  For each $m,n \in \N$ with $m \geq n+2$, on $E^m(z)$ we have that $\psi_z^{-1} \circ (\varphi_z^{m-1})^{-1}(\gamma) \subseteq V_n(z)$ for $\gamma = \eta_{m,i}^z([0,\zeta_{m,i}^z])$ for $i=1,2$ and $\gamma=\wt{\eta}_{m,2}^z([0,\wt{\zeta}_{m,2}^z])$.
\end{lemma}
\begin{proof}
We are first going to give the proof in the case that $z=i$.  Fix $m,n \in \N$ with $m \geq n+2$.  Throughout, we shall assume that we are working on $E^m$.  It follows from \cite[Corollary~3.25]{lawler2005} that if $r \in (0,\tfrac{1}{2})$ then
\begin{align}
\label{eqn::varphi_ball}
\varphi_k^{-1}(B(i,r)) \subseteq B(i,16 r e^{-\wt{\beta}-\beta}) \quad&\text{for}\quad 1 \leq k \leq m.
\end{align}
Iterating \eqref{eqn::varphi_ball} implies that
\begin{equation}
\label{eqn::varphi_iterate_contain}
\begin{split}
(\varphi^k)^{-1}(B(i,\tfrac{1}{2})) \subseteq B(i,2^{8 k} e^{-k (\wt{\beta}+\beta)}) \quad& \text{for}\quad 1 \leq k \leq m
\end{split}
\end{equation}
(provided we take $\beta_0$ large enough).

Note that $\eta_{m,i}([0,\zeta_{m,i}]) \subseteq B(i, e^{2\beta})$ for $i=1,2$ by the definition of the events.  Consequently, it follows from Lemma~\ref{lem::e_b_property} that $\varphi_{m-1}^{-1}(\eta_{m,i}([0,\zeta_{m,i}])) \subseteq B(i,e^{-\wt{\beta}/4})$ for $i=1,2$ provided $\beta_0$ is large enough.  We also assume that $\beta_0$ is sufficiently large so that $e^{-\wt{\beta}/4} < \tfrac{1}{2}$.  Applying \eqref{eqn::varphi_iterate_contain} proves the result for $\eta_{m,i}([0,\zeta_{m,i}])$ for $i=1,2$ and $\wt{\eta}_{m,2}([0,\wt{\zeta}_{m,2}])$.  This proves the result for $z=i$.  For the case that $z \neq i$, we note that applying \cite[Corollary~3.25]{lawler2005} again yields,
\begin{equation}
\label{eqn::psi_z_contain}
\psi_z^{-1}(B(i,r)) \subseteq B(i,16 r \im(z)).
\end{equation}
Combining \eqref{eqn::varphi_iterate_contain} with \eqref{eqn::psi_z_contain} gives the desired result.
\end{proof}

\begin{figure}[ht!]
\begin{center}
\includegraphics[scale=0.85]{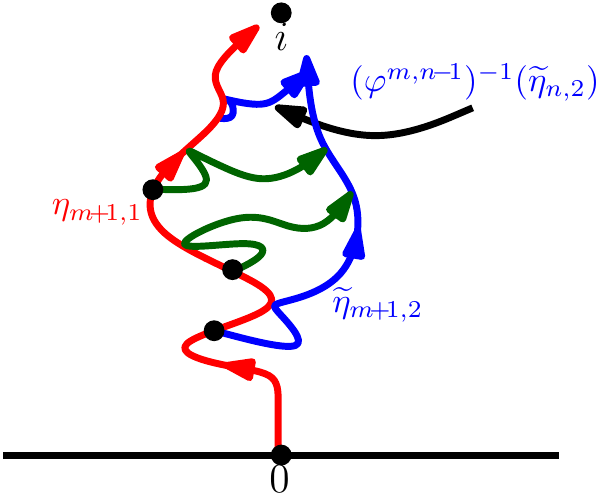}
\end{center}
\caption{\label{fig::paths_shield}
Illustration of the configuration of paths used in the proof of Lemma~\ref{lem::perfect_conditional}.  On $E^{m,n}$, $\eta_{m+1,1}$, $\wt{\eta}_{m+1,2}$, and $(\varphi^{m,n-1})^{-1}(\wt{\eta}_{n,2})$ separate the paths $(\varphi^{m,j-1})^{-1}(\wt{\eta}_{j,2})$ for $m+2 \leq j \leq n-1$ (shown in green) stopped upon hitting $\wt{\eta}_{m+1,2}$ from $i$.  Thus, once $\eta_{m+1,1}$, $\wt{\eta}_{m+1,2}$, and $(\varphi^{m,n-1})^{-1}(\wt{\eta}_{n,2})$ have been fixed, the conditional law of the remaining paths does not depend on the boundary data of $h_{m+1}$ or on the other auxiliary paths.
}
\end{figure}

For each $m \in \N$ and $z \in \h$, let $\CF_m(z)$ be the $\sigma$-algebra generated by $\eta_{k,i}^z|_{[0,\zeta_{k,i}^z]}$ for $i=1,2$ and $\wt{\eta}_{k,2}^z|_{[0,\wt{\zeta}_{k,2}^z]}$ for $1 \leq k \leq m$.

\begin{lemma}
\label{lem::perfect_conditional}
There exists $\beta_0 > 1$ such that for all $\wt{\beta} > \beta^2 > \beta \geq \beta_0$ the following is true.  Fix $\delta \in (0,\tfrac{\pi}{2})$ and $z \in \D \cap \h$ with $\arg(z) \in (\delta,\pi-\delta)$.  For each $m \in \N$ we have that
\begin{equation}
\label{eqn::perfect_conditional_same}
\p[ E^{m,n}(z) \giv \CF_m(z)] \one_{E^m(z)} \asymp  e^{O_{\beta}(1) o_{\wt{\beta}}(1) \wt{\beta}} \p[E^{n-m}] \one_{E^m(z)}
\end{equation}
where the constants in $\asymp$ depend only on $\delta$, $\kappa$, and $\theta$.
\end{lemma}
\begin{proof}
By applying $\psi_z$, we may assume without loss of generality that $z=i$.  Recall the definition of the GFF $h_{m+1}$ as well as the paths $\eta_{k,i}$ for $i=1,2$ and $\wt{\eta}_{k,2}$ from just before Remark~\ref{rem::measurable}.  By the definition of $E^m$ and the conformal invariance of Brownian motion, we know that there exists a constant $c_1 > 0$ such that the boundary data for $h_{m+1}$ in $(-c_1,0)$ (resp.\ $(0,c_1)$) is given by $-\lambda$ (resp.\ $\lambda$).  The same is likewise true for $h_1$.  Moreover, by Lemma~\ref{lem::e_b_property}, it follows that the auxiliary paths coupled with $h_{m+1}$ are far away from $i$ provided $\beta_0$ is large enough.  Consequently, by Lemma~\ref{lem::change_of_domains_bound_RN}, the laws of $\eta_{m+1,1}$ (given $E^m$) and $\eta_{1,1}$ stopped upon exiting the $\tfrac{c_1}{2}$ neighborhood of the line segment from $0$ to $i$ are mutually absolutely continuous with Radon-Nikodym derivative which is bounded from above and below by universal positive and finite constants which depend only on $\kappa$ and $\theta$.

On $E^{m,n}$, $\eta_{m+1,1}$ does not leave this tube before getting very close to $i$ and neither does $\eta_{1,1}$ on $E^{n-m}$.  For a given choice of $\eta$, by Lemma~\ref{lem::change_of_domains_bound_RN}, we moreover have that the Radon-Nikodym derivative of the conditional law of $\wt{\eta}_{m+1,2}$ given $\eta_{m+1,1}=\eta$ stopped upon exiting the tube with respect to that of $\wt{\eta}_{1,2}$ given $\eta_{1,1} = \eta$ is bounded from above and below by universal finite and positive constants which do not depend on the specific choice of~$\eta$.  On this event, the same is also true for the Radon-Nikodym derivative of the conditional law of $(\varphi^{m,n-1})^{-1}(\wt{\eta}_{n,2})$ given $\eta_{m+1,1}=\eta$ and $\wt{\eta}_{m+1,2}=\wt{\eta}$ with respect to the conditional law of
$(\varphi^{n-m-1})^{-1}(\wt{\eta}_{n-m,2})$ given $\eta_{1,1}=\eta$ and $\wt{\eta}_{1,2}=\wt{\eta}$.
The conditional law of $(\varphi^{m,j-1})^{-1}(\wt{\eta}_{j,2})$ for $m+2 \leq j \leq n-1$ stopped upon hitting $\wt{\eta}_{m+1,2}$ given $\eta_{m+1,1}$, $\wt{\eta}_{m+1,2}$, and $\wt{\eta}_{n,2}$ is independent of the boundary data of $h_{m+1}$ (as well as the other auxiliary paths).  (See Figure~\ref{fig::paths_shield}.)  The same is likewise true for the conditional law of $(\varphi^{j-1})^{-1}(\wt{\eta}_{j,2})$ for $2 \leq j \leq n-m-1$ stopped upon hitting $\wt{\eta}_{1,2}$ given $\eta_{1,1}$, $\wt{\eta}_{1,2}$, and $\wt{\eta}_{n-m,2}$.

Let $K$ be the compact hull associated with these paths and let $g$ be the conformal transformation $\h \setminus K \to \h$ with $g(z) \sim z$ as $z \to \infty$.  Conditionally on all of these paths and the event that they are contained in $B(i,2e^{-\wt{\beta}})$, the probability that $\eta_{m+1,2}$ hits $\partial B(i,10 e^{-\wt{\beta}})$ before leaving $B(i,e^{2\beta})$ is $\asymp |g'(i) e^{-\wt{\beta}}|^{\alpha+O_{\beta}(1)o_{\wt{\beta}}(1)}$ (as in the proof of Lemma~\ref{lem::one_point_both}; the extra force points only change the probability by a positive and finite factor by Lemma~\ref{lem::change_of_domains_bound_RN}.) Given that $\eta_{m+1,2}$ has hit $\partial B(i,10 e^{-\wt{\beta}})$, the conditional probability that it then merges with $\wt{\eta}_{m+1,2}$ before the latter has hit $\partial B(i,\tfrac{1}{2} e^{-\wt{\beta}})$ or $\partial B(i,2e^{-\wt{\beta}})$ is positive by Lemma~\ref{lem::sle_kappa_rho_hit_boundary_segment}.  The same is true with $\eta_{1,2}$ in place of $\eta_{m+1,2}$, which completes the proof.
\end{proof}

\begin{proof}[Proof of Proposition~\ref{prop::many_perfect}]
This follows by combining Lemma~\ref{lem::single_perfect} with Lemma~\ref{lem::perfect_conditional}.
\end{proof}

\begin{figure}[ht!]
\begin{center}
\includegraphics[scale=0.85]{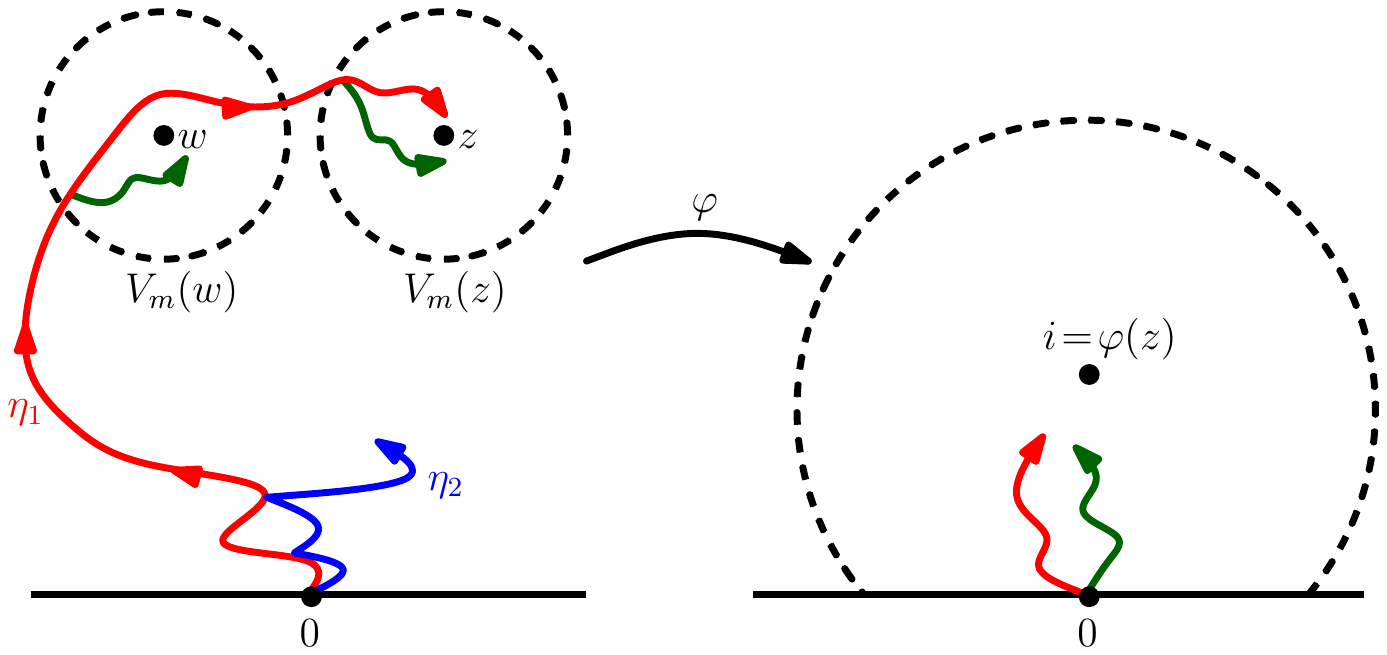}
\end{center}
\caption{\label{fig::two_point_approx_ind}
Illustration of the setup for the two point estimate (Lemma~\ref{lem::perfect_conditional} and Lemma~\ref{lem::two_point_perfect}) in the case that $\eta_1$ gets close first to $w$ and then to $z$.  Conformally map back $\eta_{1,1}$ drawn up until the path hits the neighborhood of $z$.  Then all of the auxiliary paths are outside of a large ball which is far from $i=\varphi(z)$, so we can apply the one point estimate for perfect points (Lemma~\ref{lem::single_perfect}) for this region as before.  We can also apply the one point estimate for the paths near $z$.  Finally, to complete the proof, we apply the one point estimate a final time for the paths up to when they hit a neighborhood containing both $z$ and $w$.}
\end{figure}

\begin{lemma}
\label{lem::perfect_conditional2}
Fix $\delta \in (0,\tfrac{\pi}{2})$ and $z,w \in \D \cap \h$ distinct with $\arg(z), \arg(w) \in (\delta,\pi-\delta)$ and let $m$ be the smallest integer such that $V_{m-1}(z) \cap V_{m-1}(w) = \emptyset$.  Let $P_w$ be the event that $\eta_{1,1}$ hits $V_m(w)$ before hitting $V_m(z)$.  There exists $\beta_0 > 1$ such that for every $\wt{\beta} > \beta^2 > \beta \geq \beta_0$ we have that
\begin{equation}
\label{eqn::perfect_conditional_distinct}
\p[ E^{m,n}(z) \giv \CF_k(w)] \one_{E^k(w), P_w} \leq e^{O_{\beta}(1) \wt{\beta}} \p[ E^{n-m}] \one_{E^k(w), P_w}
\end{equation}
for all $k \geq m$.
\end{lemma}
\begin{proof}
We are going to extract \eqref{eqn::perfect_conditional_distinct} from \eqref{eqn::perfect_conditional_same} of Lemma~\ref{lem::perfect_conditional}.  As before, by applying $\psi_z$, we may assume without loss of generality that $z=i$.  Fix $k \geq m$.  By Proposition~\ref{prop::many_perfect}, it suffices to prove
\begin{equation}
\label{eqn::perfect_conditional_distinct_reduced}
\begin{split}
 &\p[ E^{m+1,n} \giv E_{m+1}, \CF_k(w)] \one_{E^k(w), P_w}
 \lesssim \p[ E^{n-m-1}] \one_{E^k(w), P_w}
 \end{split}
\end{equation}
in place of \eqref{eqn::perfect_conditional_distinct}.  By Lemma~\ref{lem::q_size}, we know that the paths involved in $E^{m,n}$ are disjoint from those involved in $E^k(w)$ due to the choice of $m$.  Thus by conformally mapping back (see Figure~\ref{fig::two_point_approx_ind}) and applying Lemma~\ref{lem::change_of_domains_bound_RN} as in the proof of Lemma~\ref{lem::perfect_conditional}, it is therefore not hard to see that
\[ \p[ E^{m+1,n} \giv E_{m+1}, \CF_k(w)] \one_{E^k(w), P_w} \asymp \p[ E^{1,n-m} \giv E_1] \one_{E^k(w), P_w}.\]

Combining this with \eqref{eqn::perfect_conditional_same} completes the proof.
\end{proof}

\begin{lemma}
\label{lem::two_point_perfect}
For every $\eps > 0$ and $\delta \in (0,\tfrac{\pi}{2})$ there exists $\beta_0 > 1$ such that for all $\wt{\beta} > \beta^2 > \beta \geq \beta_0$ there exists constants $C > 0$ and $n_0 \in \N$ such that the following is true.  Fix $z,w \in \D \cap \h$ distinct with $\arg(z), \arg(w) \in (\delta,\pi-\delta)$.  Let $m$ be the smallest integer such that $V_{m-1}(z) \cap V_{m-1}(w) = \emptyset$.  Then
\[ \p[ E^n(z), E^n(w)] \leq C e^{\wt{\beta} (1 + \eps)m A}  \p[ E^n(z)] \p[ E^n(w)] \quad \text{ for all } \quad n \geq n_0.\]
\end{lemma}
\begin{proof}
Suppose that $z,w \in \h$ are as in the statement of the lemma.  Let $P_w$ be the event that $\eta_1$ hits $V_m(w)$ before hitting $V_m(z)$ and let $P_z$ be the event in which the roles of $z$ and $w$ are swapped.  We have that
\begin{align}
 &\p[ E^n(z), E^n(w)]
= \p[ E^n(z), E^n(w), P_w] + \p[ E^n(z), E^n(w), P_z] \notag\\
\leq& \p[ E^n(z) \giv  E^n(w),P_w] \p[ E^n(w)] + \p[ E^n(w) \giv  E^n(z),P_z] \p[ E^n(z)]. \label{eqn::two_point_perfect1}
\end{align}
We are going to bound the first summand; the second is bounded analogously.  We have,
\begin{align}
\label{eqn::two_point_perfect2}
    \p[ E^n(z) \giv  E^n(w), P_w]
\leq \p[ E^{m,n}(z) \giv  E^n(w), P_w].
\end{align}
By \eqref{eqn::perfect_conditional_distinct} of Lemma~\ref{lem::perfect_conditional2}, we have that
\begin{equation}
\label{eqn::two_point_perfect3}
\p[ E^{m,n}(z) \giv E^n(w), P_w] \leq e^{O_\beta(1) \wt{\beta}} \p[ E^{n-m}].
\end{equation}
By \eqref{eqn::perfect_conditional_same} of Lemma~\ref{lem::perfect_conditional} and Proposition~\ref{prop::many_perfect}, we have that
\begin{equation}
\label{eqn::two_point_perfect4}
 \p[ E^{n-m}] \leq e^{\wt{\beta}(1+\eps) m A} \p[E^n(z)]
\end{equation}
(possibly increasing $\beta_0$).  The same likewise holds when we swap the roles of $P_w$ and $P_z$.  Combining \eqref{eqn::two_point_perfect1}--\eqref{eqn::two_point_perfect4} gives the result.
\end{proof}

We can now complete the proof of Theorem~\ref{thm::two_flowline_dimension}.

\begin{proof}[Proof of Theorem~\ref{thm::two_flowline_dimension}]
We suppose that $h$ is a GFF on $\h$ with boundary conditions
\[ h|_{(-\infty,0]} \equiv -\lambda \quad\text{and}\quad h|_{(0,\infty)} \equiv \lambda-\theta \chi\]
and let $\eta_1$ (resp.\ $\eta_2$) be the flow line of $h$ starting from $0$ with angle~$0$ (resp.~$\theta \in (\pi-2\lambda/\chi,0)$).  We have already established the upper bound for $\dimH(\eta_1 \cap \eta_2 \cap \h)$ in Proposition~\ref{prop::dim_intersection_upper_bound}.  We will now establish the lower bound.  Once we have proved this, we get the corresponding dimension when $h$ has general piecewise constant boundary data as described in the theorem statement by absolute continuity for GFFs.

The proof is completed in the same manner as the proof of Theorem~\ref{thm::boundary_dimension}.    Indeed, we let $\eps_n = 2^{8n+4} e^{-(\beta+\wt{\beta})n}$.  We divide $[-1,1] \times [1,2]$ into $2\eps_n^{-2}$ squares of equal side length $\eps_n$ and let $z^n_j$ be the center of the $j$th such square for $j=1,\ldots,2\eps_n^{-2}$.  Let $\CC_n$ be the set of centers $z$ of these squares for which $E^n(z)$ occurs.  Let $S_n(z)$ be the square with center $z$ and length $\eps_n$. Finally, we let
\[ \CC=\bigcap_{k\ge 1}\overline{\bigcup_{n\ge k}\bigcup_{z\in\CC_n} S_n(z)}.\]
It is easy to see that
\[ \CC \subseteq \eta_1 \cap \eta_2 \cap \h.\]
The argument of the proof of Theorem~\ref{thm::boundary_dimension} combined with Lemma~\ref{lem::two_point_perfect} implies, for each $\xi > 0$, that $\p[ \dimH(\eta_1 \cap \eta_2) \geq 2-A-\xi] > 0$.
To finish the proof, we only need to explain the 0-1 argument: that for each $d\in[0,2]$, $\p[\dimH(\eta_1 \cap \eta_2 \cap \h)=d]\in\{0,1\}$.  For $r>0$, let $D_r=\dimH(\eta_1\cap\eta_2\cap B(0,r) \cap \h)$. It is clear that $0 < r_1<r_2$ implies $D_{r_1}\le D_{r_2}$.  By the scale invariance of the setup, we have that $D_{r_1}$ has the same law as $D_{r_2}$. Thus $D_{r_1}=D_{r_2}$ almost surely for all $0 < r_1<r_2$. In particular, $\p[D_{\infty}=D_r]=1$ for all $r>0$. Thus the events $\{D_{\infty}=d\}$ and $\{D_r=d\}$ are the same up to a set of probability zero.  The latter is measurable with respect to the GFF restricted to $B(0,r)$. Letting $r \downarrow 0$, we see that this implies that the event $\{D_{\infty}=d\}$ is trivial, which completes the proof.
\end{proof}

\section{Proof of Theorem~\ref{thm::double_point_dimension}}
\label{sec::main_results}

We will first work towards proving \eqref{eqn::double_point_dimension} for $\kappa' \in (4,8)$; let $\kappa=\frac{16}{\kappa'}\in (2,4)$.  It suffices to compute the almost sure Hausdorff dimension of the double points of the chordal $\SLE_{\kappa'}(\tfrac{\kappa'}{2}-4;\tfrac{\kappa'}{2}-4)$ processes.  Indeed, this follows since the conditional law of an $\SLE_{\kappa'}$ process given its left and right boundaries is independently that of an $\SLE_{\kappa'}(\tfrac{\kappa'}{2}-4;\tfrac{\kappa'}{2}-4)$ in each of the bubbles which lie between these boundaries (recall Figure~\ref{fig::counterflowline_and_flowline}).  In order to establish this result, we are going to make use of the path decomposition developed in \cite{IG3} which was used to prove the reversibility of $\SLE_{\kappa'}$ for $\kappa' \in (4,8)$.  This, in turn, makes use of the duality results established in \cite[Section~7]{IG1}.  For the convenience of the reader, we are going to review the path decomposition here.

Throughout, we suppose that $h$ is a GFF on the horizontal strip $\T=\R\times (0,1)$ with boundary values given by $-\lambda+\frac{\pi}{2}\chi=-\lambda'$ on the lower boundary $\partial_L\T=\R$ of the strip and $\lambda-\frac{3\pi}{2}\chi=\lambda'-\pi \chi$ on the upper boundary $\partial_U\T=\R\times\{1\}$ of the strip.  (See Figure~\ref{fig::pocket} for an illustration of the setup and recall the identities from \eqref{eqn::constants}.)  Let $\eta'$ be the counterflow line of $h$ from $+\infty$ to $-\infty$.  Then $\eta'$ is an $\SLE_{\kappa'}(\tfrac{\kappa'}{2}-4;\tfrac{\kappa'}{2}-4)$ process in $\T$ from $+\infty$ to $-\infty$ where the force points are located immediately to the left and right of the starting point of the path.  Recall that $\frac{\kappa'}{2}-4$ is the critical threshold at or below which an $\SLE_{\kappa'}(\rho)$ process fills the domain boundary.  Fix $z \in \partial \T$ and let $t(z)$ be the first time $t$ that $\eta'$ hits $z$.  Then $t(z) < \infty$ almost surely (and this holds for all boundary points simultaneously).  Assume further that $z \in \partial_L \T$ and let $\eta_z^1$ be the outer boundary of $\eta'([0,t(z)])$.  Explicitly, $\eta_z^1$ is equal to the flow line of $h$ with angle $\tfrac{\pi}{2}$ starting from $z$ stopped at time $\tau_z^1$, the first time that it hits $\partial_U \T$ (see Figure~\ref{fig::pocket}).  The conditional law of $\eta'$ given $\eta_z^1([0,\tau_z^1])$ in each of the connected components $C$ of $\T \setminus \eta_z^1([0,\tau_z^1])$ which lie to the right of  $\eta_z^1([0,\tau_z^1])$ is independently that of an $\SLE_{\kappa'}(\frac{\kappa'}{2}-4;\frac{\kappa'}{2}-4)$ process starting from the first point of $\ol{C}$ visited by $\eta'$ and terminating at the last.

\begin{figure}[ht!]
\begin{center}
\includegraphics[scale=0.85]{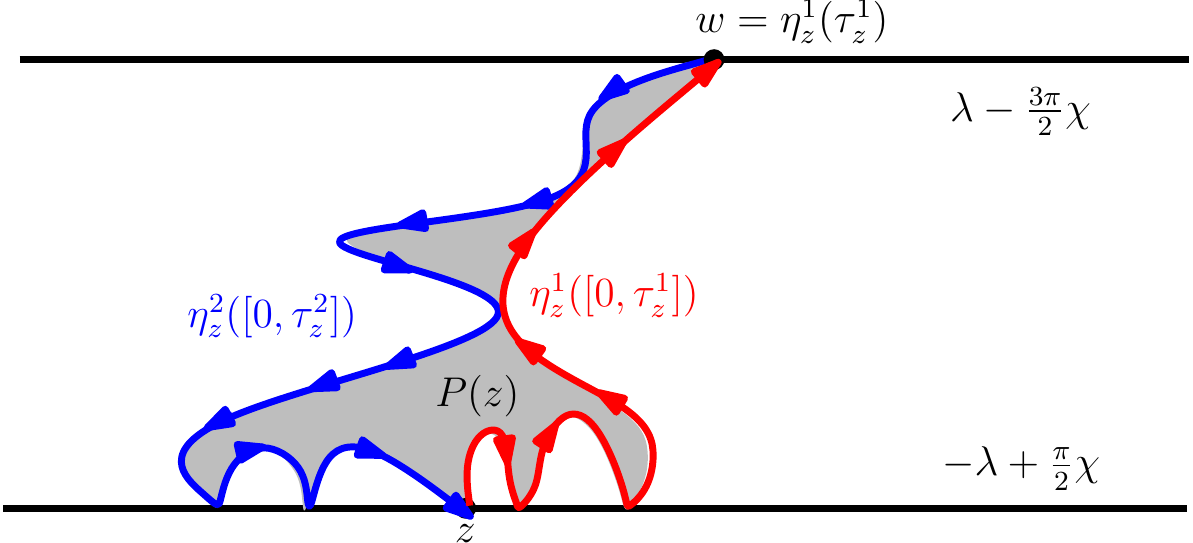}
\end{center}
\caption{\label{fig::pocket}
Suppose that $h$ is a GFF on the horizontal strip $\T = \R \times (0,1)$ with the illustrated boundary data and let $\eta'$ be the counterflow line of $h$ starting from $+\infty$ and targeted at $-\infty$.   Then $\eta'$ is an $\SLE_{\kappa'}(\tfrac{\kappa'}{2}-4;\tfrac{\kappa'}{2}-4)$ with force points located immediately to the left and right of the starting point of the path. Fix $z$ in the lower boundary $\partial_L \T = \R$ of $\T$ and let $t(z)$ be the first time that $\eta'$ hits $z$.  Since $\eta'$ is boundary filling, $t(z) < \infty$ almost surely.  Let $\eta_z^1$ be the outer boundary of $\eta'([0,t(z)])$.  Then $\eta_z^1$ is equal to the flow line of $h$ with angle $\tfrac{\pi}{2}$ starting from $z$ and stopped at time $\tau_z^1$, the first time that it hits $\partial_U\T$.  Let $w=\eta_z^1(\tau_z^1)$. Given $\eta_z^1([0,\tau_z^1])$, let $\eta_z^2$ be the outer boundary of $\eta'([t(z),\infty))$. Then $\eta_z^2$ is equal to the flow line of $h$ given $\eta_z^1([0,\tau_z^1])$ with angle $\tfrac{\pi}{2}$ started from $w$ stopped at time $\tau_z^2$, the first time it hits $z$. Let $P(z)$ be the region between $\eta_z^1([0,\tau_z^1])$ and $\eta_z^2([0,\tau_z^2])$ (indicated in gray). Given $P(z)$, the conditional law of $\eta'$ in each component $C$ of $\T\setminus P(z)$ is independently that of an $\SLE_{\kappa'}(\frac{\kappa'}{2}-4;\frac{\kappa'}{2}-4)$ from the first point in $\ol{C}$ visited by $\eta'$ to the last.  The points $\eta_z^1([0,\tau_z^1]) \cap \eta_z^2([0,\tau_z^2])$ are double points of $\eta'$.}
\end{figure}

\begin{figure}[ht!]
\begin{center}
\includegraphics[scale=0.85]{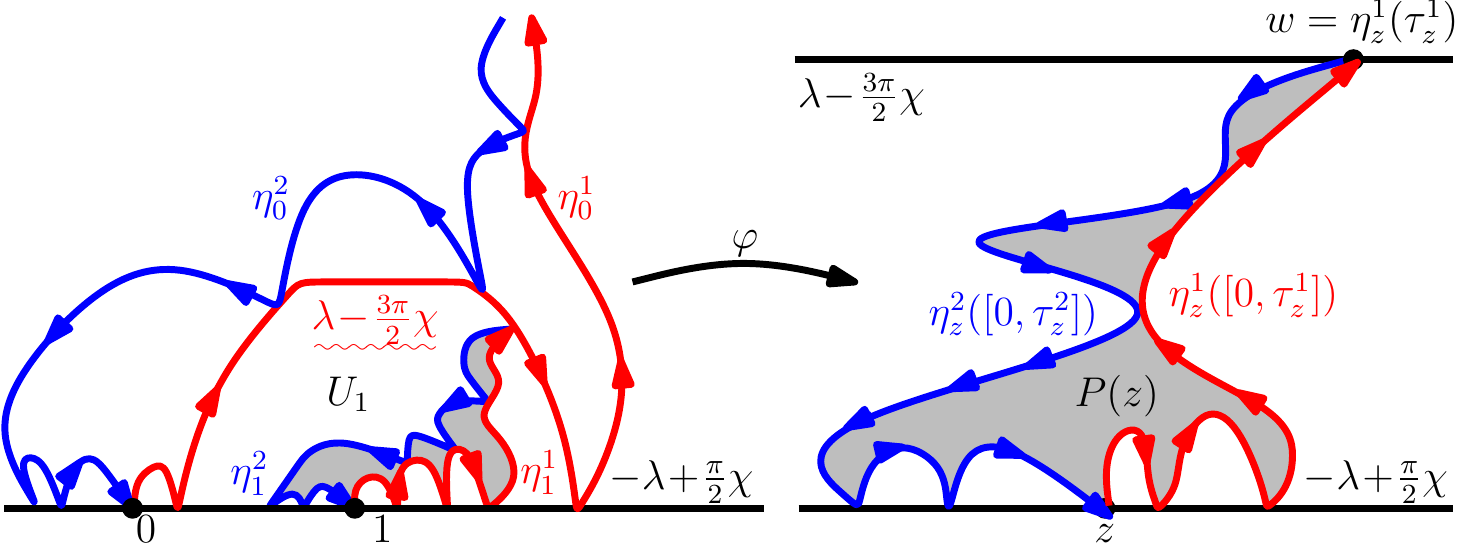}
\end{center}
\caption{\label{fig::pocket2}
(Continuation of Figure~\ref{fig::pocket}.)  Suppose that $\wt{h}$ is a GFF on $\h$ with the boundary data indicated on the left side.  Let $\eta_0^1$ be the flow line of $\wt{h}$ from $0$ to $\infty$ with angle $\tfrac{\pi}{2}$.  Given $\eta_0^1$, let $\eta_0^2$ be the flow line of $\wt{h}$ given $\eta_0^1$ from $\infty$ with angle $\tfrac{\pi}{2}$ in the connected component of $\h \setminus \eta_0^1$ which is to the left of $\eta_0^1$.  Then $\eta_0^1$ is an $\SLE_\kappa(\tfrac{\kappa}{2}-2;-\tfrac{\kappa}{2})$ in $\h$ from $0$ to $\infty$.  Moreover, the conditional law of $\eta_0^2$ given $\eta_0^1$ is that of an $\SLE_\kappa(\kappa-4;-\tfrac{\kappa}{2})$ in the component of $\h \setminus \eta_0^1$ which is to the left of $\eta_0^1$ from $\infty$ to $0$ (the $\kappa-4$ force point lies between the paths).  Shown is the boundary data for the conditional law of $\wt{h}$ given $(\eta_0^1,\eta_0^2)$ in the component $U_1$ of $\h \setminus (\eta_0^1 \cup \eta_0^2)$ which contains $1$ on its boundary.  Let $\varphi \colon U_1 \to \h$ be the conformal transformation with $\varphi(1) = z$ and which takes leftmost (resp.\ rightmost) point of $\partial U_1 \cap \partial \h$ to $-\infty$ (resp.\ $+\infty$).  Then $\wt{h} \circ \varphi^{-1} - \chi \arg (\varphi^{-1})'$ has the boundary data shown on the right side.  Let $(\eta_1^1,\eta_1^2)$ be a pair of paths defined in the same way as $(\eta_0^1,\eta_0^2)$ except starting from $1$.  Then the image of the region in $U_1$ between $\eta_1^1$ and $\eta_1^2$ under $\varphi$ has the same law as $P(z)$ described in Figure~\ref{fig::pocket}.  (See also \cite[Figure~3.2]{IG3}.)
}
\end{figure}

\begin{figure}[ht!]
\begin{center}
\includegraphics[scale=0.85]{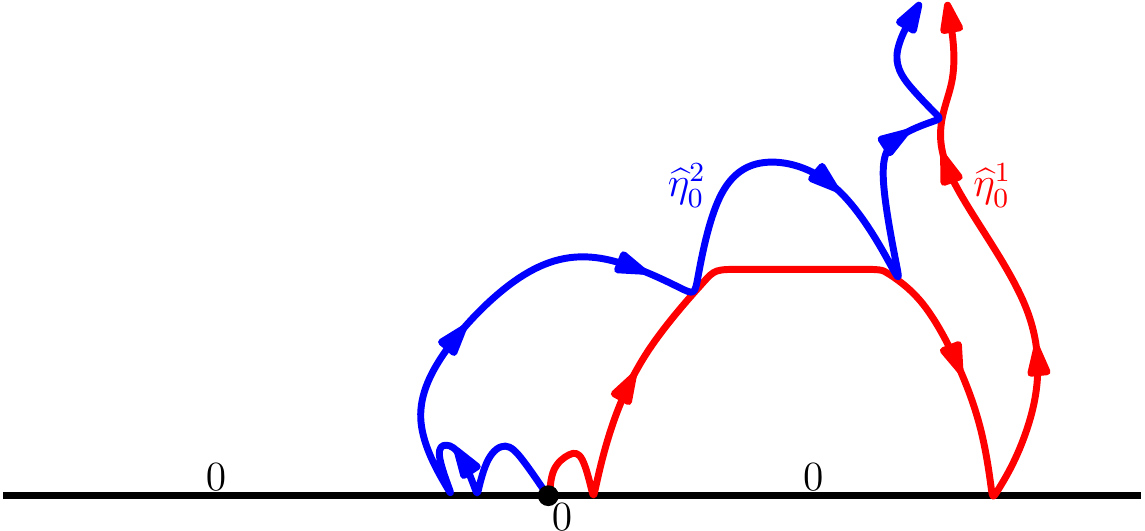}
\end{center}
\caption{\label{fig::pocket3}
Suppose that $\wh{h}$ is a GFF on $\h$ with zero boundary conditions as illustrated.  Let $\wh{\eta}_0^1$ (resp.\ $\wh{\eta}_0^2$) be the flow line of $\wh{h}$ starting from $0$ with angle $-\tfrac{1}{2}\theta_{\rm double}$ (resp.\ $\tfrac{1}{2} \theta_{\rm double}$); recall \eqref{eqn::angle_double}.  Then $\wh{\eta}_0^1$ is an $\SLE_\kappa(\tfrac{\kappa}{2}-2;-\tfrac{\kappa}{2})$ process in $\h$ from $0$ to $\infty$ (Figure~\ref{fig::gff_boundary_data_flow}) and the conditional law of $\wh{\eta}_0^2$ given $\wh{\eta}_0^1$ in the connected component of $\h \setminus \wh{\eta}_0^1$ which is to the left of $\wh{\eta}_0^1$ is an $\SLE_\kappa(-\tfrac{\kappa}{2};\kappa-4)$ process from $0$ to $\infty$ (Figure~\ref{fig::conditional_law}).  Similarly, $\wh{\eta}_0^2$ is an $\SLE_\kappa(-\tfrac{\kappa}{2};\tfrac{\kappa}{2}-2)$ process in $\h$ from $0$ to $\infty$ (Figure~\ref{fig::gff_boundary_data_flow}) and the conditional law of $\wh{\eta}_0^1$ given $\wh{\eta}_0^2$ is an $\SLE_\kappa(\kappa-4;-\tfrac{\kappa}{2})$ process from $0$ to $\infty$ in the component of $\h \setminus \wh{\eta}_0^2$ which is to the right of $\wh{\eta}_0^2$ (Figure~\ref{fig::conditional_law}).  In particular, by the main result of \cite{IG2}, the joint law of the ranges of $\wh{\eta}_0^1$ and $\wh{\eta}_0^2$ is equal to the joint law of the ranges of $\eta_0^1$ and $\eta_0^2$ from the left side of Figure~\ref{fig::pocket2}.  Consequently, we can use Theorem~\ref{thm::two_flowline_dimension} to compute the almost sure dimension of the intersection of the latter.
}
\end{figure}

Let $w=\eta_z^1(\tau_z^1)\in\partial_U\T$. Since $\eta'$ is boundary filling and cannot enter the loops it creates with itself or with the domain boundary, the first point on $\partial_U\T$ that $\eta'$ hits after time $t(z)$ is $w$. Let $\eta_z^2$ be the outer boundary of $\eta'([t(z),\infty))$.  Then $\eta_z^2$ is the flow line of $h$ given $\eta_z^1([0,\tau_z^1])$ with angle $\tfrac{\pi}{2}$ starting from $w$ and stopped at time $\tau_z^2$, the first time the path hits $z$. Let $P(z)$ be the region which lies between $\eta_z^1([0,\tau_z^1])$ and $\eta_z^2([0,\tau_z^2])$.  Then $P(z)$ separates the set of points that $\eta'$ visits before and after hitting $z$.  The right (resp.\ left) boundary of $P(z)$ is given by $\eta_z^1([0,\tau_z^1])$ (resp.\ $\eta_z^2([0,\tau_z^2])$).  The conditional law of $\eta'$ given $P(z)$ is independently that of an $\SLE_{\kappa'}(\tfrac{\kappa'}{2}-4;\tfrac{\kappa'}{2}-4)$ process in each of the components $C$ of $\T \setminus P(z)$ starting from the first point of $\ol{C}$ hit by $\eta'$ and terminating at the last --- the same as that of $\eta'$ up to a conformal transformation.  This symmetry allows us to iterate this exploration procedure to eventually discover the entire path.  Note that the intersection points $\eta_z^1([0,\tau_z^1]) \cap \eta_z^2([0,\tau_z^2])$ are double points of $\eta'$.  If $z \in \partial_U \T$, then we can define the paths $\eta_z^1,\eta_z^2$ analogously except the angle $\tfrac{\pi}{2}$ is replaced with $-\tfrac{\pi}{2}$.  This is because when $\eta'$ hits $z \in \partial_U \T$, only its right boundary is visible from $-\infty$ which is contrast to the case when it hits $z \in \partial_L \T$ when only its left boundary is visible from~$-\infty$.

The following lemma allows us to relate the dimension of the double points of $\eta'$ to the intersection dimension of GFF flow lines given in Theorem~\ref{thm::two_flowline_dimension}.  This immediately leads to the lower bound in Theorem~\ref{thm::double_point_dimension} for $\kappa' \in (4,8)$.  We will explain a bit later how to extract from this the upper bound as well.

\begin{lemma}
\label{lem::strip_paths}
Let $P_\cap(z) = \eta_z^1([0,\tau_z^1]) \cap \eta_z^2([0,\tau_z^2])$.  We have that
\[ \dimH (P_\cap(z)) = 2-\frac{(12-\kappa')(4+\kappa')}{8\kappa'} \quad\text{almost surely}.\]
That is, $\dimH (P_\cap(z))$ is almost surely equal to the Hausdorff dimension of the intersection of two GFF flow lines with an angle gap of $\theta_{\rm double}$ (recall \eqref{eqn::angle_double}) as given in Theorem~\ref{thm::two_flowline_dimension}.
\end{lemma}
\begin{proof}
See Figure~\ref{fig::pocket2} for an illustration of the argument.  We shall assume throughout for simplicity that $z \in \partial_L \T$.  A similar argument gives the same result for $z \in \partial_U \T$.  Suppose that $\wt{h}$ is a GFF on $\h$ with the boundary data as indicated in the left side of Figure~\ref{fig::pocket2}.  Let $\eta_0^1$ be the flow line of $\wt{h}$ from $0$ with angle $\tfrac{\pi}{2}$.  Given $\eta_0^1$, let $\eta_0^2$ be the flow line of $\wt{h}$ with angle $\tfrac{\pi}{2}$ from $\infty$ in the component $L$ of $\h \setminus \eta_0^1$ which is to the left of $\eta_0^1$.  Note that $\eta_0^1$ is an $\SLE_\kappa(\tfrac{\kappa}{2}-2;-\tfrac{\kappa}{2})$ process in $\h$ from $0$ to $\infty$.  Moreover, the conditional law of $\eta_0^2$ given $\eta_0^1$ is an $\SLE_\kappa(\kappa-4;-\tfrac{\kappa}{2})$ process in $L$ from $\infty$ to $0$; see \cite[Lemma~3.3]{IG3}.  (The $\kappa-4$ force point lies between $\eta_0^1$ and $\eta_0^2$.)  By the main result of \cite{IG2}, the time-reversal $\wt{\eta}_0^2$ of $\eta_0^2$ is an $\SLE_\kappa(-\tfrac{\kappa}{2};\kappa-4)$ process in $L$ from $0$ to $\infty$.  As explained in Figure~\ref{fig::pocket3}, it consequently follows from Theorem~\ref{thm::two_flowline_dimension} that
\begin{equation}
\label{eqn::two_path_trick}
 \dimH(\eta_0^1 \cap \eta_0^2) = 2-\frac{(12-\kappa')(4+\kappa')}{8\kappa'} \quad\text{almost surely}
\end{equation}
since this is the almost sure dimension of $\wh{\eta}_0^1 \cap \wh{\eta}_0^2$ (using the notation of Figure \ref{fig::pocket3}).  Thus to complete the proof, we just have to argue that $\dimH(P_\cap(z))$ is also given by this value.

Let $U_1$ be the component of $\h \setminus (\eta_0^1 \cup \eta_0^2)$ which contains $1$ on its boundary.  Let $\varphi \colon U_1 \to \T$ be the conformal transformation which takes $1$ to $z$ and the leftmost (resp.\ rightmost) point of $\partial U_1 \cap \R$ to $-\infty$ (resp.\ $+\infty$).
Let $(\eta_1^1,\eta_1^2)$ be a pair of paths constructed in exactly the same manner as $(\eta_0^1,\eta_0^2)$ except starting from $1$ rather than $0$.  We consequently have that the image under $\varphi$ of the region between $\eta_1^1$ and $\eta_1^2$ is equal in distribution to $P(z)$ as described before the lemma statement.  Since $\dimH(\eta_1^1 \cap \eta_1^2)$ is also almost surely given by the value in \eqref{eqn::two_path_trick}, the desired result follows.
\end{proof}

\begin{figure}[ht!]
\begin{center}
\includegraphics[scale=0.85]{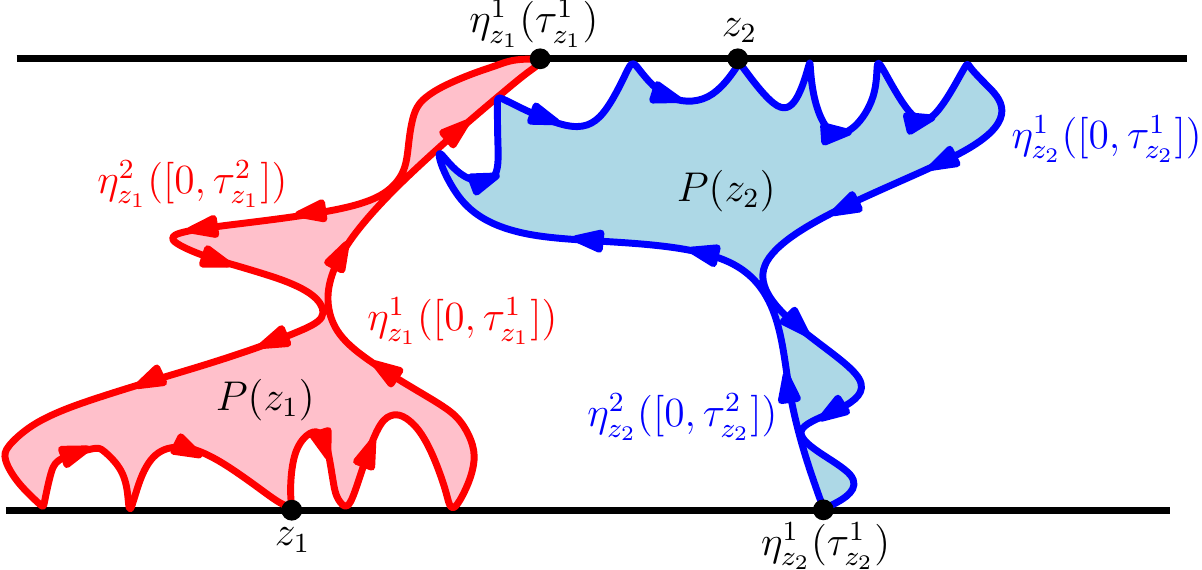}
\end{center}
\caption{\label{fig::two_pockets}
Suppose that we have the same setup as described in Figure~\ref{fig::pocket}.  Shown is $P(z_1)$ where $z_1 \in \partial \T$ is fixed.  The conditional law of $\eta'$ given $P(z_1)$ is independently that of an $\SLE_{\kappa'}(\tfrac{\kappa'}{2}-4;\tfrac{\kappa'}{2}-4)$ in each of the components $C$ of $\T \setminus P(z_1)$ starting from the first point of $\ol{C}$ hit by $\eta'$ and exiting at the last.  Fix $z_2$ on the boundary of a component $C$ of $\T \setminus P(z_1)$.  Then we can consequently form the set $P(z_2)$ which describes the interface between the set of points that $\eta'$, viewed as a path in $C$, hits before and after hitting $z_2$.  The intersection of the left and right boundaries of $P(z_2)$ consists of double points of $\eta'$.  Moreover, the conditional law of $\eta'$ given both $P(z_1)$ and $P(z_2)$ is independently that of an $\SLE_{\kappa'}(\tfrac{\kappa'}{2}-4;\tfrac{\kappa'}{2}-4)$ in each of the components of $\T \setminus (P(z_1) \cup P(z_2))$.  Consequently, we can iterate this procedure to eventually explore the entire trajectory of $\eta'$ (and, as we will explain in Lemma~\ref{lem::double_points_covered}, the double points of $\eta'$).  We will use this in Lemma~\ref{lem::double_points_covered} to reduce the double point dimension to computing the intersection dimension of GFF flow lines with an angle gap of $\theta_{\rm double}$ (recall \eqref{eqn::angle_double}).}
\end{figure}

Let $\CD$ be the set of double points of $\eta'$.  To complete the proof of Theorem~\ref{thm::double_point_dimension}, we will show that every double point of $\eta'$ is in fact in some $P_\cap(z)$. To this end, we explore the trajectory of $\eta'$ as follows.  Let $(d_j)_{j \in \N}$ be a sequence that traverses $\N\times\N$ in diagonal order, i.e. $d_1=(1,1)$, $d_2=(1,2)$, $d_3=(2,1)$, etc.  Let $(z_{1,k})_{k \in \N}$ be a countable dense subset of $\partial\T$, and set $z_1=z_{d_1}$.  Let $P(z_1)$ be the set which separates $\T$ into the set of points visited by $\eta'$ before and after hitting $z_1$, as in Figure~\ref{fig::pocket}.  We then let $(z_{2,k})_{k\in\N}$ be a countable dense subset of $\partial (\T\setminus P(z_1))$ and set $z_2=z_{d_2}$.  Recall that the conditional law of $\eta'$ given $P(z_1)$ is independently that of an $\SLE_{\kappa'}(\tfrac{\kappa'}{2}-4;\tfrac{\kappa'}{2}-4)$ process in each of the components of $\T \setminus P(z_1)$ --- this is the same as the law of $\eta'$ itself, up to conformal transformation.  Consequently, once we have fixed $P(z_1)$, we define $P(z_2)$ analogously in terms of the segment of $\eta'$ which traverses the component of $\T \setminus P(z_1)$ with $z_2$ on its boundary (see Figure \ref{fig::two_pockets}).
Generally, given $P(z_1),\ldots,P(z_n)$, we let $(z_{n+1,k})_{k\in\N}$ be a countable dense subset of $\partial (\T\setminus \cup_{j=1}^n P(z_j))$ and set $z_{n+1}=z_{d_{n+1}}$.  The conditional law of $\eta'$ given $P(z_1),\ldots,P(z_n)$ is independently that of an $\SLE_{\kappa'}(\tfrac{\kappa'}{2}-4;\tfrac{\kappa'}{2}-4)$ in each of the components of $\T\setminus \cup_{j=1}^n P(z_j)$.  Thus given $P(z_1),\ldots,P(z_n)$, we define $P(z_{n+1})$  analogously in terms of the segment of $\eta'$ which traverses the component which has $z_{n+1}$ on its boundary.  For each $n \in \N$, $\eta'$ almost surely hits $z_n$ only once at time $t(z_n)$.  Moreover, from the construction, we have that $(t(z_n))_{n\in\N}$ is a dense set of times in $[0,\infty)$ (see \cite[Section~3.3]{IG3}).

\begin{lemma}
\label{lem::double_points_covered}
Almost surely, $\CD \subseteq \cup_{j=1}^\infty P_\cap(z_j)$.
\end{lemma}
\begin{proof}
For each $\omega \in \CD$, let $t^f(\omega)$ and $t^\ell(\omega)$ be the first and last time that $\eta'$ hits $\omega$.  For each $\delta>0$ we let $\CD_{\delta} = \{\omega \in \CD : t^\ell(\omega) - t^f(\omega) \geq \delta\}$.  Clearly, the sets $\CD_\delta$ increase as $\delta > 0$ decreases and $\CD = \cup_{\delta > 0} \CD_\delta$.  Therefore it suffices to show that $\CD_\delta  \subseteq \cup_{n=1}^\infty P_\cap(z_n)$ for each $\delta > 0$.  Fix $\omega\in\CD_{\delta}$ and consider $P(z_1)$. If $t^f(\omega)<t(z_1)<t^\ell(\omega)$, then $\omega\in P_\cap(z_1)$ and we stop the exploration. If $t(z_1)>t^\ell(\omega)$ or $t(z_1)<t^f(\omega)$, then $\omega$ is a double point of $\eta'|_{[0,t(z_1)]}$ or a double point of $\eta'|_{[t(z_1),\infty)}$, respectively. Consider $P(z_2)$. If $t^f(\omega)<t(z_2)<t^\ell(\omega)$, then $\omega\in P_\cap(z_2)$ and we stop the exploration. If $t(z_2)<t^f(\omega)$ or $t(z_2)>t^\ell(\omega)$, we continue the exploration.  We continue to iterate this until the first $k$ that $\omega \in P(z_k)$.  To see that the exploration terminates after a finite number of steps, recall that $(t(z_n))_{n \in \N}$ is a dense set of times in $[0,\infty)$.  In particular, letting
\[ k = \min\left\{j \geq 1 : t^f(\omega)<t(z_j)<t^\ell(\omega) \right\}\]
we have that $\omega \in P_\cap(z_k)$.
\end{proof}

We now have all of the ingredients to complete the proof of Theorem~\ref{thm::double_point_dimension} for $\kappa' \in (4,8)$.

\begin{proof}[Proof of Theorem~\ref{thm::double_point_dimension} for $\kappa' \in (4,8)$]
Lemma~\ref{lem::strip_paths} and Lemma~\ref{lem::double_points_covered} together imply that $\dim(\CD) = 2-(12-\kappa')(4+\kappa')/(8\kappa')$ almost surely, as desired.
\end{proof}

\medbreak

We finish by proving Theorem~\ref{thm::double_point_dimension} for $\kappa' \geq 8$.

\begin{proof}[Proof of Theorem~\ref{thm::double_point_dimension} for $\kappa' \geq 8$]
Fix $\kappa'\ge 8$ and let $\kappa=\tfrac{16}{\kappa'} \in (0,2]$.  Let $\eta'$ be an $\SLE_{\kappa'}$ process in $\h$ from $0$ to $\infty$ and let $\CD$ be the set of double points of $\eta'$.  Then $\eta'$ is space-filling \cite{\RohdeSchramm}.  For each point $z \in \h$, let $t(z)$ be the first time that $\eta'$ hits $z$ and let $\gamma(z)$ be the outer boundary of $\eta([0,t(z)])$.  It follows from \cite[Theorem~1.1 and Theorem~1.13]{IG4} and \cite{beffara2008} that the dimension of $\gamma(z)$ is equal to $1+\frac{\kappa}{8}=1+\frac{2}{\kappa'}$. Given $\gamma(z)$, $\eta'([t(z),\infty))$ is an $\SLE_{\kappa'}$ process in the remaining domain, and thus almost surely hits every point on $\gamma(z)$ except the point $z$. This implies that every point on $\gamma(z)$ except for $z$ is contained in $\CD$. This gives the lower bound for $\dimH(\CD)$.

Let $(z_k)_{k\in\N}$ be a countable dense set in $\h$.  For the upper bound, we will show that every element of $\CD$ is in fact on $\gamma(z_k)$ for some $k$.  Note that $(t(z_k))_{k \in \N}$ is a dense set of times in $[0,\infty)$ because $\eta'$ is continuous. For each $\omega \in \CD$, let $t^f(\omega)$ and $t^\ell(\omega)$ be the first and last times, respectively, that $\eta'$ hits $\omega$.  For each $\delta>0$, $\CD_{\delta} = \{ \omega \in \CD : t^\ell(\omega) - t^f(\omega) \geq \delta\}$. Then $\CD=\cup_{\delta>0}\CD_{\delta}$.  Since the sets $\CD_\delta$ are increasing as $\delta > 0$ decreases, it suffices to show that $\CD_\delta \subseteq \cup_k \gamma(z_k)$ for each $\delta > 0$.  Fix $\delta > 0$ and $\omega \in \CD_\delta$.  Since $(t(z_k))_{k \in \N}$ is dense, we have that
\[  k = \min\{ j \geq 1 : t^\ell(\omega) > t(z_j) > t^f(\omega)\} < \infty.\]
Clearly, $\omega\in\gamma(z_k)$.  This completes the proof for $\kappa'\ge 8$.
\end{proof}

\begin{remark}
\label{rem::triple_points}
We note that $\SLE_\kappa'$ for $\kappa' \in (4,8)$ does not have triple points and, when $\kappa' \geq 8$, the set of triple points is countable.  Indeed, to see this we note that if $z$ is a triple point of an $\SLE_\kappa'$ process $\eta'$ then there exists rational times $t_1 < t_2$ such that $z$ is a single-point of and contained in the outer boundary of $\eta'|_{[0,t_1]}$ and a double point of and contained in the outer boundary of $\eta'|_{[0,t_2]}$.  For each pair $t_1 < t_2$ there are precisely two points which satisfy these properties.  The claim follows for $\kappa' \in (4,8)$ since $\SLE_\kappa'$ for $\kappa' \in (4,8)$ almost surely does not hit any given boundary point distinct from its starting point.  The claim likewise follows for $\kappa' \geq 8$ because this describes a surjection from $\Q_+ \times \Q_+$, $\Q_+ = (0,\infty) \cap \Q$, to the set of triple points.
\end{remark}

\section*{Acknowledgments}
H.W.'s work is supported by the Fondation CFM JP Aguilar pour la recherche.
We thank Scott Sheffield and Wendelin Werner for helpful discussions.
We thank the support and hospitality of Microsoft Research as well as
FIM at ETH Z\"urich where part of this research was conducted.

\bibliographystyle{hmralpha}
\bibliography{gff_sle_dimension}

\newcommand{\etalchar}[1]{$^{#1}$}
\begin{thebibliography}{CDCH{\etalchar{+}}13}

\bibitem[AS08]{sheffield2008}
Tom Alberts and Scott Sheffield.
\newblock Hausdorff dimension of the {SLE} curve intersected with the real
  line.
\newblock {\em Electron. J. Probab.}, 13:no. 40, 1166--1188, 2008. \MR{2430703
  (2009e:60025)}

\bibitem[Bef04]{beffara2004}
Vincent Beffara.
\newblock Hausdorff dimensions for {$\rm SLE\sb 6$}.
\newblock {\em Ann. Probab.}, 32(3B):2606--2629, 2004. \MR{2078552
  (2005k:60295)}

\bibitem[Bef08]{beffara2008}
Vincent Beffara.
\newblock The dimension of the {SLE} curves.
\newblock {\em Ann. Probab.}, 36(4):1421--1452, 2008. \MR{2435854
  (2009e:60026)}

\bibitem[CDCH{\etalchar{+}}13]{CDCHKS}
D.~Chelkak, H.~Duminil-Copin, C.~Hongler, A.~Kemppainen, and S.~Smirnov.
\newblock Convergence of {I}sing interfaces to {S}chramm's {SLE}s.
\newblock 2013.
\newblock Manuscript.

\bibitem[CN07]{MR2322705}
Federico Camia and Charles~M. Newman.
\newblock Critical percolation exploration path and {${\rm SLE}\sb 6$}: a proof
  of convergence.
\newblock {\em Probab.\ Theory Related Fields}, 139(3-4):473--519, 2007.
\newblock \arxiv{math/0604487}. \MR{2322705 (2008k:82040)}

\bibitem[CS09]{chelkak-smirnov}
Dmitry Chelkak and Stanislav Smirnov.
\newblock Universality in the {2D} {Ising} model and conformal invariance of
  fermionic observables.
\newblock {\em Inventiones Mathematicae}, 2009.
\newblock \arXiv{0910.2045}.

\bibitem[DS89]{PhysRevLett.63.2536}
Bertrand Duplantier and Hubert Saleur.
\newblock Exact fractal dimension of {2D} {Ising} clusters.
\newblock {\em Phys.\ Rev.\ Lett.}, 63:2536, 1989.

\bibitem[Dub09a]{MR2571956}
Julien Dub{\'e}dat.
\newblock Duality of {S}chramm-{L}oewner evolutions.
\newblock {\em Ann. Sci. \'Ec. Norm. Sup\'er. (4)}, 42(5):697--724, 2009.
  \MR{2571956 (2011g:60151)}

\bibitem[Dub09b]{dubedat2009}
Julien Dub{\'e}dat.
\newblock S{LE} and the free field: partition functions and couplings.
\newblock {\em J. Amer. Math. Soc.}, 22(4):995--1054, 2009. \MR{2525778
  (2011d:60242)}

\bibitem[Dup04]{DB2004}
Bertrand Duplantier.
\newblock Conformal fractal geometry \& boundary quantum gravity.
\newblock In {\em Fractal geometry and applications: a jubilee of {B}eno\^\i t
  {M}andelbrot, {P}art 2}, volume~72 of {\em Proc. Sympos. Pure Math.}, pages
  365--482. Amer. Math. Soc., Providence, RI, 2004. \MR{2112128 (2005m:82057)}

\bibitem[HBB10]{HagendorfBauerBernard10}
Christian Hagendorf, Denis Bernard, and Michel Bauer.
\newblock The {G}aussian free field and {${\rm SLE}_4$} on doubly connected
  domains.
\newblock {\em J. Stat. Phys.}, 140(1):1--26, 2010.

\bibitem[HK11]{HONG_KYT_ISING}
C.~{Hongler} and K.~{Kyt{\"o}l{\"a}}.
\newblock {Ising Interfaces and Free Boundary Conditions}.
\newblock {\em ArXiv e-prints}, August 2011, 1108.0643.

\bibitem[IK10]{IzyurovKytola10}
K.~{Izyurov} and K.~{Kyt{\"o}l{\"a}}.
\newblock {Hadamard's formula and couplings of SLEs with free field}.
\newblock {\em ArXiv e-prints}, June 2010, 1006.1853.

\bibitem[Law]{LawlerBrownianLoopMeasure}
Gregory Lawler.
\newblock A note on the {B}rownian loop measure.

\bibitem[Law96]{MR1386294}
Gregory~F. Lawler.
\newblock Hausdorff dimension of cut points for {B}rownian motion.
\newblock {\em Electron. J. Probab.}, 1:no.\ 2, approx.\ 20 pp.\ (electronic),
  1996. \MR{1386294 (97g:60111)}

\bibitem[Law05]{lawler2005}
Gregory~F. Lawler.
\newblock {\em Conformally invariant processes in the plane}, volume 114 of
  {\em Mathematical Surveys and Monographs}.
\newblock American Mathematical Society, Providence, RI, 2005. \MR{2129588
  (2006i:60003)}

\bibitem[LSW01a]{MR1879850}
Gregory~F. Lawler, Oded Schramm, and Wendelin Werner.
\newblock Values of {B}rownian intersection exponents. {I}. {H}alf-plane
  exponents.
\newblock {\em Acta Math.}, 187(2):237--273, 2001. \MR{1879850 (2002m:60159a)}

\bibitem[LSW01b]{MR1879851}
Gregory~F. Lawler, Oded Schramm, and Wendelin Werner.
\newblock Values of {B}rownian intersection exponents. {II}. {P}lane exponents.
\newblock {\em Acta Math.}, 187(2):275--308, 2001. \MR{1879851 (2002m:60159b)}

\bibitem[LSW02]{MR1899232}
Gregory~F. Lawler, Oded Schramm, and Wendelin Werner.
\newblock Values of {B}rownian intersection exponents. {III}. {T}wo-sided
  exponents.
\newblock {\em Ann. Inst. H. Poincar\'e Probab. Statist.}, 38(1):109--123,
  2002. \MR{1899232 (2003d:60163)}

\bibitem[LSW03]{LawlerSchrammWernerConformalRestrictionChordal}
Gregory Lawler, Oded Schramm, and Wendelin Werner.
\newblock Conformal restriction: the chordal case.
\newblock {\em J. Amer. Math. Soc.}, 16(4):917--955 (electronic), 2003.
  \MR{1992830 (2004g:60130)}

\bibitem[LSW04]{MR2044671}
Gregory~F. Lawler, Oded Schramm, and Wendelin Werner.
\newblock Conformal invariance of planar loop-erased random walks and uniform
  spanning trees.
\newblock {\em Ann.\ Probab.}, 32(1B):939--995, 2004.
\newblock \arxiv{math/0112234}. \MR{2044671 (2005f:82043)}

\bibitem[LW04]{LawlerWernerBrownianLoopsoup}
Gregory~F. Lawler and Wendelin Werner.
\newblock The {B}rownian loop soup.
\newblock {\em Probab. Theory Related Fields}, 128(4):565--588, 2004.
  \MR{2045953 (2005f:60176)}

\bibitem[Mil10]{arXiv:1010.1356}
Jason Miller.
\newblock Universality for {SLE}$_4$.
\newblock 2010.
\newblock Preprint, \arXiv{1010.1356}.

\bibitem[MP10]{BM}
Peter M{\"o}rters and Yuval Peres.
\newblock {\em Brownian motion}.
\newblock Cambridge Series in Statistical and Probabilistic Mathematics.
  Cambridge University Press, Cambridge, 2010.
\newblock With an appendix by Oded Schramm and Wendelin Werner. \MR{2604525
  (2011i:60152)}

\bibitem[MS10]{MakarovSmirnov09}
Nikolai Makarov and Stanislav Smirnov.
\newblock Off-critical lattice models and massive {SLE}s.
\newblock pages 362--371, 2010.

\bibitem[MS12a]{IG1}
J.~{Miller} and S.~{Sheffield}.
\newblock {Imaginary Geometry I: Interacting SLEs}.
\newblock {\em ArXiv e-prints}, January 2012, 1201.1496.

\bibitem[MS12b]{IG2}
J.~{Miller} and S.~{Sheffield}.
\newblock {Imaginary geometry II: reversibility of
  SLE\_$\backslash$kappa($\backslash$rho\_1;$\backslash$rho\_2) for
  $\backslash$kappa $\backslash$in (0,4)}.
\newblock {\em ArXiv e-prints}, January 2012, 1201.1497.

\bibitem[MS12c]{IG3}
J.~{Miller} and S.~{Sheffield}.
\newblock {Imaginary geometry III: reversibility of
  SLE\_$\backslash$kappa$\backslash$ for $\backslash$kappa $\backslash$in
  (4,8)}.
\newblock {\em ArXiv e-prints}, January 2012, 1201.1498.

\bibitem[MS13]{IG4}
J.~{Miller} and S.~{Sheffield}.
\newblock {Imaginary geometry IV: interior rays, whole-plane reversibility, and
  space-filling trees}.
\newblock {\em ArXiv e-prints}, February 2013, 1302.4738.

\bibitem[Pom92]{MR1217706}
Ch. Pommerenke.
\newblock {\em Boundary behaviour of conformal maps}, volume 299 of {\em
  Grundlehren der Mathematischen Wissenschaften [Fundamental Principles of
  Mathematical Sciences]}.
\newblock Springer-Verlag, Berlin, 1992. \MR{1217706 (95b:30008)}

\bibitem[RS05]{MR2153402}
Steffen Rohde and Oded Schramm.
\newblock Basic properties of {SLE}.
\newblock {\em Ann.\ of Math.\ (2)}, 161(2):883--924, 2005.
\newblock \arxiv{math/0106036}. \MR{2153402 (2006f:60093)}

\bibitem[Sch00]{MR1776084}
Oded Schramm.
\newblock Scaling limits of loop-erased random walks and uniform spanning
  trees.
\newblock {\em Israel J. Math.}, 118:221--288, 2000.
\newblock \arxiv{math/9904022}. \MR{1776084 (2001m:60227)}

\bibitem[She]{She_SLE_lectures}
S.~Sheffield.
\newblock Local sets of the gaussian free field: slides and audio.
  www.fields.utoronto.ca/0506/percolationsle/sheffield1,
  www.fields.utoronto.ca/audio/0506/percolationsle/sheffield2,
  www.fields.utoronto.ca/audio/0506/percolationsle/sheffield3.

\bibitem[She07]{sheffield2007}
Scott Sheffield.
\newblock Gaussian free fields for mathematicians.
\newblock {\em Probab. Theory Related Fields}, 139(3-4):521--541, 2007.
  \MR{2322706 (2008d:60120)}

\bibitem[{She}10]{2010arXiv1012.4797S}
S.~{Sheffield}.
\newblock {Conformal weldings of random surfaces: SLE and the quantum gravity
  zipper}.
\newblock {\em ArXiv e-prints}, December 2010, 1012.4797.

\bibitem[Smi05]{MR2227824}
Stanislav Smirnov.
\newblock Critical percolation and conformal invariance.
\newblock In {\em X{IV}th {I}nternational {C}ongress on {M}athematical
  {P}hysics}, pages 99--112. World Sci.\ Publ., Hackensack, NJ, 2005.
  \MR{2227824 (2007h:82030)}

\bibitem[Smi10]{MR2680496}
Stanislav Smirnov.
\newblock Conformal invariance in random cluster models. {I}. {H}olomorphic
  fermions in the {I}sing model.
\newblock {\em Ann.\ of Math.\ (2)}, 172(2):1435--1467, 2010.
\newblock \arXiv{0708.0039}. \MR{2680496 (2011m:60302)}

\bibitem[SS09]{MR2486487}
Oded Schramm and Scott Sheffield.
\newblock Contour lines of the two-dimensional discrete {G}aussian free field.
\newblock {\em Acta Math.}, 202(1):21--137, 2009.
\newblock \arxiv{math/0605337}. \MR{2486487 (2010f:60238)}

\bibitem[SS10]{SchrammShe10}
O.~{Schramm} and S.~{Sheffield}.
\newblock {A contour line of the continuum Gaussian free field}.
\newblock {\em ArXiv e-prints}, August 2010, 1008.2447.

\bibitem[SW05]{MR2188260}
Oded Schramm and David~B. Wilson.
\newblock S{LE} coordinate changes.
\newblock {\em New York J. Math.}, 11:659--669 (electronic), 2005. \MR{2188260
  (2007e:82019)}

\bibitem[VL09]{lawler2009}
Fredrik~Johansson Viklund and Gregory~F. Lawler.
\newblock Almost sure multifractal spectrum for the tip of an sle curve, 2009,
  arXiv:0911.3983.

\bibitem[Wer04a]{MR2060031}
Wendelin Werner.
\newblock Girsanov's transformation for {${\rm SLE}(\kappa,\rho)$} processes,
  intersection exponents and hiding exponents.
\newblock {\em Ann. Fac. Sci. Toulouse Math. (6)}, 13(1):121--147, 2004.
  \MR{2060031 (2005b:60262)}

\bibitem[Wer04b]{W03}
Wendelin Werner.
\newblock Random planar curves and {S}chramm-{L}oewner evolutions.
\newblock In {\em Lectures on probability theory and statistics}, volume 1840
  of {\em Lecture Notes in Math.}, pages 107--195. Springer, Berlin, 2004.
\newblock \arxiv{math/0303354}. \MR{2079672 (2005m:60020)}

\bibitem[WW13]{WERNER_WU_CLE}
Wendelin Werner and Hao Wu.
\newblock From cle($\backslash$kappa) to
  sle($\backslash$kappa,$\backslash$rho)'s.
\newblock {\em Electron. J. Probab.}, 18:no. 36, 1--20, 2013.

\bibitem[Zha08]{MR2439609}
Dapeng Zhan.
\newblock Duality of chordal {SLE}.
\newblock {\em Invent. Math.}, 174(2):309--353, 2008. \MR{2439609
  (2010f:60239)}

\bibitem[Zha10]{MR2682265}
Dapeng Zhan.
\newblock Duality of chordal {SLE}, {II}.
\newblock {\em Ann. Inst. Henri Poincar\'e Probab. Stat.}, 46(3):740--759,
  2010. \MR{2682265 (2011i:60155)}

\end{thebibliography}
\end{document}